\documentclass[11pt]{article}

  \usepackage{latexsym}
  \usepackage{comment}
  \usepackage{hyperref}
  \usepackage{amsmath,amsthm,amsfonts,amssymb,graphicx,epsfig,latexsym,color,amssymb}
  \usepackage{amscd,latexsym,esint,upref,stmaryrd,enumerate,verbatim,yfonts,bm,mathtools,comment}

\usepackage{indentfirst} 
  \usepackage{epsf,enumerate}
  \usepackage{tikz}
  \usepackage{mathtools}
  \usepackage{outlines}
  \usepackage{float}
  \usepackage [english]{babel}
\usepackage [autostyle, english = american]{csquotes}
\MakeOuterQuote{"}
\usepackage[utf8]{inputenc}
  
\usepackage{enumitem} 

\DeclareMathAlphabet{\mathbbold}{U}{bbold}{m}{n}

\newcommand{\bi}{\bibitem}

  \newtheorem{innercustomgeneric}{\customgenericname}
\providecommand{\customgenericname}{}
\newcommand{\newcustomtheorem}[2]{%
  \newenvironment{#1}[1]
  {%
   \renewcommand\customgenericname{#2}%
   \renewcommand\theinnercustomgeneric{##1}%
   \innercustomgeneric
  }
  {\endinnercustomgeneric}
}

\newcustomtheorem{customthm}{Theorem}
\newcustomtheorem{customlemma}{Lemma}

  \def\R{{\mathbb R}}

  \def\Z{{\mathbb Z}}
  \def\N{{\mathbb N}}
  \def\H{{\mathbb H}}

  \newcommand{\<}{\langle}
  \renewcommand{\>}{\rangle}

\newtheorem{theorem}{Theorem}
\numberwithin{theorem}{section}
\newtheorem{proposition}[theorem]{Proposition}
\newtheorem{lemma}[theorem]{Lemma}
\newtheorem{corollary}[theorem]{Corollary}

\theoremstyle{definition}
\newtheorem{definition}[theorem]{Definition}
\newtheorem{remark}[theorem]{Remark}
\newtheorem{example}[theorem]{Example}

\theoremstyle{remark}

    \title{Combinations of parabolically geometrically finite groups and their geometry}
\author{Brian Udall }
\date{\vspace{-5ex}}

\begin{document}
\maketitle

\begin{abstract}
    In this paper, we study the class of parabolically geometrically finite (PGF) subgroups of mapping class groups, introduced by Dowdall--Durham--Leininger--Sisto. We prove a combination theorem for graphs of PGF groups (and other generalizations) by utilizing subsurface projection to obtain control on the geometry of fundamental groups of graphs of PGF groups, generalizing and strengthening methods of Leininger--Reid. From this result, we construct new examples of PGF groups and provide methods for how to apply the combination theorem in practice. We also show that PGF groups are undistorted in their corresponding mapping class groups.
\end{abstract}


\section{Introduction}
Much of the study of subgroups of mapping class groups of surfaces and their actions on Teichmüller space has been motivated by various partial analogies with the actions of subgroups of $\text{Isom}(\H^n)$ on hyperbolic space $\H^n$. Farb--Mosher \cite{FM02} introduced a robust notion of convex cocompactness for subgroups of mapping class groups in terms of Teichmüller space analogous to the definition of convex cocompactness in $\text{Isom}(\H^n)$. Kent--Leininger \cite{KL08} and Hamenstädt \cite{H05} both gave an equivalent condition in terms of the action of the group on the curve graph of the surface. Namely, if $\Sigma$ is a hyperbolic surface of finite area, a finitely generated subgroup $G$ of the mapping class group MCG$(\Sigma)$ is \textit{convex cocompact} if and only if there is an equivariant map from $G$ to the curve graph of $\Sigma$, $C(\Sigma)$, that is a quasi-isometric embedding.
\par 
In Definition 1.10 of \cite{DDLS21}, Dowdall--Durham--Leininger--Sisto define a notion of geometric finiteness in mapping class groups motivated by the characterization of convex cocompactness in terms of $C(\Sigma)$ given above. A finitely generated subgroup $G$ of MCG$(\Sigma)$ is \textit{parabolically geometrically finite} (PGF) if it is relatively hyperbolic, relative to a finite collection of twist groups (groups containing finite index abelian subgroup generated by multitwists), and so the coned off graph of $G$ equivariantly quasi-isometrically embeds into $C(\Sigma)$ (see Definition \ref{PGFGrp} for a precise definition). 
\par 
There are few known classes of examples of PGF groups. Results of Tang \cite{T21} imply that finitely generated Veech groups are PGF, and Loa \cite{L21} showed that free products of multitwist groups on sufficiently far apart multicurves are PGF as well. 
\par 
The first result of this paper provides a method for building new PGF groups from old ones. First some brief definitions. Fix a closed surface $\Sigma$. A \textit{normalized PGF graph of groups} $\mathcal{G}$ is a graph of groups whose vertex groups are PGF groups and twist groups in MCG$(\Sigma)$, arranged in a particular manner (see Definition \ref{PGFGraph} for more details). A \textit{compatible homomorphism} $\phi$ is a homomorphism from the fundamental group of a normalized PGF graph of groups to MCG$(\Sigma)$ which on vertex groups restricts to inclusion up to conjugation in MCG$(\Sigma)$. Given a twist group $H$ with finite index subgroup $H'$ generated by multitwists, we say that the multicurve $A$ is \textit{associated} to $H$ if $A$ is the multicurve containing all the curves that elements of $H'$ twist on.
\par 
 We say the pair $(\mathcal{G}, \phi)$ satisfies the $L$-\textit{local large projections property} for some $L\geq 0$ if the following holds. Take any two vertices $v_1$ and $v_2$ in the Bass--Serre tree of $\mathcal{G}$ with PGF stabilizers $G_1$ and $G_2$, respectively, so that there are no other vertices on $[v_1,v_2]$ which have a PGF stabilizer. Let $v_{12}$ be the vertex on $[v_1,v_2]$ neighboring $v_1$ with stabilizer $H_{12}$, which is a twist group. Let $B_{12}$ denote the multicurve associated to $\phi(H_{12})$. Then for all multicurves $B_1, B_2$ associated to $\phi(G_1), \phi(G_2)$, respectively, that are not equal to $B_{12}$, we have for any component $S$ of $\Sigma\setminus B_{12}$ that
 $$d_S(B_1,B_2)\geq L.$$
 We now state the main theorem (see Theorem \ref{ComboQIEmb} for a more complete statement).

\begin{customthm}{1}
Suppose $G$ is the fundamental group of a normalized PGF graph of groups of a closed surface $\Sigma$, and $\phi:G \to \text{MCG}(\Sigma)$ is a compatible homomorphism that satisfies the $L$-local large projections property for some sufficiently large $L$. Then $\phi$ is injective, its image is a PGF group, and all infinite order elements not contained in a twist group are pseudo-Anosov.
\end{customthm}

This theorem is inspired by and generalizes the main results of Leininger--Reid in \cite{LR06}, and just like the results there Theorem $1$ is motivated by the classical Klein--Maskit combination theorems. Along the way we develop a more general framework to allow for other types of combination theorems using PGF groups. See Theorems \ref{ConvexCocompactCombo}, \ref{ConvexCoCptTwists}, and \ref{LoaAnalog} for examples of this. 
\par 
Here is a non-exhaustive list of the PGF groups that we produce in Section \ref{SectionExamplesApplications}.
\begin{itemize}
    \item Fundamental groups of books of $I$-bundles over a surface of a fixed genus (Corollary \ref{BookOfIBundleCombo}). Compare this to the results of \cite{LR06}.
    \item All isomorphism classes of RAAG's that could potentially be a PGF group (Lemma \ref{GenericPGFRAAGS} and the discussion after it, along with Example \ref{GeneralPGFRAAGS}).
    \item Convex cocompact groups generated by arbitrarily many pseudo-Anosov's with small dilatation (Theorem \ref{SuperIndependentpA}).
    \item Free products of convex cocompact groups and arbitrarily many surface groups (with restrictions on genus) (Example \ref{PGFFreeProductSurfaceGrp})
    \item PGF free groups with Dehn twist generators on curves which are at most $4$ apart in $C(\Sigma)$ (Corollary \ref{ComboMultitwistGroups} and Example \ref{PGFFreePrdocutsTwistGroupExamples}). Compare this to the results of \cite{L21}.
\end{itemize}
\par 
We also prove the following result about PGF groups, generalizing results of Tang \cite{T21} and Loa \cite{L21} (actually we prove this for a more general class of groups, see Section \ref{SectionStatementsProofsUndistorted} and Theorem \ref{Undistorted}).

\begin{customthm}{2}
If $G$ is a PGF subgroup of $\text{MCG}(\Sigma)$ with $\Sigma$ closed, then $G$ is undistorted in $\text{MCG}(\Sigma)$.
\end{customthm}

There are other potential notions of geometric finiteness in MCG$(\Sigma)$ that one may wish to study. In \cite{DDLS21} Dowdall--Durham--Leininger--Sisto propose another such notion in terms of surface group extensions (see the first section, in particular the discussion after Conjecture 1.11). They suggest that a subgroup $G$ of MCG$(\Sigma)$ should be geometrically finite if the associated $\pi_1 \Sigma$ extension group of $G$ is a hierarchically hyperbolic group. This is motivated by results of Hamenstädt \cite{H05}, generalizing a result of Farb and Mosher in \cite{FM02}, which show that a subgroup $G$ of a mapping class group of a closed hyperbolic surface $\Sigma$ is convex cocompact if and only if the associated $\pi_1(\Sigma)$ extension of $G$ is a hyperbolic group (see \cite{FM02} and \cite{H05} for the definition of these extensions). In Proposition 5.17 of \cite{MS12}, Mj--Sadar also prove a more general result for pure mapping class groups, allowing for $\Sigma$ to have punctures. Here if there are punctures then the extension group is relatively hyperbolic. 
\par 
It was conjectured in \cite{DDLS21} that PGF groups are geometrically finite in the sense given above, and they proved this result for lattice Veech groups. On the other hand, there are examples of groups satisfying this definition of geometrically finite that are not PGF. For example, MCG$(\Sigma)$ itself, along with multicurve stabilizers as proved by Russell \cite{R21}. There are other groups that are not PGF but are natural candidates for being considered geometrically finite, see the discussion after Conjecture 1.11 in \cite{DDLS21} for more about this. 
\par 
There is also another notion of geometric finiteness defined using the boundary of hierarchically hyperbolic groups given by Durham--Hagen--Sisto in \cite{DHS17}. There the authors show that Veech groups and Leininger--Reid surface groups \cite{LR06} are examples of such groups, so it seems natural to conjecture that all PGF groups are as well, or at least all PGF groups with a suitable restriction. See Section 6 of Loa for some discussion about this in the context of free products of multitwist groups \cite{L21}. 
\par 
Regardless of the definitions given so far, it may be possible that no single notion of geometric finiteness is sufficient to handle every case one might wish to study without being too general to be useful. Instead, perhaps different versions should be adapted to handle different scenarios. As of now our knowledge is limited, but the groups constructed from the techniques proving Theorem $1$ provide many new examples to help explore these notions in greater detail. 

\vskip 10pt

\noindent\textbf{Outline}: In Section \ref{SectionDefinitionsAndFacts}, we give the necessary background and definitions for the paper, focusing on subsurface projection and relative hyperbolicity. We continue this in Section \ref{SectionParabolicallyGeometricallyFinite} by giving the definition of parabolically geometrically finite groups and the graphs of groups we will be working with. We also state the full version of Theorem $1$ and prove some basic lemmas. In Section \ref{SectionWorkingTowardsComboTheorem} we prove several technical results towards the proof of Theorem $1$, with the goal of getting control on the geometry on relevant geodesics in $C(\Sigma)$. We give the proof of Theorem $1$ in Section \ref{SectionProofMainThemOtherCases}, and we also discuss other combination theorems and their proofs. Section \ref{SectionStatementsProofsUndistorted} introduces a class of groups containing PGF groups and shows that they are undistorted in MCG$(\Sigma)$, in particular proving Theorem $2$. Finally, in Section \ref{SectionExamplesApplications}, we give a variety of applications and examples of the combination theorems proven in Section \ref{SectionProofMainThemOtherCases}.

\vskip 5pt
\noindent\textbf{Acknowledgements}: The author would like to give a big thanks to his advisor Christopher J. Leininger for introducing him to the problem, and for his guidance and support throughout the entire process of writing this paper. He'd also like to thank Jacob Russell for helpful conversations involving relatively hyperbolic groups, as well as his comments on the paper. Thanks to Dan Margalit for helpful comments as well. A big thanks to the fellow graduate students who supported him, in particular Junmo Ryang whose conversations with the author helped to inspire some of the ideas in the proof of Lemma \ref{InductiveProjBound}. We also thank the anonymous referee for their many helpful suggestions, in particular for their advice to include Lemma \ref{lem:admSequenceQIlowerbound} as a way to streamline the proofs of the results in Section 5. In addition, the author acknowledges partial support from NSF grant DMS-1745670.

\section{Definitions and facts}\label{SectionDefinitionsAndFacts}

In this section we will lay out all the basic definitions and results that will be used throughout the rest of the paper. The reader already familiar with subsurface projection, as well as standard definitions of relative hyperbolicity and basic results about that, can feel free to simply reference this section when needed.
\par 
Throughout this paper we consider a surface $\Sigma$ which is closed, oriented, connected and of genus $g\geq 2$. Whenever convenient, we may assume $\Sigma$ is equipped with a hyperbolic metric. In general, given any compact surface $S$ of genus $g$ and $b$ boundary components, we denote by $\xi(S)$ the \textit{complexity} of the surface $S$, defined by
$$\xi(S)=3g-3+b.$$
\par 
\subsection{Hyperbolicity and basic notation}\label{SubsectionHyperbolicityBasicNotation}
Given a geodesic space $X$ and two points $x,y \in X$, we write $[x,y]$ for an arbitrary geodesic between $x$ and $y$. Occasionally, we will actually fix such a geodesic and continue to use the same notation to denote that fixed path.

\begin{definition}
    A geodesic space $X$ is said to be $\delta$-hyperbolic for $\delta\geq 0$ if for all points $x,y,z$, we have
    $$[x,y] \subset N_{\delta}([x,z] \cup [y,z]).$$
    Here, for any $A\subset X$,
    $$N_{\delta}(A)=\{x\in X\ | \ d(x, a)\leq \delta \text{ for some } a\in A\}.$$
\end{definition}

As all spaces of interest are geodesic, this definition will suffice. In fact, we will typically be working with graphs, and unless otherwise stated we will use the standard combinatorial path metric on graphs given by letting edges have length $1$. Such spaces are always geodesic spaces.
\par 
Let $\lambda\geq 1, \kappa\geq 0$. Given two nonnegative quantities $A$ and $B$, we use the notation $\approx_{\lambda, \kappa}$ and $\preceq_{\lambda, \kappa}$ as follows. We write $A\approx_{\lambda, \kappa} B$ if
$$\frac{1}{\lambda}B-\kappa \leq A \leq \lambda B + \kappa$$
and $A\preceq_{\lambda, \kappa} B$ if 
$$A \leq \lambda B + \kappa.$$
The notation $\succeq_{\lambda, \kappa}$ is used symmetrically. Oftentimes, we let $\lambda=\kappa$ to reduce the number of constants being used, and in this case we write $\approx_{\kappa}$ or $\preceq_{\kappa}$. We also often drop $\lambda$ and $\kappa$ from the notation of $\preceq_{\lambda, \kappa}$ and $\approx_{\lambda, \kappa}$ if we don't care about the actual constants.
\par
In particular, if $f:(X, d_X)\to (Y, d_Y)$ is a map between two metric spaces, and there is a pair $\lambda, \kappa$ so that for all $x_1, x_2\in X$, $$d_Y(f(x_1),f(x_2))\approx_{\lambda, \kappa} d_X(x_1,x_2),$$ we say that $f$ is a $(\lambda, \kappa)$-\textit{quasi-isometric embedding}, or a $(\lambda, \kappa)$-QI embedding. A $(\lambda, \kappa)$-quasi-isometric embedding $f$ is a $(\lambda,\kappa)$-\textit{quasi-isometry} if its image is \textit{quasi-dense}. That is, there is a constant $B$ so that for all $y\in Y$ there exists an $x\in X$ so that $d_Y(f(x),y)\leq B$. Similarly, $f$ is $(\lambda,\kappa)$-\textit{coarsely Lipschitz} if  $$d_Y(f(x_1),f(x_2))\preceq_{\lambda, \kappa} d_X(x_1,x_2).$$
We will often say that there is a coarse Lipschitz upper bound on a quantity $A$ in terms of another quantity $B$ if $A \preceq B$, and a coarse Lipschitz lower bound if the opposite relation holds. A map $c:[a,b]\to X$ is a $(\lambda, \kappa)$-\textit{quasi-geodesic} if it is a $(\lambda, \kappa)$-quasi-isometric embedding. Note that (parameterized) geodesics are exactly the $(1,0)$-quasi-geodesics. We will often drop the $(\lambda, \kappa)$ prefix.
\par 
We will also be making extensive use of Bass--Serre theory. The reader is referred to \cite{S02} for background, for example.
\par 
 We recall a classical fact about hyperbolic spaces and their quasi-geodesics. First, given a metric space $X$, the \textit{Hausdorff distance} between two subset $A, B\subset X$ is

$$d_{Haus}(A,B) = \inf \{r\ | \ A\subset N_r(B) \text{ and } B \subset N_r(A)\}.$$
%

\begin{proposition}[{\cite[Theorem III.H.1.7]{BH99}}]\label{Morse}
Let $X$ be a $\delta$-hyperbolic geodesic metric space. For all $K, C$, there exists a number $N=N(\delta, K, C)$ such that for any two points $x,y \in X$ and any $(K,C)$-quasi geodesic $c$ with $x$ and $y$ as endpoints, $d_{Haus}([x,y], im(c))\leq N$.
\end{proposition}

\subsection{Curves, arcs, and the Mapping Class Group}\label{SubsectionCurvesArcsMappingClassGroup}
\begin{definition}[Mapping Class Group]
Fix a compact surface $S$, and let Homeo$(S,\partial S)$ denote the group of homeomorphisms of $S$ restricting to the identity on $\partial S$. Let this group be given the compact open topology, and write Homeo$_0(S, \partial S)$ to be the connected component of the identity. We define the \textit{mapping class group} MCG$(S)$ of $S$ as
$$\text{MCG}(S):=\text{Homeo}(S,\partial S)/\text{Homeo}_0(S, \partial S).$$
\end{definition}

As paths in Homeo$(S,\partial S)$ are the same things as isotopies fixing $\partial S$, elements of MCG$(S)$ can also be considered as elements of Homeo$(S,\partial S)$ up to isotopy. It is well known that the induced quotient topology on MCG$(S)$ is discrete when $S$ is compact, potentially with finitely many marked points. In fact, for such $S$, MCG$(S)$ is finitely generated, see Section 4.3 of \cite{FM11}. We fix a finite generating set for MCG$(S)$.

\begin{definition}\label{Curves and arcs}[Curves and arcs]
Fix a compact surface $S$ that is not an annulus. 
\begin{enumerate}[label=\alph*)]
    \item A \textit{closed curve} $c$ in $S$ is the image of a continuous mapping of $S^1$ into $S$. A curve is \textit{simple} if the map is an embedding, and it is \textit{essential} if it is not homotopic to a point, or homotopic into a component of $\partial S$. 
    \item A \textit{multicurve} is a collection $A$ of distinct homotopy classes of essential simple closed curves that are pairwise disjoint. A \textit{tubular neighborhood} $N(A)$ of a multicurve $A$ is a neighborhood of $A$ so that the neighborhoods of each component of $A$ are pairwise disjoint from each other. 
    \item An \textit{arc} $a$ of $S$ is the image of an embedding $\alpha:[0,1]\to S$ such that $\alpha^{-1}(\partial S)=\{0,1\}$. An arc $a$ is \textit{essential} if $\alpha$ is not homotopic rel endpoints to a map with image in $\partial S$. We will consider homotopy classes of essential arcs where homotopies can move the end points, but they always stay in the boundary.
    \item Suppose $\alpha$ and $\beta$ denote two homotopy classes of essential simple closed curves or essential arcs (one may be a curve and one may be an arc). We denote by $i(\alpha, \beta)$ the smallest number of intersections that isotopy representatives of $\alpha$ and $\beta$ can have.
    \item Two multicurves $A$ and $B$ \textit{fill} a subsurface $R$ of $S$ if every essential curve in $R$ intersects some component of $A$ or $B$. Similarly, two subsurfaces $R_1$ and $R_2$ of $S$ \textit{fill} a subsurface $R$ of $S$ if every essential curve in $R$ intersects $R_1$ or $R_2$.
    \item  The \textit{curve graph} of $S$, denoted $C(S)$, is a graph whose vertices are isotopy classes of essential simple closed curves, with an edge between two classes $\alpha_1$ and $\alpha_2$ if there exist representatives of each class with the minimal possible intersection of distinct essential curves in $S$ (i.e. the minimal possible value of $i(\cdot, \cdot)$ for $S$). We denote the induced graph metric by $d_S$.
    \item Let $R$ be an annulus. The curve graph of $R$, also denoted $C(R)$ consists of vertices that are isotopy classes of essential simple arcs up to isotopy fixing the boundary, and edges given by disjointness in the interior of the annulus. We denote the induced graph metric by $d_{R}$. Letting $|\alpha_1 \cdot \alpha_2|$ denote the algebraic intersection number of two arcs in $C(R)$, it is straightforward to see that when $\alpha_1$ and $\alpha_2$ are distinct,
    $$d_R(\alpha_1, \alpha_2)=|\alpha_1\cdot \alpha_2|+1.$$
\end{enumerate}
\end{definition}

We remark here the well known fact that as long as $S$ has complexity at least $2$, then minimal intersection of essential curves in $S$ is given by disjointness. In the case of the torus or torus with one boundary component, minimal intersection is $1$ intersection point, and in the case of the sphere with four boundary components, the minimal intersection is $2$. In the case of pairs of pants (sphere with three boundary components) the curve graph is empty. In all these cases, it is known that $C(S)$ is connected and has infinite diameter \cite{MM98, MM00}.
 \par 
 In the future, whenever we consider a "curve", it will always be an essential simple closed curve, typically up to isotopy, unless stated otherwise.
 \par 
 If $A$ is a multicurve in a compact surface $S$, we let $S\setminus A$ denote the compact subsurface of $S$ obtained by cutting $S$ along $A$. That is, take the complement of the union of a set of pairwise disjoint open tubular neighborhoods of the components of $A$. This surface is a compact and possibly disconnected surface with two boundary components for every curve in $A$. We can rebuild $S$ by gluing up the components of $\partial(S\setminus A)$ which are homotopic to a component of $A$. Then $S\setminus A$ naturally embeds as a subsurface of $S$ so that its boundary components that are not components of $\partial S$ are each homotopic to some curve in $A$.
 \par 
 A multicurve $A$ is \textit{sparse} if every component of $S\setminus A$ has complexity at least $1$. In other words, no component is a pair of pants.
 \par 
 We say $R\subset S$ is an \textit{essential subsurface} of $S$ if $R$ is a connected surface and its boundary components are curves that are either essential in $S$ or homotopic to a component of $\partial S$, and is further not a pair of pants or an annulus with boundary homotopic to a component of $\partial S$. The choice to remove pairs of pants from this definition is just to simplify language later on, to remove phrases like "except for pairs of pants" and the like. 
 \par 
 We have the following key result about the curve complex of compact surfaces.
 
\begin{proposition}[{\cite[Theorem 1.1]{MM98}}]\label{MMHyp}
For any compact surface $S$, $C(S)$ is $\delta$-hyperbolic for some $\delta$.
\end{proposition}

In fact, $\delta$ can be taken to be independent of $S$, as proved in \cite{A13}, \cite{B14}, \cite{HPW15}, \cite{PS17}, although we will never need this.
\par 

\begin{definition}\label{Markings}[Markings]
 Fix a nonannular compact surface $S$. A \textit{marking} $\mu$ of $S$ consists of two pieces of data. First, we choose a maximal multicurve, and adjoin to it the components of $\partial S$ to form the set $b$, which is called the \textit{base} of $\mu$ (in general such a collection $b$ is called a \textit{pants decomposition}). Next, for each curve $\alpha\in b$, fix an annular neighborhood $Y_{\alpha}$, and choose a diameter $1$ set $t_{\alpha}$ in $A(Y_{\alpha})$, which we call the \textit{transversals} of the marking. 
 \end{definition}

 Markings are also considered up to isotopy. One can form a graph whose vertices are the set of (isotopy classes of) markings of a surface, but as we will not need to work with this graph we won't give a formal definition. The \textit{marking graph} of a surface $S$ is denoted by $\mathcal{M}(S)$. For more about $\mathcal{M}(S)$, see \cite{MM00}.
\par 

 \begin{definition}\label{SSProjC}[Subsurface Projection]
  Suppose $S$ is a compact surface and $R$ an essential nonannular subsurface. We define a map $\pi_R:C(S) \to \mathcal{P}(C(R))$ (the power set of $C(R)$) as follows. Fix an element $\alpha\in C(S)$, and pick a representative $a$. Perform an isotopy of $a$ so that $a$ intersects $\partial R$ minimally. There are three cases to consider.
 \begin{enumerate}
     \item If $a\cap R = \varnothing$, then $\pi_R(\alpha)=\varnothing$.
     \item If $a\subset R$, then $\pi_R(\alpha)=\alpha$.
     \item If $a$ intersects $\partial R$, then for each component $a_0$ of $a\cap R$, take the tubular neighborhood in $R$ of $a_0 \cup \partial R$ and consider the collection of boundary components of such neighborhoods. The set $\pi_R(\alpha)$ is the union of the resulting curves. 
 \end{enumerate}
 \end{definition}
 %

 \begin{definition}\label{SSProjAnnulus}[Projection to annuli on closed surfaces]
 Let $S$ be a closed surface, and assume $R$ is an essential annulus. Take the cover $\tilde{S}$ of $S$ corresponding to $\pi_1(R)$. The surface $\tilde{S}$ is an open annulus, and we can identify the natural compactification of $\tilde{S}$ with $R$. Fix $\alpha\in C(S)$, and let $a$ denote a representative. Consider all the lifts of $a$ to the closure of $\tilde{S}$. We define $\pi_R:S\to \mathcal{P}(C(R))$ as follows.
 \begin{enumerate}
     \item If there is a lift of $a$ connecting the two boundary components of the closure of $\tilde{S}$, let $\pi_R(\alpha)$ be the set of all such arcs. (Such a lift exists if $a$ intersects $R$ nontrivially).
     \item Otherwise, $\pi_R(\alpha)$ is empty.
 \end{enumerate}
 \end{definition}
We can also project multicurves and markings to subsurfaces. If $B$ is a multicurve then $\pi_R(B)$ is a the union of the projections of the components of $B$. If $\mu$ is a marking with base curves $b$, then $\pi_R(\mu)=\pi_R(b)$, unless $R$ is an annulus whose core curve is a component of $b$. In this case $\pi_R(\mu)$ is the transversal of $\mu$ at the core curve of $R$. Note then that for any marking $\mu$ and essential subsurface $R$, $\pi_R(\mu)\neq \varnothing$.
\par 
Given a compact surface $S$ with essential subsurface $R$, and two multicurves or markings $\mu_1$ and $\mu_2$ of $S$, we will want to understand the distance of the projections of $\mu_1$ and $\mu_2$ to $R$. We make the following definition for notational convenience.
\begin{definition}\label{projDistDef}
    Suppose in the case that $\mu_1$ and $\mu_2$ are multicurves that they both have nonempty projection to $R$. Then we define
$$d_R(\mu_1, \mu_2)=\text{diam}_{C(R)}(\pi_R(\mu_1)\cup \pi_R(\mu_2)).$$
\end{definition}

\par 
This notion of distance satisfies the triangle inequality, which we will make frequent use of. We see in the following few lemmas that subsurface projections "respects" the geometry of $C(S)$ in a variety of ways. 

\begin{lemma}[\cite{MM00}]\label{ProjLip}
    Let $S$ be a compact surface, and suppose $R$ is an essential subsurface of $S$. Fix a curve or marking $\mu$. Then
    $$\text{diam}_R(\pi_R(\mu))\leq 2.$$
    In particular, if $\alpha_1$ and $\alpha_2$ are disjoint curves with $\pi_R(\alpha_i)\neq \varnothing$, then
    $$d_R(\alpha_1, \alpha_2)\leq 2.$$
    If $\alpha_i$ is a base curve the marking $\mu_i$ and $\pi_R(\alpha_i)\neq \varnothing$,
    $$d_R(\mu_1, \mu_2)-4\leq d_R(\alpha_1, \alpha_2) \leq d_R(\mu_1, \mu_2).$$
    Lastly, the intersection number of components of $\pi_R(\alpha_1)$ and $\pi_R(\alpha_2)$ is bounded above by a function of $i_S(\alpha_1, \alpha_2)$.
\end{lemma}

In particular, if two curves have sufficiently far apart projections to a subsurface, then the pair must intersect in the original surface. 
\par 
The following allows us to control distance of projections to $R$, a subsurface of a compact surface $S$, in terms of intersection number in $S$. 

\begin{lemma}[{\cite[Lemma 2.1]{H01}}]\label{IntersectionBound}
Given vertices $\alpha, \beta \in C(S)$ with $i(\alpha, \beta)>0$,
$$d_{R}(\alpha, \beta)\leq 2+2\log_2(i(\alpha, \beta)).$$
\end{lemma}

The next result, known as the Masur--Minsky Bounded Geodesic Image Theorem for curve complexes, provides strong control on both the local and global behavior of geodesics in $C(\Sigma)$. Note that we phrase this result more explicitly, and in terms of the contrapositive which is how we will typically make use of the result.

\begin{proposition}[{\cite[Theorem 3.1]{MM00}}]\label{MMBGI}
Given a compact surface $S$ and any essential subsurface $R$, there is a number $M$ depending only on $S$ so that for any pair of markings or curves $\mu_1$ and $\mu_2$, if $d_R(\mu_1,\mu_2)\geq M$, then every geodesic in $C(S)$ between $\mu_1$ and $\mu_2$ (between any pair of base curves with nonempty projection to $R$ if these are markings) must have a vertex $\alpha$ with $\pi_R(\alpha)=\varnothing$. In particular, $\alpha$ is disjoint from $\partial R$. 
\end{proposition}

Next we state the Behrstock inequality. It will be key to ensure that certain bounds on the distance between projections of collections of curves are sufficiently large. We use an explicit version due to Mangahas given in \cite{M10}, although the original version is due to Behrstock \cite{B06}.

\begin{proposition}[{\cite[Lemma 2.5]{M10}}]\label{Behrstock}
Let $S$ be a compact surface, and fix $R_1$ and $R_2$ two distinct essential proper subsurfaces so that $\pi_{R_1}(\partial R_2)$ and $\pi_{R_2}(\partial R_1)$ are nonempty. If $A$ is a multicurve, then 
$$d_{R_1}(A, \partial R_2)\geq 10 \implies d_{R_2}(A, \partial R_1)\leq 4.$$
\end{proposition}

We will also need the following distance formula of Masur-Minsky to show that PGF groups are undistorted \cite{MM00}. We introduce some notation. Let $\sigma \in \R_+$. For $M\in \R_+$, define
$$ \{\!\{M\}\!\}_{\sigma}=\begin{cases}
M & \text{if } M\geq \sigma \\
0 & \text{otherwise}
\end{cases} 
$$
We give a basic lemma using this notation.
\begin{lemma}[{\cite[Lemma 2.1]{L21}}]\label{DoubleBracketQIBounds}
    Let $\{x_i\}_{i=1}^N$ and $\{y_i\}_{i=1}^N$ be two finite sequences of nonnegative numbers. Suppose 
    $$x_i\approx_{K,C} y_i$$
    with $K\geq 1$, $C\geq 0$, If $\kappa \geq 2KC$, then 
    $$\sum_{i=1}^N \{\!\{x_i\}\!\}_{\kappa} \preceq_{2K,0} \sum_{i=1}^N \{\!\{y_i\}\!\}_C$$
\end{lemma}

\begin{proposition}[{\cite[Theorem 6.12]{MM00}}]\label{MMDistance}
Fix a marking $\mu$ on a compact surface $S$. There exists a constant $\sigma_0>0$ so that for all $\sigma\geq \sigma_0$, there exists a constant $\kappa$ so that for all $f, g\in \text{MCG}(S)$,
$$d_{MCG(S)}(f,g) \approx_{\kappa} \sum_{R\subset S}\{\!\{d_R(f(\mu),g(\mu))\}\!\}_{\sigma}$$
where the sum is over all isotopy classes of essential subsurfaces of $S$.
\end{proposition}

One last fact that we will often use is that distance between projections is equivariant with respect to MCG$(S)$. Specifically, we note the following, which we will typically use without reference.

\begin{remark}
    If $R$ is a subsurface, then for all $\alpha$ and $\beta$ in $C(S)$ with nonempty projection to $R$ and all $f\in \text{MCG}(S)$,

$$d_{f(R)}(f(\alpha), f(\beta))=d_{R}(\alpha, \beta).$$
\end{remark}

We recall the Nielsen--Thurston classification of elements of $\text{MCG}(S)$, as well as the notion of support of a mapping class which arises from the classification. 

\begin{theorem}\label{NTClassification}[{Nielsen--Thurston Classification}]
    Let $S$ be a compact surface, and fix an element $f\in \text{MCG}(S)$. 
\begin{enumerate}
    \item $f$ is \textit{elliptic} if it has finite order.
    \item $f$ is \textit{reducible} if it preserves some multicurve $A$. Such a multicurve is called a \textit{reducing system} for $f$. 
    \item $f$ is \textit{pseudo-Anosov} if it is not elliptic or reducible.
\end{enumerate}
\end{theorem}
The Nielsen--Thurston classification gives further structure to pseudo-Anosov elements \cite{FLP12}, but we will not need this. Our interest in pseudo-Anosovs is that they are exactly the elements of MCG$(S)$ that act loxodromically on $C(S)$ \cite{MM98}.
\par 
An important example of a reducible mapping class is a Dehn twist. The \textit{Dehn twist} on a curve $\alpha$, denoted $\tau_{\alpha}$, is the image of a generator under the inclusion of $\text{MCG}(A)\cong \Z$ into $\text{MCG}(S)$, where $A$ is the annulus with core curve $\alpha$. In particular, it has a representative homeomorphism that is the identity outside of $A$. We will make the convention that our Dehn twists are right handed. More generally, if we fix a multicurve $A$ with components $\alpha_1,\ldots, \alpha_m$, a \textit{multitwist} $\tau$ is an element of MCG$(S)$ of the form
$$\tau=\tau_{\alpha_1}^{n_1}\cdots \tau_{\alpha_n}^{n_m}$$
where $n_i\in \Z$. A \textit{multitwist group} is a group generated by multitwists on multicurves whose components all lie in some fixed multicurve. In particular, a multitwist group is abelian and all its elements are multitwists. Whenever we write a multitwist as a product of Dehn twists, we assume that the given element is fully reduced.
\par 
Given a reducible $f\in \text{MCG}(\Sigma)$, by \cite{I96} there is a \textit{canonical reducing system} $A$, that is a multicurve, unique up to isotopy, so that there is a representative homeomorphism $\tilde{f}\in \text{Homeo}(S, \partial S)$ of $f$ so that some power $\tilde{f}^n$ of $\tilde{f}$ fixes each component of a tubular neighborhood $N(A)$ of $A$ and stabilizes each component of $S\setminus N(A)$. Further, the restriction of $\tilde{f}^n$ to each such component of $S\setminus N(A)$ is the identity or a pseudo-Anosov, and the restriction to each component of $N(A)$ is a (possibly trivial) power of a Dehn twist.
\par 
The \textit{support} of a mapping class $f\in \text{MCG}(S)$ is  defined as follows. If $f$ is not reducible, then the support of $f$ is $S$. Otherwise, given a canonical reducing system $A$ for $f$, a representative homeomorphism $\tilde{f}$ for $f$ that has a power $\tilde{f}^n$ stabilizing $N(A)$ and every component of $S\setminus N(A)$, the support is defined as the union of the isotopy classes of the $\tilde{f}$ orbits of the collection of components $R$ of $N(A)$ and $S\setminus N(A)$ so that the the induced action of $\tilde{f}^n$ on $C(R)$ is loxodromic. 
\par 
We note the following special type of reducible element, which we will make use of in Section \ref{SectionStatementsProofsUndistorted}.
\begin{definition}\label{pureRed}
    An element $f\in \text{MCG}(S)$ is a \textit{pure reducible mapping class} if it is reducible with reducing system $A$ so that there is a representative $\tilde{f}\in \text{MCG}(S)$ of $f$ that stabilizes each component of a tubular neighborhood $N(A)$ and each component of $S\setminus N(A)$. We further say that $f$ is a \text{partial pseudo-Anosov} if $f$ acts as a pseudo-Anosov on each component of $S\setminus N(A)$.
\end{definition}
In other words, $f$ is pure reducible if $\tilde{f}$ already stabilizes the components of $N(A)$ and $S\setminus N(A)$, instead of having to take a power.

\par 
\subsection{Relative Hyperbolicity}\label{SubsectionRelativeHyperbolicity}
We now discuss the formulation of relatively hyperbolic groups that we will use. Relatively hyperbolic groups were originally introduced by Gromov \cite{G87}, and expanded upon by Farb \cite{F98} and Bowditch \cite{B99}. The definition we give here is equivalent to a definition given in \cite{B99}. See \cite{H10} and \cite{S10} for more discussion on the various equivalent definitions of relatively hyperbolic groups and the definitions and facts discussed here.

\begin{definition}\label{RelGenSets}[Relative generating sets]
Let $G$ be a group and $\mathcal{H}=\{H_i\}_{i\in I}$ a (possibly empty) collection of subgroups of $G$. A set $X\subset G$ with $X=X^{-1}$ is said to \textit{generate} $G$ \textit{relative to the collection} $\mathcal{H}$ if $X\cup \bigcup_{i\in I} \{H_i\}$ is a generating set for $G$. In this case we say that $X$ is a \textit{relative generating set} of $G$. Note that it is possible for $X$ to be empty if $\mathcal{H}$ is nonempty.
\end{definition}
%

\begin{definition}\label{RelCayleyGraph}[Relative Cayley graph]
Suppose $G$ is generated by $X$ relative to a collection of subgroups $\{H_i\}$. Fix a generating set $X_i=X_i^{-1}$ of $H_i$ for all $i$. We can define the (right) \textit{relative Cayley Graph} $C(G, X, \{X_i\})$ of $G$ to be the graph whose vertices are elements of $G$, and two vertices $g_1, g_2\in G$ are connected by an edge if there exists an element $s$ of $X \cup \bigcup X_i$ so that $g_1s=g_2$. We make $C(G,X, \{X_i\})$ into a geodesic metric space by letting all edges have length $1$. The induced metric on $G$ will be denoted by $d_G$, and word length by $|\cdot |_{G}$.
\end{definition}

Of course, this is just the normal Cayley graph for $G$ with generating set $X\cup \bigcup X_i$. Also, if the collection $\{H_i\}$ is empty, then this is just the Cayley graph of $G$ with the generating set $X$. We prefer to think of it like this because in the future we will want to think of elements of $X$ and elements of the $H_i$'s as being distinct (see the next definition).

\begin{definition}\label{ConedOffCG}[Coned off Cayley graph]
Let $G$ be a group with generating set $X$ relative to a collection of subgroups $\{H_i\}$, each with a fixed generating set $X_i$. For every coset $gH_i$, let $\nu(gH_i)$ denote a point. Let us form a new graph from $C(G, X, \{X_i\})$, which we will denote as $\widehat{G}$. The graph $\widehat{G}$ is obtained from $C(G,X,\{X_i\})$ by adding one edge of length $1/2$ from $\nu(gH_i)$ to every vertex in $gH_i$. We call $\widehat{G}$ the \textit{coned off Cayley graph} of $G$. The induced graph metric making $\widehat{G}$ into a geodesic space is denoted by $d_{\widehat{G}}$. We will write $|g|_{\widehat{G}}$ for $d_{\widehat{G}}(id, g)$. We also use $\widehat{B}(g, R)$ to denote the ball of radius $R$ in $\widehat{G}$ centered at $g\in G$. 
\end{definition}

We make three remarks.
\begin{enumerate}
    \item The graph $\widehat{G}$ is easily seen to be quasi-isometric to the relative Cayley graph $C(G, X, \{H_i\})$, and the relative Cayley graph $C(G, X, \{X_i\})$ naturally embeds as a subgraph of $\widehat{G}$.
    \item While $\widehat{G}$ depends on the chosen generating sets, there will be no ambiguity if we don't include it in our notation, as the generating set will always be clear or unimportant.
    \item For PGF groups, we will slightly modify the definition of the coned off Cayley graph to potentially have more than $1$ cone point for a given peripheral subset, to account for the assumption of equivariance in the definition (see Definition \ref{PGFGrp}). This modified graph is uniformly quasi-isometric to $\widehat{G}$ as defined here.
\end{enumerate}
To define relatively hyperbolic groups, we require one more graph theoretic condition.

\begin{definition}\label{Fine}[Fine graphs]
A graph $X$ is said to be \textit{fine} if for all vertices $x, y\in X$ and all $n\in \N$, the number of embedded paths between $x$ and $y$ of length $n$ is finite.
\end{definition}

We never use fineness directly, but its consequences are of key importance for many of the following results. 

\begin{definition}\label{RelHypGroups}[Relatively hyperbolic groups]
A group $G$ is \textit{relatively hyperbolic} if there exists a finite collection of proper finitely generated subgroups $\{H_i\}$ and a finite relative generating set $X$ so that $\widehat{G}$ is fine and hyperbolic. In this case we say that $G$ is \textit{hyperbolic relative} to $\{H_i\}$, and the collection of all cosets of the $H_i$'s are called the \textit{peripheral subsets} of $G$, while the conjugates of the $H_i$'s are the \textit{peripheral subgroups}. We will always give $G$ as above the metric given by the Cayley graph $C(G, X, \{X_i\})$ where $X_i$ is some choice of finite generating set for $H_i$. 
\end{definition}

A useful tool for studying relatively hyperbolic group $G$ are the closest point projection maps to its peripheral sets. These will be important to understand the geometry of $G$ and $\widehat{G}$, as well as maps from them into various spaces. See \cite{S13} for more about these projection maps. 

\begin{definition}\label{ProjRelHyp}[Peripheral projections]
Let $G$ be a relatively hyperbolic group, and let $P$ be a peripheral subset. For $g\in G$, we denote by $\pi_P(g)$ the set of points of $P$ that are within $d_G(g, P)+1$ from $g$. If $g_1, g_2 \in G$, we let
$$d_P(g_1,g_2):= \text{diam}(\pi_P(g_1)\cup \pi_P(g_2))$$
where the diameter is measured in the word metric on the peripheral subgroup that has $P$ as a coset.
\end{definition}

We will want to be able to "lift" geodesics in $\widehat{G}$ to paths in $G$. 

\begin{definition}\label{LiftRelHyp}[Lifts]
Let $\widehat{\gamma}$ be a geodesic in $\widehat{G}$. We define a \textit{lift} $\gamma$ of $\widehat{\gamma}$ as follows. The geodesic $\widehat{\gamma}$ can only pass through a given peripheral $P$ at most once (a path entering $P$ twice can be shortened). If $\widehat{\gamma}$ passes through a peripheral $P$, and it also passes through $\nu(P)$, then $\widehat{\gamma}$ passes through exactly two vertices $p_1,p_2$ of $P$. Replace the length $1$ subpath of $\widehat{\gamma}$ between these two points passing through $\nu(P)$ with a shortest length path in $C(G, X, \{X_i\})$ all of whose vertices are in $P$. Doing this for all peripherals, we obtain a path $\gamma$ in $C(G, X, \{X_i\})$, which we call a \textit{lift} of $\widehat{\gamma}$. We parameterize any lift by arc length.
\end{definition}

Here we list some lemmas about the projection maps and these lifts. In Lemmas \ref{SistoEndpointBound} and \ref{SistoBGI}, one should think of part (a) as a "coarse" version of the statement, and part (b) as an "exact" version of the statement. This intuition can be made more precise via asymptotic cones, see \cite{DS05}. 
\par 
Let us fix a relatively hyperbolic group $G$ with relative generating set $X$, and a peripheral subset $P$.

\begin{lemma} [{\cite[Lemma 1.13]{S13}}]\label{SistoEndpointBound}
\begin{enumerate}[label=(\alph*)]
    \item If $\alpha$ is a continuous $(K,C)$-quasi geodesic connecting a point $x\in G$ to $P$, then there is a $D_0=D_0(K,C)$ so that for $D\geq D_0$, there is an $E$ so that the first point in $\alpha \cap N_D(P)$ is at a distance less than or equal to $E$ from $\pi_P(x)$. 
    \item There is an $E$ so that if $\widehat{\gamma}$ is a geodesic in $\widehat{G}$ connecting $x\in G$ to $P$ then the first point in $\widehat{\gamma}\cap P$ is at most $E$ from $\pi_P(x)$.
\end{enumerate}
\end{lemma}
We note that the $E$ in the two parts of the previous lemma are different. The $E$ in (a) depends on $D$, while the $E$ is (b) is absolute. We will never make use of part (a), however, so there is no risk of confusion. We only state it for completeness.

\begin{lemma}[{\cite[Lemma 1.15]{S13}}]\label{SistoBGI}
There is an $L$ and $R=R(K,C)$ so that if $d_P(x,y)\geq L$, then
\begin{enumerate}[label=(\alph*)]
    \item All $(K,C)$-quasi geodesics connecting $x$ and $y$ intersect both  $B_G(\pi_P(x), R)$ and $B_G(\pi_P(y),R)$.
    \item All geodesics in $\widehat{G}$ connecting $x$ to $y$ pass through $\nu(P)$.
\end{enumerate}
\end{lemma}

The following lemma shows that distinct peripheral subsets cannot "fellow travel".

\begin{lemma}[{\cite[Lemma 1.9]{S13}}]\label{lem:PeripheralsFellowTravel}
    For all $H\geq 0$ there is a $B\geq 0$ so that for all pairs of peripheral subsets $P$ and $Q$ with $P\neq Q$, we have that $\text{diam}(N_H(P)\cap N_H(Q))\leq B$.
\end{lemma}

This implies the following simple result, which we will use in the proof of Theorem \ref{Undistorted}.

\begin{lemma}\label{lem:distincePeripherals}
    Suppose $G$ is hyperbolic relative to $\mathcal{H}$, Fix $H_1, H_2 \in \mathcal{H}$ and assume $f_1H_1f_1^{-1}=f_2H_2f_2^{-1}$ for some $f_1, f_2\in G$. Then $f_1H_1=f_2H_2$. 
\end{lemma}
\begin{proof}
    We have that $f_1H_1=(f_2H_2)f_2^{-1}f_1$. This implies that there is some $H\geq 0$ so that $\text{diam}(N_H(f_1H_1)\cap N_H(f_2H_2))=\infty$. By Lemma \ref{lem:PeripheralsFellowTravel}, this is only possible if $f_1H_1=f_2H_2$.
\end{proof}

\par 
We have a distance formula analogous to Proposition \ref{MMDistance} for relatively hyperbolic groups. Comparing it and Proposition \ref{MMDistance} will be the final step in showing that PGF groups are undistorted.

\begin{proposition}[{\cite[Theorem 0.1]{S13}}]\label{SistoDistance}
Let $G$ be a relatively hyperbolic group and let $\mathcal{P}$ denote its collection of peripheral subsets. Then there is a $\sigma_0$ so that for $\sigma\geq \sigma_0$, there is a $\kappa$ so that for all $g_1, g_2\in G$,
$$d_G(g_1,g_2)\approx_{\kappa} d_{\widehat{G}}(g_1,g_2)+\sum_{P \in \mathcal{P}} \{\!\{d_P(g_1,g_2)\}\!\}_{\sigma}.$$
\end{proposition}

Here we give the necessary results due to Dahmani required to show that the groups arising in Theorem \ref{ComboQIEmb} are actually relatively hyperbolic. Specifically, we reference parts (2),(3), and $(3')$ of Theorem 0.1 of \cite{D03}.
\begin{proposition}[{\cite[Theorem 0.1]{D03}}]\label{DahmaniCombination}
\hphantom
\\
    \begin{enumerate}
        \item Let $G_1$ and $G_2$ be hyperbolic relative to finite collections $\mathcal{H}_1$ and $\mathcal{H}_2$ respectively. Fix $H$ a conjugate of an element of $\mathcal{H}_1$. Suppose $H$ embeds as a subgroup of some $H'\in \mathcal{H}_2$. Then $\Gamma= G_1*_H G_2$ with the amalgamation defined via the embedding of $H$ into $H'$ is hyperbolic relative to $\mathcal{H}_1'\cup \mathcal{H}_2$, where $\mathcal{H}_1'$ is $\mathcal{H}_1$ with the group conjugate to $H$ removed.  
         \item Let $G$ be a group that is hyperbolic relative to a finite collection of subgroups $\mathcal{H}$. Let $H$ be a $G$ conjugate of some element of $\mathcal{H}$, and let $A$ be a finitely generated group that $H$ embeds into. Then $\Gamma = A *_{H} G$ is hyperbolic relative to $\mathcal{H}' \cup \{A\}$, where $\mathcal{H}'$ is $\mathcal{H}$ with the group conjugate to $H$ removed.
        \item Let $G$ be hyperbolic relative to $\mathcal{H}$, and suppose we take distinct $H_1, H_2\in \mathcal{H}$ with an isomorphism $\phi:H_1\to H_2$. Let $\Gamma=G*_{\phi}$ be the corresponding HNN extension. Then $\Gamma$ is hyperbolic relative to $\mathcal{H}-\{H_1\}$.
    \end{enumerate}
\end{proposition}

We next give the following proposition due to Osin, which among other things will be essential to show that the groups in Theorem \ref{ComboQIEmb} will actually inject into MCG$(\Sigma)$.
\begin{proposition}[{\cite[Theorem 1.14]{O06}}]\label{NonperiLoxodromic}
Let $G$ be a relatively hyperbolic group, and suppose $g\in G$ has infinite order and is not conjugate into any peripheral subgroup. Then there exists a $\lambda>0$ such that
$$d_{\widehat{G}}(e, g^n)\geq \lambda |n|$$
for all $n\in \Z$. In particular, $g$ acts loxodromically on $\widehat{G}$.
\end{proposition}

We end with a proposition which is a consequence of some of the main results of Dahmani--Guirardel--Osin in \cite{DGO17}. We provide a short proof, without completely defining all the relevant terms. Those interested in more detail should look at the referenced work.
\begin{proposition}\label{DGORotatingSubgroups}
    Given a relatively hyperbolic group $G$, hyperbolic relative to $\mathcal{H}=\{H_1,\ldots, H_n\}$, then there are finite sets $F_i\subset H_i$ so that if $N_i\lhd H_i$ is a normal subgroup with $N_i \cap F_i=\varnothing$, then the smallest normal subgroup in $G$ containing each $N_i$ is equal to a free product of (possibly infinitely many) $G$-conjugates of the $N_i$'s. Further, every element of this subgroup is either conjugate into $N_i$ or acts loxodromically on $\widehat{G}$.
\end{proposition}
\begin{proof}
    Proposition 4.28 of \cite{DGO17} allows us to phrase both Theorem 5.3 and Corollary 6.36 of the same paper in terms of relatively hyperbolic groups, instead of hyperbolically embedded subgroups as they are given there. Then Corollary 6.36 gives for every $\alpha>0$ and every $i$ a finite set $F_i\subset H_i$ so that if $N_i\lhd H_i$ and $N_i\cap F_i=\varnothing$, then the collection $\{N_i\}$ is $\alpha$-rotating, as defined in Definition 5.2 of the paper. Then by taking $\alpha$ sufficiently large, Theorem 5.3 gives the desired result about the smallest normal subgroup in $G$ containing each $N_i$.
\end{proof}

\section{Parabolically Geometrically Finite Groups}\label{SectionParabolicallyGeometricallyFinite}

\subsection{Definitions and Examples}\label{SubsectionDefinitionsExamples}

Recall that $\Sigma$ is a closed oriented surface of genus $g\geq 2$. 

\begin{definition}\label{TwistGrp}[Twist group]
A subgroup $H<\text{MCG}(\Sigma)$ is a \textit{twist group} if it is infinite and contains a finite index subgroup that is a multitwist group. Any finite index multitwist subgroup consists of elements twisting on the components of some maximal multicurve $A$. We call $A$ the multicurve \textit{associated} to $H$.
\end{definition}

Here we give the definition of the main object of study of this paper, proposed in \cite{DDLS21}.

\begin{definition}\label{PGFGrp}[PGF group]
A subgroup $G<\text{MCG}(\Sigma)$ is a \textit{parabolically geometrically finite group}, or a \textit{PGF group} if 
\begin{enumerate}
    \item $G$ is hyperbolic relative to a finite collection $\{H_i\}$ of twist subgroups on the multicurves $\{A_i\}$.
    \item $\widehat{G}$ admits an equivariant QI embedding into $C(\Sigma)$.
\end{enumerate}
Here, $\widehat{G}$ is a modification of Definition \ref{ConedOffCG}, where instead of one cone point, $\nu(gH_i)$ consists of $|A_i|$ points all of which are connected to the points of $gH_i$ by edges of length $1/2$. If $\alpha \in gA_i$, we let $\nu(\alpha)$ denote the corresponding cone point, and $\nu(gH_i)=\bigcup_{\alpha \in gA_{i}} \nu(\alpha)$. The group $G$ acts on $\widehat{G}$ in a natural way, via permutations of cone points defined by $g\nu(\alpha)=\nu(g(\alpha))$.
\par 
We will say that $G$ is \textit{PGF relative to} $\{H_i\}$ if it is PGF using this collection of twist groups. Whenever a PGF group is given, it is implicitly assumed that such a collection has already been chosen. 
\end{definition}

In particular, convex cocompact groups are PGF groups relative to the empty set. We note that the issue of having extra cone points is only for equivariance. Namely, twist groups can permute the components of their associated multicurve, and if this happens then there would be no way to guarantee equivariance if we only used a single cone point. 
\par 
Definition \ref{PGFGrp} does not specify the image of cone points, but there is a natural choice for what their image should be. Namely, $\nu(\alpha)$ should be sent to $\alpha$. The following lemma shows that the modification of an equivariant quasi-isometric embedding to take on these values on the cone points will still be a quasi-isometric embedding. Note that it suffices to look only at the vertices, as a quasi isometric embedding defined on the vertices extends equivariantly to a quasi-isometric embedding on the whole graph. 

\begin{lemma}\label{ConePointImages}
Let $G$ be a PGF group relative to $\{H_1,\ldots, H_n\}$. Fix any $\gamma \in C(\Sigma)$. Then the map $\psi:\widehat{G}\to C(\Sigma)$ given by $\psi(g)=g\gamma$ and $\psi(\nu(\alpha))=\alpha$, where $\alpha$ is any component of the multicurve associated to the peripheral subset $P$, is a quasi-isometric embedding. 
\end{lemma}
\begin{proof}
By the definition of PGF groups, the restriction of $\psi$ to the vertices in $G$ (with the $\widehat{G}$ metric) is a QI embedding for any choice of $\gamma\in C(\Sigma)$. Indeed, the definition implies that there is some choice of $\gamma$ so that this restriction is a QI embedding, and the triangle inequality implies that it is true for any. Let $D$ denote the maximum distance in $C(\Sigma)$ from $\gamma$ to any component of any multicurve associated to one of $H_1,\ldots, H_n$. Fix a peripheral $gH_i$. It follows that $g\gamma$ is at most $D$ from any component of $gA_i$, where $A_i$ is the multicurve associated to $H_i$. Thus $gA_i$ is a uniformly finite distance from $gH_i \cdot \gamma$, so it follows that $\psi$ is also a QI embedding. 
\end{proof}
\par 
 We give here the two types of examples of PGF groups that were known before the writing of this paper.

\begin{proposition}[{\cite[Theorem 1.1]{L21}}]\label{LoaCombo}
    There exists a constant $D_0\geq 3$ independent of $\Sigma$ with the following property. Let $A$ and $B$ denote two multicurves with $d_{\Sigma}(A,B)\geq D_0$. Fix multitwist groups $H_A$, $H_B$ generated by multitwists on multicurve subsets of $A, B$ respectively. Then the natural homomorphism $\Phi: H_A * H_B \to \<H_A, H_B\>$ is injective and hence an isomorphism, and $\<H_A, H_B\>$ is PGF relative to $\{H_A, H_B\}$. Any element not conjugate into a factor is pseudo-Anosov.
\end{proposition}

In Example \ref{FreeProductMultitwists}, we prove a related result (note that the results of this paper do not prove Proposition \ref{LoaCombo}). 
\par
We have the following result due to Tang \cite{T21}, although the referenced paper does not state the result in the language of PGF groups. 

\begin{proposition}[{\cite[Theorem 1.3]{T21}}]\label{TangVeechPGF}
    Finitely generated Veech groups of MCG$(\Sigma)$ are PGF relative to any maximal parabolic subgroup.
\end{proposition}
Before going into the details of the combination theorem, we give the following result of Leininger--Reid which serves as a direct inspiration for it. See \cite{LR06} for the relevant notation and definitions.

\begin{proposition} [{\cite[Theorem 6.1]{LR06}}]\label{LeiningerReidVeechCombo}
    Suppose $G(q_1)$ and $G(q_2)$ are finitely generated Veech groups, and $h$, $G_0$, $G(q_1)$ and $G(q_2)$ are compatible along the sparse multicurve $A_0$. Then there exists a $K$ so that the natural map
    $$G(q_1)*_{G_0}h^KG(q_2)h^{-K} \to \text{MCG}(\Sigma)$$
    is injective. Moreover, every element not conjugate into an elliptic or parabolic subgroup of either factor is pseudo-Anosov.
\end{proposition}
It will follow from Theorem \ref{ComboQIEmb} that the groups from Proposition \ref{LeiningerReidVeechCombo} are PGF (with $K$ potentially taken to be larger). 
\par 
The main results of this paper vastly generalize Proposition \ref{LeiningerReidVeechCombo}. Not only does it provide a combination theorem for PGF groups which also allows for more general graphs of groups, the techniques we prove also give a framework to prove other combination theorems in other contexts as well. For example, see Theorems \ref{ConvexCocompactCombo}, \ref{ConvexCoCptTwists}, and \ref{LoaAnalog}.

\subsection{PGF Graphs of Groups}\label{SubsectionPGFGraphsofGroups}

We now introduce the types of combinations of PGF groups we will be working with. For simplicity, we will assume the underlying graphs of the graphs of groups are of a particular form, which we detail in the following definition. This definition is given to be as general as reasonably possible for Theorem \ref{ComboQIEmb} to still hold. It is much easier in practice to consider simpler examples, such as those that appear in Example \ref{GraphOfGroupsExamples}. We suggest that the reader only worry about parts (a) and (b) of the following definition when reading Section \ref{SectionExamplesApplications}, as they are not at all relevant in the proof of Theorem \ref{ComboQIEmb}.

\begin{definition}\label{PGFGraph}[PGF graph of groups]
    Let $\mathcal{G}$ be a graph of groups. We allow for multiple edges between pairs of vertices, but disallow loop edges. We say that $\mathcal{G}$ is a \textit{normalized PGF graph of groups} if its vertex groups are all PGF groups and twist groups and its edge groups are all twist groups, and so that all the following holds. We assume there is at least one vertex with a PGF vertex group. Further, for every edge $e$ with associated edge group $H_e$ and vertex groups $G_{e^\pm}$, along with monomorphisms $\phi^{e^\pm}:H_e\to G_{e^{\pm}}$, we assume that $H_e$ has a sparse associated multicurve, and that the following two conditions hold up to change of orientation on $e$.
    \begin{enumerate}[label={(\arabic*)}]
        \item $G_{e^+}$ is a PGF group, $G_{e^-}$ is a twist group, the map $\phi^{e^-}$ is the identity, and there is some $g\in \text{MCG}(\Sigma)$ so that $\phi^{e^+}(h)=ghg^{-1}$ for all $h\in H_e$. Further, we assume that $\phi^{e^+}(H_e)$ is one of the finitely many twist groups that $G^{e^+}$ is PGF with respect to.
        \item $G_{e^{\pm}}$ are both twist groups. In this case $\phi^{e^-}$ is the identity and $\phi^{e^+}$ is inclusion. The image of this inclusion is a direct factor having complementary factor generated by multitwist in a multicurve disjoint from the multicurve associated to $H_e$. Further, no edge other than $e$ contains the vertex with vertex group $G_{e^+}$, and every other edge containing the other vertex of $e$ is of type (1).
    \end{enumerate}
    We impose two more constraints. (a) If $e$ and $e'$ are two distinct edges of type (1) that share a vertex with a PGF vertex group, then the images of the edge groups in the PGF vertex group are distinct. (b) For any edge $e$ as in (2), we assume that every complementary component of the multicurve associated to $G_{e^-}$ contains a component of the multicurve associated to $G_{e^+}$. Further, we also assume that for every such complementary component $S$, every multitwist element of $G_{e^+}\setminus G_{e^-}$ twists on some curve contained in $S$.
    \par 
    We call a vertex $v$ of the associated Bass--Serre tree a \textit{PGF vertex} if the stabilizer of $v$ is a PGF group. \textit{Twist vertices} are defined similarly. Vertices of $\mathcal{G}$ with PGF vertex groups will also be called \textit{PGF vertices}, and similarly for those with twist vertex groups. Vertices whose groups extend a twist vertex group as in (2) will be called \textit{extension vertices} and all other twist vertices will be called \textit{base vertices}, and we use this terminology in both $T$ and $\mathcal{G}$. Lastly, if $H$ and $H'$ are vertex groups as in (2) so that $H'$ contains $H$ as a direct factor, then we say that $H'$ \textit{extends} $H$ \textit{by multitwists}, and that $H'$ is an \textit{extension} of $H$. If $\tau\in H'$ can be written as a composition of Dehn twists on curves not in the multicurve associated to $H$, we say that $\tau$ is a \textit{word in the new multitwists of} $H'$. 
\end{definition}
%
\begin{figure}[H]
    \centering
    \includegraphics[scale=.64]{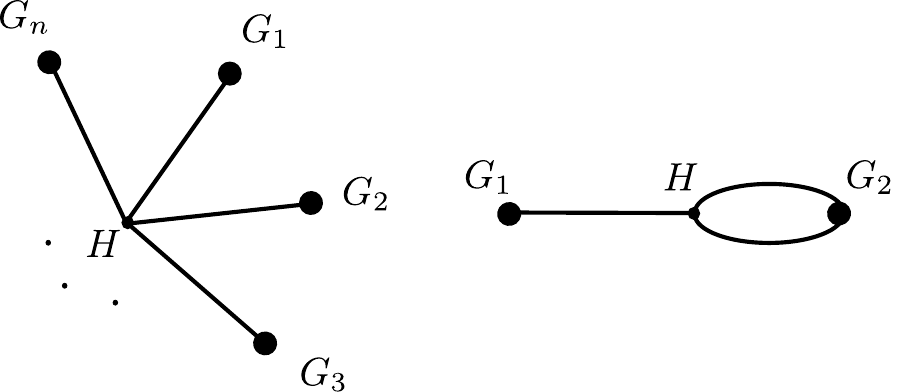}
    \caption{Two different examples of normalized PGF graphs of groups}
    \label{FigurePGFExamples}
\end{figure}
%

\begin{example} \label{GraphOfGroupsExamples}
Here we give a few examples of normalized PGF graphs of groups, along with a nonexample. In Figure \ref{FigurePGFExamples}, on the left we have $n$ PGF groups $G_1,\ldots G_n$ with one twist group in each being identified with every other via a single twist group $H$. On the right, $H$ is identified with one twist group in $G_1$ and two twist groups of $G_2$, and by definition the latter two twist groups must be distinct and nonconjugate. In Figure \ref{FigurePGFExtension}, $H'$ is an extension of $H$ as in Definition \ref{PGFGraph}(2). Note that in the Bass--Serre tree, such vertices as in Figure \ref{FigurePGFExtension} give pairs of PGF vertices of the tree that have no other PGF vertices between them (that is, there are no other PGF vertices along the geodesic between them), but have more than one vertex between them. Occasionally the language has to account for this (see Lemma \ref{LocalLargeEquiv} for example). See Figure \ref{FigurePGFExtensionBSTree}.
\par 
\begin{figure}[H]
    \centering
    \includegraphics[scale=.65]{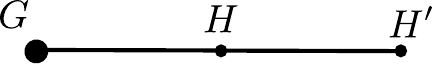}
    \caption{An example of an extension as in Definition \ref{PGFGraph}(2)}
    \label{FigurePGFExtension}
\end{figure}
%
\begin{figure}[H]
    \centering
    \includegraphics[scale=.6]{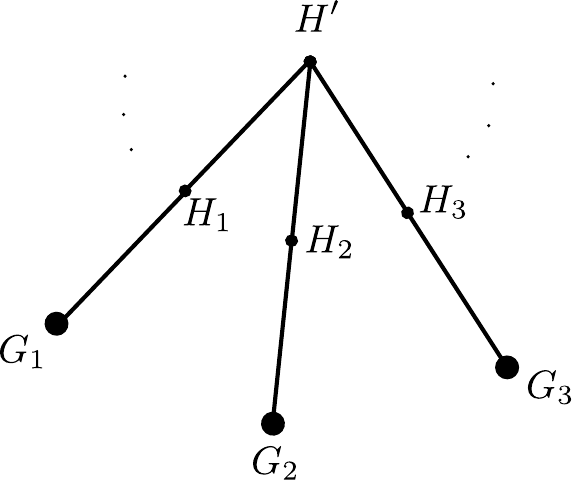}
    \caption{A piece of the Bass--Serre tree from the graph in Figure \ref{FigurePGFExtension}. Here each $G_i$ is conjugate to $G_j$ by an element in the new multitwists of $H'$. Every $H_i$ is equal, but the cosets associated to the labelled vertices are distinct.} 
    \label{FigurePGFExtensionBSTree}
\end{figure}
\end{example}

\par 
On the other hand, Figure \ref{FigurePGFNonExample} gives a nonexample. We assume here that every pair of twist groups map to the same twist subgroup of $G_i$. That is, $H_{12}$ and $H_{13}$ map to the same twist subgroup of $G_1$, and similarly for $G_2$ and $G_3$. There are two issues illustrated in this example. The first is from having the same twist subgroup in each $G_i$ identified with two different twist vertices. This violates Definition \ref{PGFGraph}(a). The issue that arises is that the stable letter of the loop in Figure \ref{FigurePGFNonExample} would have to map to a reducible element in MCG$(\Sigma)$, as it will have to commute with the common images of $H_{12}, H_{13}$ and $H_{23}$. But stable letters, being nonperipheral, must map to pseudo-Anosov elements if the image is PGF (see Corollary \ref{InfOrderpA}).
\begin{figure}[H]
    \centering
    \includegraphics[scale=.75]{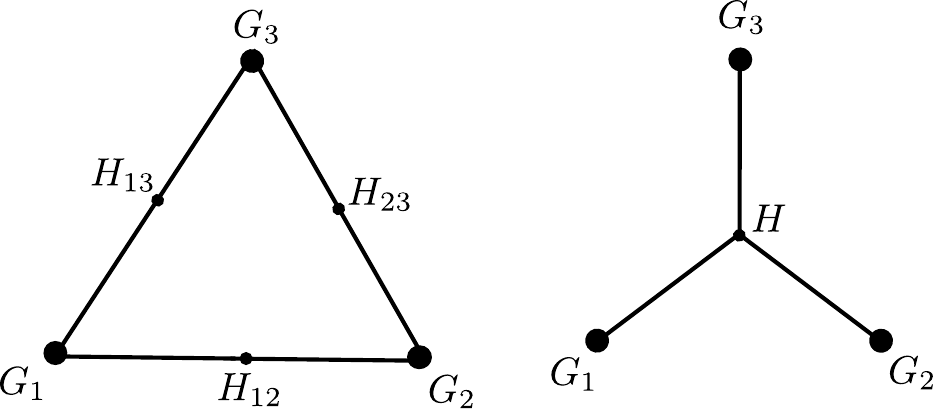}
    \caption{On the left an example of a graph of PGF groups that is not a normalized. The right side involves the same groups and homomorphisms, but is now normalized.}
    \label{FigurePGFNonExample}
\end{figure}
\par 
Instead, if one wishes to combine the three PGF groups $G_1$, $G_2$, and $G_3$, one may for example use the graph on the right side of the figure, which combines all the twist vertices into one vertex with vertex group $H$, eliminating the stable letter. 
\par

We note the following proposition.

\begin{proposition}\label{ComboRelHyp}
    Fix a normalized PGF graph of groups $\mathcal{G}$ with fundamental group $G$. Then $G$ is hyperbolic relative to the collection of twist subgroups that the PGF vertices of $\mathcal{G}$ are hyperbolic relative to, with some twist groups identified with the vertex group of a twist vertex as in Definition \ref{PGFGraph}(1), or extended as in Definition \ref{PGFGraph}(2). 
\end{proposition}
\begin{proof}
    The proof is effectively just inductively applying Proposition \ref{DahmaniCombination}. More specifically, one may first choose a maximal tree $\mathcal{T}$ of $\mathcal{G}$. Starting at a base PGF vertex $v$, one can inductively apply Proposition \ref{DahmaniCombination}(1) to pairs of PGF vertices in $\mathcal{T}$, and then apply \ref{DahmaniCombination}(2) to the extension vertices.  Any edge not in the maximal tree contains no extension vertices, and we may apply Proposition \ref{DahmaniCombination}(3) to each such edge one at a time. 
\end{proof}

\begin{remark}
There is still only one cone point collection for every twist group extension, and we will continue to denote the cone point of a peripheral subset using the original twist group, instead of its extension. (In particular, there may be infinitely many cosets giving the same cone points). By doing this we are somewhat abusing our convention that there should be one cone point for every component of the associated multicurve (as we are not including new cone points for the components of the multicurve being added in the extension), but this will not change the proof in a substantial way. 
\end{remark}

\begin{definition}\label{AssMC}[Associated multicurves]
Let $G$ be a subgroup of MCG$(\Sigma)$, and suppose $H$ is a maximal twist subgroup of $G$. The multicurve $A$ associated to $H$ is also said to be \textit{associated} to the group $G$. 
\end{definition}

In particular, if $G$ is a PGF group relative to $\mathcal{H}$, the multicurves associated to $G$ are the $G$ orbits of the multicurves associated to the elements of $\mathcal{H}$.
\par 
\begin{definition}\label{CompHom}[Compatible Homomorphism]
    Let $\mathcal{G}$ be a normalized PGF graph of groups with fundamental group $G$. A homomorphism $\phi:G \to \text{MCG}(\Sigma)$ is a \textit{compatible homomorphism} if its restriction to the vertex groups of $\mathcal{G}$ respects the inclusion of the vertex groups into MCG$(\Sigma)$, up to conjugation in MCG$(\Sigma)$.
\end{definition}
Note that this definition also implies that the vertex stabilizers of vertices in the Bass--Serre tree $T$ also have their $\phi$ image respecting the inclusion of the stabilizer in MCG$(\Sigma)$, up to conjugation. This is because each vertex stabilizer is conjugate in $G$ to a vertex group.
\par
In Section \ref{SectionExamplesApplications}, we will provide a way to construct desirable compatible homomorphisms for many examples of normalized PGF graphs of groups. Their actual existence is not important for the proof of Theorem \ref{ComboQIEmb}, however.
\par 
Suppose we have $\mathcal{G}$ a normalized PGF graph of groups with fundamental group $G$ and a compatible homomorphism $\phi:G \to \text{MCG}(\Sigma)$. We can define a map $\Psi :\widehat{G}\to C(\Sigma)$ as follows. We fix some $\gamma\in C(\Sigma)$, and for $g\in G$ we define 
$$\Psi(g)=\phi(g)(\gamma).$$
To define $\Psi$ on the cone points of $\widehat{G}$, for each peripheral subset $P$, we can choose a bijection $\theta_P$ from $\nu(P)$ to the components of the multicurve associated to the peripheral subset $\phi(P)\subset\text{MCG}(\Sigma)$ so that we can define $\Psi$ on $\nu(P)$ so it satisfies
$$\Psi(gx)=\phi(g)\theta_P(x),$$
where $x\in \nu(P)$ and $g\in G$. In other words, we can define $\Psi$ on $\nu(P)$ by matching up via $\theta_P$ the elements of $\nu(P)$ with the components of the multicurve of $\phi(P)$ in such a way so that this "matching up" is $\phi$-equivariant as above. The map $\theta_P$ may not be unique, but it suffices to simply make some choice of such a map. Only finitely many choices need to be made by equivariance. Lastly, we make some equivariant choice for the images of the edges, where again only finitely many choices have to be made.
\par 
This map $\Psi$ is the map we will utilize in order to show that the image of $G$ (under certain assumptions on $\mathcal{G}$ and $\phi$) by $\phi$ is a PGF group.
\par 
The next lemma allows us to simplify the proof of Theorem \ref{ComboQIEmb} so that we only need to study $\Psi$ restricted to the cone points of $\widehat{G}$. It follows trivially as the cone points of $\widehat{G}$ are $1$-dense.

\begin{lemma}\label{ConePointImages2}
 If $\Psi$ restricted to the cone points of $\widehat{G}$ is a quasi-isometric embedding with the metric on the cone points the restriction of that on $\widehat{G}$, then $\Psi$ itself is also a quasi-isometric embedding.
\end{lemma}
\par 
Throughout, we always assume that given $G$ a fundamental group of a normalized PGF graph of groups and a compatible homomorphism $\phi$, we have an associated equivariant map $\Psi:\widehat{G} \to C(\Sigma)$ that is defined on its cone points as above Lemma \ref{ConePointImages2}. We have the following lemma about $\Psi$, independent of any further assumptions about $\mathcal{G}$.

\begin{lemma}\label{PGFGraphEmbLipschitz}
    The map $\Psi$ as defined above is a coarse Lipschitz map.
\end{lemma}
\begin{proof}
    The restriction of $\Psi$ to $G$ with the $\widehat{G}$ metric is an orbit map of some curve $\gamma\in C(\Sigma)$. There is an $C$ so that for all elements $g\in G$ with $|g|_{\widehat{G}}=1$ and any cone point $x$ of a twist group $H$ that $G$ is hyperbolic relative to,
$$d_{C(\Sigma)}(\gamma, g\gamma)\leq C$$
$$d_{C(\Sigma)}(\gamma, \theta_H(x)) \leq C.$$
Indeed, the first bound follows as $g$ is either an element of the finite relative generating set, or an element of one of the finitely many twist groups that the PGF vertices of $G$ are hyperbolic relative to. All the elements of a twist groups always have a uniform translation bound on any fixed curve as they stabilize a simplex in $C(\Sigma)$. The second inequality follows simply because there are only finitely many choices for $x$. The triangle inequality gives the required Lipschitz upper bound in general. The coarse Lipschitz bound on all of $\widehat{G}$, including on the edges, then follows by equivariance and the triangle inequality.
\end{proof}

  By combining Proposition \ref{ComboRelHyp} and Lemma \ref{PGFGraphEmbLipschitz}, it follows that to establish that the image under a compatible homomorphism $\phi$ of the fundamental group $G$ of a normalized PGF graph of groups $\mathcal{G}$ is a PGF group, we just need to establish that $\phi$ is injective and that $\Psi$ admits coarse Lipschitz lower bounds with respect to distance in $C(\Sigma)$.

We now give the last definition needed to make sense of Theorem $1$. 

\begin{definition}\label{LocalLarge}[Local Large Projections]
    Given a normalized PGF graph of groups $\mathcal{G}$ with compatible homomorphism $\phi$, we say the pair $(\mathcal{G}, \phi)$ satisfies the $L$-\textit{local large projections} property if the following holds. Let $T$ denote the Bass--Serre tree of $\mathcal{G}$. Fix two PGF vertices $v_1$ and $v_2$ of $T$ with stabilizers $G_1$ and $G_2$ with no PGF vertices between them. Let $v_{12}$ denote a base twist vertex between $v_1$ or $v_2$, with stabilizer $H_{12}$ (there are potentially two choices of such a vertex as there may be an extension vertex between $v_1$ and $v_2$, but they have the same stabilizer, see Figure \ref{FigurePGFExtensionBSTree}). Let $A_{12}$ be the multicurve associated to $\phi(H_{12})$, and take $S$ to be any component of $\Sigma \setminus A_{12}$. For all multicurves $B_1$ and $B_2$ associated to $\phi(G_1)$ and $\phi(G_2)$ respectively, that are distinct from $A_{12}$, we have
    
    $$d_S(B_1,B_2)\geq L.$$
\end{definition}

We shall see in the future (see Lemma \ref{FillingMC}) that, following the notation of Definition \ref{LocalLarge}, every component of $B_1$ and $B_2$ has nonempty projection to at least one component of $\Sigma \setminus A_{12}$, and also that for every component $S$ of $\Sigma \setminus A_{12}$ some component of $B_1$ and $B_2$ projects nontrivially to $S$ (without any assumption on the distance between their projections). This second point will be essential in applications of Theorem \ref{ComboQIEmb} (see Section \ref{SectionExamplesApplications}).
\par

We now have the language to state a precise version of Theorem $1$.
\begin{theorem}\label{ComboQIEmb}
Suppose $(\mathcal{G}, \phi)$ satisfies the $L$-local large projection property for $L\geq M+18$ with $M$ as in Proposition \ref{MMBGI}. Then $\phi$ is injective. Further, its image is PGF relative to the $\phi$ images of the twists groups of the PGF vertex groups of $\mathcal{G}$, with any extension group replacing its base group, and some twist groups are removed if they are identified with another as in Proposition \ref{DahmaniCombination}(3). All infinite order elements not contained in a twist group are pseudo-Anosov.
\end{theorem}

As stated after Lemma \ref{PGFGraphEmbLipschitz}, it suffices to prove that $\phi$ is injective and that $\Psi$ admits coarse Lipschitz lower bounds in terms of distance in $C(\Sigma)$. The majority of the work for finding this lower bound is done in Section \ref{SectionWorkingTowardsComboTheorem}. Along the way, we will develop a more general language providing other kinds of combination theorems (see Lemmas \ref{FarCurvesIntersect}-\ref{ShortSegmentsLowerBound} for the general language and Theorems \ref{ConvexCocompactCombo}, \ref{ConvexCoCptTwists}, and \ref{LoaAnalog} for the other examples of combinations).
\par 
We state the following two lemmas, the former of which will be used to show that the compatible homomorphism in Theorem \ref{ComboQIEmb} is injective.

\begin{lemma}\label{EquivQIEmb}
    Let $G$ be a relatively hyperbolic group, and $\phi:G\to \text{MCG}(\Sigma)$ a homomorphism. Fix a map $\Psi:\widehat{G}\to C(\Sigma)$ which is a $\phi$ equivariant quasi-isometric embedding, (that is, $\Psi(gx)=\phi(g)\Psi(x)$ for $g\in G$, $x\in \widehat{G}$). Suppose $f\in G$ has infinite order and is not conjugate into any peripheral subgroup of $G$. Then $\phi(f)$ is pseudo-Anosov. Further, if $\phi$ restricts to an injective map on each peripheral subgroup of $G$ and if every nontrivial finite order element has nontrivial image, then $\phi$ is injective.
\end{lemma}
\begin{proof}
    By Proposition \ref{NonperiLoxodromic}, an infinite order nonperipheral element $f\in G$ acts loxodromically on $\widehat{G}$. As $\Psi$ is a $\phi$-equivariant quasi-isometric embedding, it follows that $\phi(f)$ will also act loxodromically on $C(\Sigma)$. But then $\phi(f)$ is pseudo-Anosov, as by Theorem \ref{NTClassification} any element which is not pseudo-Anosov has a power fixing a curve.
    \par 
    Every element $g\in G$ is either peripheral or not peripheral. If $g$ is peripheral, then by assumption $\phi(g)\neq id$. If $g$ is nonperipheral, then either it has infinite order and hence $\phi(g)$ is pseudo-Anosov (and in particular nontrivial), or $g$ has finite order and by assumption $\phi(g)$ in nontrivial. This shows that $\phi$ is injective.
\end{proof}

\begin{corollary}\label{InfOrderpA}
Let $G$ be a PGF group, and suppose $f$ has infinite order and is not contained in the conjugate of any twist subgroup of $G$. Then $f$ is pseudo-Anosov.
\end{corollary}

\par 

 The following result essentially says that, by equivariance of projections, to show that Definition \ref{LocalLarge} is satisfied, it suffices to check the condition on orbits. This effectively means you can reduce to checking in the graph of groups. The statement of the lemma makes this formal.
 \par 
 The Bass--Serre tree comes equipped with a projection map $\pi: T\to \mathcal{G}$. This map sends PGF vertices to PGF vertices, base vertices to base vertices, and extension vertices to extension vertices.  
 \par 
 For each base vertex $v$ of $\mathcal{G}$, choose a fixed vertex $\widetilde{v}\in T$ so that $\pi(\widetilde{v})=v$, and consider the closure of the component of the $\pi$ preimage of the open $1$ neighborhood of $v$ containing $\widetilde{v}$. This is a subgraph of $T$ consisting of all the edges and their vertices containing $\widetilde{v}$. For convenience let us call this set the \textit{star} with center $\widetilde{v}$. 
 \par 
 For each base vertex $v$ in $\widehat{G}$, we make a choice of a star with center $\widetilde{v}$ for some $\widetilde{v}$ with $\pi(\widetilde{v})=v$. If a given star contains an extension vertex, let $\tau$ be either the identity or denote a word in the new multitwists of the vertex. If there is not a extension vertex, assume $\tau=id$. We then have the following lemma.

\begin{lemma}\label{LocalLargeEquiv}
     Suppose that for each chosen star with center $\widetilde{v}$ with stabilizer $H$, and any pair of distinct PGF vertex groups $G_1$ and $G_2$ of vertices in this star, the following holds. Let $A$ denote the multicurve associated to $\phi(H)$. For all multicurves $B_1$ and $B_2$ associated to $\phi(G_1)$ and $\phi(G_2)$ not equal to $A$, we have for any component $S$ of $\Sigma \setminus A$ that
    $$d_S(B_1,\tau(B_2))\geq L.$$
    Then $(\mathcal{G}, \phi)$ satisfies the $L$-local large projections property.
\end{lemma}
\begin{proof}
    This follows easily from equivariance. Namely, given any two PGF vertices $w_1$ and $w_2$ of $T$ with no PGF vertices between them, we may translate so that the first two or last two edges of $[w_1,w_2]$ lies in one of the chosen stars. If $[w_1,w_2]$ is only length $2$, then we are done by equivariance (namely, the above inequality with $\tau=id$ suffices to give the required inequality between $w_1$ and $w_2$ as in Definition \ref{LocalLarge}). Otherwise, the translate of $[w_1,w_2]$ contains an extension vertex, and in this case we can apply the inequality using a nontrivial $\tau$, a word in the new multitwists of this vertex.
\end{proof}

It may seem initially that Lemma \ref{LocalLargeEquiv} requires checking an infinite number of conditions (coming from the different choices of multitwist $\tau$), but as we will see in Lemma \ref{FarMulticurvesFarTranslation}, there will be no issue in our applications.

\section{Projections and the Local to Global Property}\label{SectionWorkingTowardsComboTheorem}

\subsection{Multicurves of PGF Groups and their projections}\label{SubsectionMulticurvesPGFGroupsProjections}
We start with proving an important result (Proposition \ref{PGFProjBound}) which is the key ingredient to construct actual examples of PGF groups arising from applications of Theorem \ref{ComboQIEmb}. In the statement of Theorem \ref{ComboQIEmb}, it is required that various collections of sets have sufficiently large distance from each other (see Definition \ref{LocalLarge}). To construct examples where such sets exists as in Section \ref{SectionExamplesApplications}, it is useful to know that each of these sets have bounded diameter. Once this is known, we can push these sets to be arbitrarily apart from each other by applying a pseudo-Anosov on the curve complex of the subsurface that the sets lie in. We note however that this result is not actually used in the proof of Theorem \ref{ComboQIEmb}.
\par 
We first set the following notation. If $\mathcal{A}$ is a collection of curves and $\mathcal{T}\subset \text{MCG}(\Sigma)$, we write
$$\mathcal{T}(\mathcal{A})=\{\beta\in C(\Sigma)\ | \ \beta=\tau(\alpha), \tau\in \mathcal{T} \text{ and } \alpha \in \mathcal{A}\}.$$
and for a collection of multicurves $\mathcal{M}$, 
$$\mathcal{T}(\mathcal{M})=\{B\ | \ B=\tau(A), \tau\in \mathcal{T} \text{ and } A \in \mathcal{M}\}.$$
\begin{lemma}\label{InductiveProjBound}
   Fix $\mathcal{M}$ a finite collection of multicurves and $\mathcal{A}$ a finite collection of curves in $\Sigma$. Take $R\in \Z_{\geq 0}$, and let $\mathcal{T}_R$ be the set of all elements of MCG$(\Sigma)$ that can be written as a product of at most $R$ multitwists in the multicurves of $\mathcal{M}$. Let $\mathcal{S}_R$ be the collection of all essential subsurfaces of $\Sigma$ except annuli with core curves that are a component of a multicurve in $\mathcal{T}_R(\mathcal{M})$. Then there is a $K=K(R, \mathcal{M}, \mathcal{A})$ so that for all  $S\in \mathcal{S}_R$,
   $$\text{diam}(\pi_S(\mathcal{T}_R(\mathcal{A}))) \leq K.$$
   In particular, $K$ is independent of $S$. 
\end{lemma}
\begin{proof}
 We use induction on $R$. Fix $\alpha \in \mathcal{A}$. By adding a new curve to $\mathcal{A}$ we may assume there is a $\beta\in \mathcal{A}$ with $d_{\Sigma}(\alpha, \beta)\geq 3$. This will ensure that we can always find nonempty projections to a given subsurface.
 \par 
 We also need to extend every $A\in \mathcal{M}$ to two different pants decompositions $A_1$ and $A_2$ chosen so that for any essential subsurface $S$, either $A_1$ or $A_2$ has nonempty projection to $S$. The fact that one extension does not suffices is because we are also projecting to annuli. We do not twist on the new components of the multicurve, so in particular these new curves do not change the set of annuli excluded from $\mathcal{S}_R$.
\par 
Now, the base case $R=0$ is immediate, as $\mathcal{A}$ is a finite collection, so there is there is a bound on pairwise intersection numbers on pairs of elements of $\mathcal{A}$. The last statement of Lemma \ref{ProjLip} along with Lemma \ref{IntersectionBound} immediately gives a bound independent of $S$.
\par 
Suppose the result is true for $R\geq 0$. Note that 

$$\mathcal{T}_{R+1}(\mathcal{A})=\mathcal{T}_1\mathcal{T}_{R}(\mathcal{A})$$
so $\mathcal{T}_{R+1}(\mathcal{A})$ can be written as a union of translates of $\mathcal{T}_R(\mathcal{A})$ under multitwists on multicurves in $\mathcal{M}$. By equivariance we then obtain for any $\tau \in \mathcal{T}_1$ and $S\in \mathcal{S}_{R+1}$ that
\begin{equation}\label{UniformFiniteDiameter}
\text{diam}(\pi_S(\tau(\mathcal{T}_R(\mathcal{A})))=\text{diam}(\pi_{\tau^{-1}(S)}(\mathcal{T}_R(\mathcal{A}))).    
\end{equation}

The right hand side is uniformly bounded for nonannular $S$ by induction, as $\tau^{-1}(S)\in \mathcal{S}_R$ (it will often lie in $\mathcal{S}_{R+2}\subset \mathcal{S}_{R}$). Thus each translated copy of $\mathcal{T}_R(\mathcal{A})$ by some $\tau \in \mathcal{T}_1$ has uniformly bounded diameter projections to every $S\in \mathcal{S}_{R+1}$.
\par 
On the other hand, given $\tau\in \mathcal{T}_1$ which is a multitwist on $A$, take the extensions $A_1$ and $A_2$ as given above. Fix a subsurface $S\in \mathcal{S}_{R+1}$. Suppose $\pi_S(A_1)\neq \varnothing$ and $\pi_S(\tau(\alpha))\neq \varnothing$. then note that by equivariance
\begin{equation}\label{FiniteDistanceToMulticurves}
    d_{\tau^{-1}(S)}(A_1, \alpha)=d_S(A_1, \tau(\alpha)).
\end{equation}
In particular, as $\alpha\in \mathcal{A}\subset \mathcal{T}_{R+1}(\mathcal{A})$ is some fixed curve and $A_1$ comes from a finite set of multicurves, there is a bound on the left hand side, by Lemmas \ref{ProjLip} and \ref{IntersectionBound}. On the other hand, if $\pi_S(\tau(\alpha))= \varnothing$ or $\pi_S(A_1)=\varnothing$, then we may just do the same argument with $\beta$ and/or $A_2$. We thus obtain using Lemma \ref{IntersectionBound} a bound on the distance from all the projections of the translates of $\mathcal{T}_R(\mathcal{A})$ that make up $\mathcal{T}_{R+1}(\mathcal{A})$ to the projections of the extensions of the multicurves of $\mathcal{M}$.
\par
Thus the projection of $\mathcal{T}_{R+1}(\mathcal{A})$ to any surface in $\mathcal{S}_{R+1}$ is made up of a union of uniformly bounded diameter sets (which is due to Equation \eqref{UniformFiniteDiameter}) that are all some uniform distance from the collection of the projections of finitely many extensions of multicurves of $\mathcal{M}$ (which is due to Equation \eqref{FiniteDistanceToMulticurves}). Altogether this gives a uniform bound on the diameter of the projection of $\mathcal{T}_{R+1}(\mathcal{A})$ to $C(S)$, only depending on $\mathcal{M}$, $\mathcal{A}$, and $R$. 
\end{proof}
We remark here that we could also have started with a collection of curves with uniformly finite diameter projection to all subsurfaces. We phrased it in terms of only a finite collection because that is all we need. 

\begin{lemma}\label{InductiveProjBound2}
Let $G$ be any subgroup of $\text{MCG}(\Sigma)$ generated by a finite set $X$ relative to a collection of twist subgroups $\{H_1,\ldots, H_n\}$. Fix a finite set of curves $\mathcal{A}$ and an essential subsurface $S\subset \Sigma$. If $S$ is an annulus then assume it is not a neighborhood of a component of any of the associated multicurves of a $G$-coset of an element of $\{H_1,\ldots, H_n\}$. Then for every $R\geq 0$ there is a $K=K(R, G)$ independent of $S$ such that
$$\textup{diam}\bigg(\bigcup_{g\in \widehat{B}(e, R)}\pi_{S}(g\mathcal{A})\bigg) \leq K.$$
\end{lemma}
\begin{proof}
 We may assume by adding coset representatives of finite index multitwist subgroups of $H_i's$ to $X$ that each $H_i$ is a multitwist group. Using this generating set, the $\widehat{G}$-balls change, but given any $R$ there is an $R'$ so that the ball of radius $R$ in the original generating set is contained in the ball of radius $R'$ in the new generating set. 
\par 
Now, any $g\in \widehat{B}(e,R)$ (in the new generating set) can be written as $$g=g_1\cdots g_m f$$ with $m\leq R$. Here $f$ is a group element in at most $R$ letters of $X$, and $g_i$ is a multitwist on a multicurve that is the image of a multicurve of some $H_j$ by a word of length at most $R$ in the set $X$. We can write every such $g$ like this because we can move letters in $X$ past multitwists via conjugation. For example, we may write $fh$ as $h'f$, where $h$ is a multitwist on $B$ and $h'=fhf^{-1}$ is a multitwist on $f(B)$.
\par 
For every $R$, we obtain a finite collection $\mathcal{M}_R$ of multicurves coming from the images of the multicurves of $\{H_1,\ldots H_n\}$ by elements generated by $X$ of length at most $R$. Then in the notation of Lemma \ref{InductiveProjBound}, the collection we are projecting is a subset of $\mathcal{T}_R(\mathcal{A}_R)$, where $\mathcal{A}_R$ is the finite collection of curves that are images of elements of $\mathcal{A}$ by words in $X$ of length at most $R$, and $\mathcal{T}_R$ is as in Lemma \ref{InductiveProjBound}, on the collection $\mathcal{M}_R$. The assumption that the subsurface $S$ is not a neighborhood of any component of any of the associated multicurves of a $G$-coset of an element of $\{H_1,\ldots, H_n\}$ ensures that $S$ lies in the collection $\mathcal{S}_R$ as defined in the statement of Lemma \ref{InductiveProjBound}. This follows as by construction, for all $R$, every curve which appears as a component of a multicurve in $\mathcal{M}_R$ must lie a multicurve associated to a $G$ coset of an element of $\{H_1,\ldots, H_n\}$. Lemma \ref{InductiveProjBound} then gives the desired result (note that for each $R$ we are using a different collection of initial curves and a different collection of multicurves). 
\end{proof}

We now use Lemma \ref{InductiveProjBound2} to get a bound on the projection of all the multicurves associated to a PGF group $G$. 

\begin{proposition}\label{PGFProjBound}
Let $G$ be a PGF group on a closed surface $\Sigma$. Fix any proper essential subsurface $S\subset \Sigma$ that is not an annulus with core curve a component of a multicurve associated to $G$. For any nonempty finite set of curves $\mathcal{A}$, $\pi_S(G\cdot \mathcal{A})\neq \varnothing$ and
$$\mathrm{diam}(\pi_S(G\cdot \mathcal{A})) \leq K$$
for $K=K(\mathcal{A},G)$, which is independent of $S$.
\end{proposition}

\begin{proof}
We first note that $\pi_S(G\cdot \mathcal{A})$ is nonempty because PGF groups are not twist groups, so there are two distinct conjugates of twist subgroups in $G$. Corollary \ref{InfOrderpA} then shows that $G$ has a pseudo-Anosov element, implying nonempty projection to any subsurface.
\par 
Fix $\delta$ so that $C(\Sigma)$ is a $\delta$-hyperbolic space, and choose an equivariant $(\lambda, \lambda)$-quasi-isometric embedding $\widehat{G}\to C(\Sigma)$ which restricts to the orbit map of $G$ on some element $\alpha$ of $\mathcal{A}$, for some $\lambda\geq 1$. Let $S$ be a subsurface as in the statement. Suppose first that for all $g, h \in G$ there is a geodesic between elements of $g\mathcal{A}$ and $h\mathcal{A}$ so that every curve on this geodesic has nonempty projection to $S$. For such $S$, by the contrapositive of Proposition \ref{MMBGI} and Lemma \ref{ProjLip}, there is a uniform bound on $d_S(g\mathcal{A}, h\mathcal{A})$ for all choices of $g$ and $h$.
\par 
Otherwise, suppose that there is some pair $g$ and $h$ in $G$ so that every geodesic between components of $g\mathcal{A}$ and $h\mathcal{A}$ contains a curve with empty projection to $S$. To deal with this case, we apply Lemma \ref{InductiveProjBound2} along with an argument using Proposition \ref{Morse} and the contrapositive of Proposition \ref{MMBGI} to obtain a bound on the diameter of the collection of projections of all the $G$ images of $\mathcal{A}$. 
\par 
By Proposition \ref{Morse} applied to $C(\Sigma)$ and the assumptions on $g, h$, and $S$, there is a constant $N$ independent of $g$ and $h$ so that the image of any geodesic $[g,h]$ in $\widehat{G}$ to $C(\Sigma)$ under the orbit map on $\alpha$ contains a point within $N+1$ of every component of $\partial S$. Fix one such geodesic, denoted again by $[g,h]$, and pick a point $p$ within $N+1$ of $\partial S$ on the image of $[g,h]$. Fix a vertex $k\in G$ within $\frac{1}{2}$ of a preimage of $p$ in $\widehat{G}$. Then as the orbit map on $\alpha$ is a $(\lambda, \lambda)$-quasi-isometric embedding, it follows that $k\alpha$ is within $\lambda/2+\lambda=3\lambda/2$ from $p$. By the triangle inequality it follows that every component of $\partial S$ is within $N+1+3\lambda/2$ of $k\alpha$. 
\par 
By Lemma \ref{InductiveProjBound2}, for any $R$ the projection of $\widehat{B}(e,R)\cdot A$ to $k^{-1}S$ is bounded independently of $S$. We may take $R=R(\lambda, \delta)$ large enough so that for any $g_1\notin \widehat{B}(e,R)$, there is a $g_2\in \widehat{B}(e,R)$ so that every geodesic between $g_1\alpha$ and $g_2\alpha$ doesn't intersect the $1$ neighborhood of $k^{-1}\partial S$. To see this, choose $R$ so that there is a $g_2\in \widehat{B}(e,R)$ with $[g_1,g_2]$ not intersecting a large neighborhood of the identity. If this neighborhood is large enough, then the image of $[g_1,g_2]$ avoids the $N+1$-neighborhood of $k^{-1}\partial S$, as this multicurve is within $N+1+3\lambda/2$ from $\alpha$, the image of the identity. In particular, by the definition of $N$, any geodesic from $g_1\alpha$ to $g_2\alpha$ does not intersect the $1$-neighborhood of $k^{-1}\partial S$. It follows by Proposition \ref{MMBGI} and Lemma \ref{ProjLip} that 
\begin{equation}\label{FarOutBounds}
\pi_{k^{-1}S}(g_1\mathcal{A}, g_2\mathcal{A}) \leq M+D(\mathcal{A})
\end{equation}
\par 
where $D(\mathcal{A})$ is a constant that only depends on the maximum intersection number between elements of $\mathcal{A}$. Combining the bound given by Lemma \ref{InductiveProjBound2} for multicurves $g\mathcal{A}$ with $g\in \widehat{B}(e, R)$ and the bound given in inequality \ref{FarOutBounds} for multicurves $g\mathcal{A}$ with $g\notin \widehat{B}(e,R)$ using the triangle inequality, we obtain a bound on the projection of $G\cdot \mathcal{A}$ to $k^{-1}S$, which does not depend on $S$. By equivariance (that is, multiplying by $k$, which is in $G$), this gives a bound on the projection to $S$ as well, which finishes the proof. 
\end{proof}

We note the following corollary of Proposition \ref{PGFProjBound} which we utilize in Section \ref{SectionExamplesApplications}.
\begin{corollary}\label{PGFNotQuasiDense}
    Let $G$ be a PGF group relative to $\mathcal{H}$. Then the collection of all the multicurves associated to $G$ is not quasi-dense in $C(\Sigma)$.
\end{corollary}
\begin{proof}
    Suppose for a contradiction that every vertex of $C(\Sigma)$ is within $R$ of a $G$ translate of some multicurve associated to an element of $\mathcal{H}$. Let $\mathcal{A}$ denote the union of the multicurves associated to the twist groups in the finite set $\mathcal{H}$. Fix a nonannular connected essential proper subsurface $S$ of $\Sigma$, and take two curves $\alpha$ and $\beta$ which are both at least distance $R+2$ from the components of $\partial S$ and so that
    $$d_{S}(\alpha, \beta)\geq K+2M+1$$
    where $K=K(\mathcal{A}, G)$ is as in Proposition \ref{PGFProjBound} and $M$ is as in Proposition \ref{MMBGI}. Such a choice is possible as one can first pick two curves at least $R+2$ from the components of $\partial S$, and then modify one of them via applying a partial pseudo-Anosov supported in the complement of $\partial S$ to it sufficiently many times to produce the curves $\alpha$ and $\beta$ with the required distance between their projections to $S$. By assumption, there exist $\alpha’,\beta’ \in G \cdot \mathcal A$ within distance $R$ of $\alpha,\beta$, respectively.  All geodesics from $\alpha$ to $\alpha’$ and from $\beta$ to $\beta’$ are therefore entirely outside the $1$-neighborhood of $\partial S$.  Thus, by Proposition 2.12, $d_S(\alpha,\alpha’),d_S(\beta,\beta’) \geq M$, By the triangle inequality and Proposition \ref{MMBGI} it follows that
    $$d_{S}(\alpha', \beta')\geq d_S(\alpha, \beta)-d_S(\alpha, \alpha')-d_S(\beta, \beta')\geq K+1$$
    which is a contradiction by Proposition \ref{PGFProjBound}.
\end{proof}
It would be convenient in Section \ref{SectionExamplesApplications} if this corollary could be extended to the union of the collections of multicurves of finitely many PGF groups, but we leave it at this.
\par 
We will use the following lemma often. It gives a strong relationship between pairs of multicurves associated to a PGF group.
\begin{lemma}\cite[Lemma 5.3]{L21}\label{FillingMC}
Let $G$ be a PGF group relative to $\{H_1,\ldots, H_n\}$, and suppose $A_1$ and $A_2$ are distinct multicurves associated to $G$. Then $A_1$ and $A_2$ fill $\Sigma$, and they share no components in common. 
\end{lemma}

\subsection{Large projections and admissible sequences}\label{SubsectionLargeProjectionsLtoGProperty}
We make the following definition to provide a general framework for the proofs of the results of Section \ref{SectionProofMainThemOtherCases}.
\begin{definition} \label{AdmissibleSequence}
    Fix $L\geq 0$ and $n\geq 1$. A sequence of multicurves $A_0, B_1, B_2, \ldots, B_n, A_n$ in $C(\Sigma)$ is called $L$-\textit{admissible}, or just \textit{admissible}, if 
    \begin{enumerate}[label=(\alph*)]
    \item The multicurve $B_i$ is sparse for all $i=1,\ldots n$.
    \item The multicurves $A_0, A_n$ share no components with $B_1, B_n$, respectively. Also, every component of $\Sigma \setminus B_1$ intersects at least one component of $A_0$, and similarly every component of $\Sigma \setminus B_n$ intersects at least one component of $A_n$.
    \item For $1<i<n$, the multicurve $B_i$ shares no components with $B_{i-1}$ and $B_{i+1}$, and every component of $\Sigma \setminus B_i$ intersects at least one component of $B_{i-1}$ and $B_{i+1}$.
    \item If $n=1$, then for all components $S$ of $\Sigma\setminus B_1$, 
    $$d_S(A_{0}, A_1) \geq L$$
    If $n\geq 2$, then for all components $S$ of $\Sigma \setminus B_1$,
    $$d_S(A_{0}, B_{2}) \geq L$$
    and similarly for all components $S$ of $\Sigma \setminus B_n$,
    $$d_S(B_{n-1}, A_{n}) \geq L.$$
    If $n\geq 3$, then we further assume for $i=2,\ldots, n-1$ that for all components $S$ of $\Sigma \setminus B_i$,
    $$d_S(B_{i-1}, B_{i+1}) \geq L.$$
\end{enumerate}
\end{definition}
\par 
In particular, Lemma \ref{FillingMC} shows that sequences of multicurves following a path in the Bass--Serre tree of a normalized PGF graph of groups with $L$-local large projections satisfies these properties. That is, we have the following lemma. 
\begin{lemma} \label{AdmissiblePGFSequence}
    Suppose $(\mathcal{G}, \phi)$ has $L$-local large projections, and let $T$ denote the Bass--Serre tree of $\mathcal{G}$. Take a sequence of PGF vertices $v_0, \ldots, v_n$ along the geodesic $[v_0, v_n]$ in $T$ so that $[v_i, v_{i+1}]$ contains no other PGF vertices other than its endpoints, for $i=0,\ldots, n-1$. Let $G_0$ and $G_n$ denote the stabilizers of $v_0$ and $v_n$ respectively. Let $B_{i}$ be the multicurve associated to the $\phi$ image of the edge group of the edge before $v_i$ on $[v_0, v_n]$, for $i=1, \ldots, n$, and let $A_0\neq B_1$ and $A_n\neq B_n$ denote multicurves associated to $\phi(G_0)$ and $\phi(G_n)$, respectively. Then $A_0, B_1, \ldots, B_n, A_n$ is an $L$-admissible sequence of multicurves.
\end{lemma}
\begin{proof}
    By the definition of normalized PGF graphs of groups, each $B_i$ is sparse, so (a) is satisfied. As the pair $B_{i-1}$ and $B_i$, as well as the two pairs $A_0$ and $B_1$, $B_n$ and $A_n$, are distinct multicurves in the same PGF group (a different group for each pair), this sequence satisfies properties (b) and (c) by Lemma \ref{FillingMC}. Condition (d) is satisfied by the assumption of $L$-local large projections. See Figure \ref{FigureLtoGPic} for a schematic of this case.
\end{proof}

\begin{figure}[H]
    \centering
    \includegraphics[scale=.70]{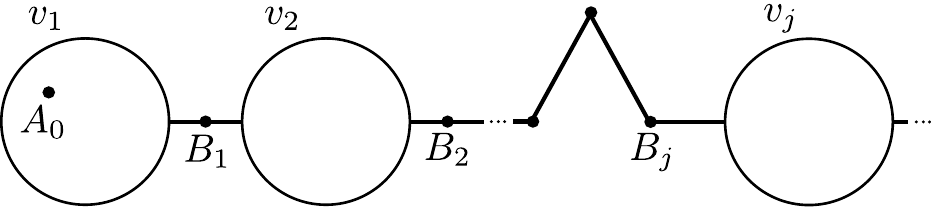}
    \caption{A schematic for the multicurves in Lemma \ref{AdmissiblePGFSequence}. Here we have blown up the PGF vertices of $T$ and drawn the multicurves $B_i$ on the edges of the groups they correspond to. There may be places as pictured where the path runs through a vertex of an extension group as in Definition \ref{PGFGraph}(2), but the multicurve $B_j$ is associated to the base twist group, not the extension group.}
     \label{FigureLtoGPic}
\end{figure}

From now until the end of the section, we will fix an $L$-admissible sequence of multicurves $A_0, B_1, \ldots, B_n, A_n$ in $C(\Sigma)$, with $L$ chosen sufficiently large as needed. Note that any subsequence of neighboring terms is also an $L$-admissible sequence. That is, $B_i, B_{i+1},\ldots, B_{j-1}, B_j$ is $L$-admissible for $i \leq j$. If $i=1$ we can also add $A_0$ to the start and obtain an $L$-admissible sequence, and if $j=n$ then a similar claim holds for $A_n$.

The next lemma is one of the main tools in proving the combination theorems in this paper. Applying it with Proposition \ref{MMBGI} gives a very strong control over the geometry of the images of the relevant graphs in $C(\Sigma)$ and this is ultimately what allows us to show that these graphs actually quasi-isometrically embed.

\begin{lemma}\label{LocalToGlobal}
    Suppose $L\geq 18$. For all components $S$ of $\Sigma \setminus B_{i}$, $i=1, \ldots n$, we have
$$d_{S}(A_0, A_{n})\geq L-8.$$
\end{lemma}
\begin{proof}
    We prove this by induction on the number of terms in the admissible sequence (using the fact that a subsequence of neighboring terms in an admissible sequence is admissible). The base case follows immediately from the definition of $L$-admissible sequences. Namely, we have for all components $S$ of $\Sigma \setminus B_n$,
    $$d_{S}(B_{n-1}, A_n) \geq L>L-8$$
    and for all components $S$ of $\Sigma \setminus B_1$,
    $$d_{S}(A_0, B_2) \geq L>L-8.$$
    
    For the inductive step, we may assume that for $1\leq i < j \leq n$ with and any component $S_j$ of $\Sigma\setminus B_j$, 
    $$d_{S_j}(B_i, A_n) \geq L-8$$
    and for $1\leq j < k \leq n$, 
    $$d_{S_j}(A_0, B_k) \geq L-8.$$
    as both $A_0, B_1,\ldots, B_n$ and $B_1,\ldots, B_n, A_n$ are admissible sequences.
    \par 
    As $L-8\geq 10$, Proposition \ref{Behrstock} implies for components $S_i$ and $S_k$ of $\Sigma \setminus B_i$ and $\Sigma\setminus B_k$, respectively, that
    $$d_{S_i}(B_j, A_n)\leq 4$$
    $$d_{S_k}(A_0, B_j)\leq 4.$$
    But then for all $2\leq j\leq n-1$, 
    $$d_{S_j}(A_0, A_n)\geq d_{S_j}(B_{j-1}, B_{j+1})-d_{S_j}(A_{0}, B_{j-1})-d_{S_j}(B_{j+1}, A_{n})\geq L-8.$$
     If $j=1$ then ignore $B_{j-1}$ and use $A_0$ instead, and if $j=n$ then ignore $B_{j+1}$ and use $A_n$ instead, giving an $L-4$ lower bound in both cases. 
\end{proof}

We first note the following lemma.
\begin{lemma}\label{lem:NoSharedComponentsAdmSeq}
    Suppose $L\geq 18$. Then the multicurves in the $L$-admissible sequence $A_0, B_1,\ldots, B_n, A_n$ share no common components. In particular, all components of $A_0, B_i, $ and $A_n$ have a nonempty projection to some component of $\Sigma \setminus B_j$, for $i\neq j$.
\end{lemma}
\begin{proof}
    It suffices to show that $\delta_0\in A_0$ is not contained in $B_i$ or $A_n$ for $i=1, \ldots, n$. By Definition \ref{AdmissibleSequence}(b), $\delta_0$ intersects some component of $\Sigma \setminus B_1$. Now suppose $\delta_0$ intersects some component $S$ of $\Sigma \setminus B_k$ for $k\geq 1$. Then by Lemmas \ref{LocalToGlobal} and \ref{ProjLip}, 
     $$d_{S}(\delta_0, B_{k+1})\geq L-8-2\geq 3$$
     so in particular $\delta_0$ intersects a component of $B_{k+1}$ (as multicurves project to diameter at most $2$ sets by Lemma \ref{ProjLip}), which implies it intersects a component of $\Sigma \setminus B_{k+1}$. 
\end{proof}

In the next four lemmas, we establish two important facts. First, the ordering of the sequence of multicurves $A_0, B_1, \ldots, B_n, A_n$ is coarsely respected when we pass to the image in $C(\Sigma)$. This is the content of Lemma \ref{OrderingOnMulticurves}. Second, we find a uniform coarse Lipschitz lower bounds on distance in $C(\Sigma)$ between $A_0$ and $A_n$ in terms of $n$. This is the content of Lemma \ref{ShortSegmentsLowerBound}.

\begin{lemma}\label{FarCurvesIntersect}
 Suppose $L\geq 18$. There is a number $s\leq 2(\xi(\Sigma)+1)$ so that the following holds. Suppose  $n\geq s$. Let $\delta_0$ be a component of $A_0$. Then $\delta_0$ and $A_n$ fill $\Sigma$, and in particular $\delta_0$ intersects every component of $A_n$ and $\Sigma \setminus A_n$.
\end{lemma}
\begin{proof}
    For $1\leq k \leq n$ let $\Sigma_k$ denote the surface filled by $\delta_0$ and $B_k$, without the annuli coming from components of $B_k$ disjoint from $\delta_0$. Define $\Sigma_n'$ from $\delta_0$ and $A_n$ in the same way. We want to show that $\Sigma_k\subset \Sigma_{k+1}$ and $\Sigma_n\subset \Sigma_n'$, and as long as $\Sigma_k\neq \Sigma$, then the containment $\Sigma_k \subset \Sigma_{k+2}$ is strict. It follows that eventually $\Sigma_k=\Sigma_n'=\Sigma$, so $\delta_0$ and $A_n$ fill $\Sigma$, as desired.
    \par 
    First note that $\delta_0$ must intersect some component of $B_2$. Namely, by condition (b) of Definition \ref{AdmissibleSequence}, Lemma \ref{LocalToGlobal}, and Lemma \ref{ProjLip}, there is a component $S$ of $\Sigma\setminus B_1$ so that
    $$d_S(\delta_0,B_2) \geq d_S(A_0, B_2)-d_S(A_0, \delta_0)\geq L-8-2\geq 3$$
    so $\delta_0$ and $B_2$ intersect. Thus $\Sigma_2$ is a not an annulus, and instead contains at least one component of $B_2$. We will thus assume $k\geq 2$ (we haven't shown that $\Sigma_1 \subset \Sigma_2$, but this will follow).
\par 
We show first that $\Sigma_{k+1}$ contains $\Sigma_k$. The proof that $\Sigma_n \subset \Sigma_n'$ is similar. Let $\gamma$ be any curve intersecting $\Sigma_k$. To show that $\Sigma_k$ is contained  in $\Sigma_{k+1}$, it will suffice to show that $\gamma$ also intersects $\Sigma_{k+1}$. If $\gamma$ and $\delta_0$ intersect, this is obvious. Otherwise, by the definition of $\Sigma_k$ there is a component $S$ of $\Sigma\setminus B_k$ so that $\delta_0$ and $\gamma$ both have nonempty projection to $S$. By Lemmas \ref{LocalToGlobal} and \ref{ProjLip}, 
$$d_S(\delta_0, B_{k+1})\geq d_S(A_0, B_{k+1})-d_S(A_0, \delta_0)\geq L-8-2\geq 5.$$
 Hence $d_S(\gamma, B_{k+1})\geq d_S(\delta_0,B_{k+1})-d_S(\delta_0,\gamma)\geq 3$, so $\gamma$ and $B_{k+1}$ must intersect by Lemma \ref{ProjLip}.
\par 
Now suppose $\Sigma_k\neq \Sigma$. We consider two cases. First, we assume there a component $S$ of $\Sigma \setminus B_k$ which intersects $\delta_0$ but $S$ is not contained in $\Sigma_k$. Second, we assume that every component of $\Sigma \setminus B_k$ which $\delta_0$ intersects is contained in $\Sigma_k$.
\par 
In the first case, $S\subset \Sigma_{k+1}$, as following the same reasoning as above we see that $d_S(\delta_0, B_{k+1})\geq 5$, and thus $\delta_0$ must intersect every component of $B_{k+1}$ which enters $S$, and every essential curve in $S$ must intersect either $\delta_0$ or $B_{k+1}$. In particular, the containment $\Sigma_k \subset \Sigma_{k+1}$ is strict. 
\par 
In the second case, there must be some boundary component $\alpha \in B_k$ of a component $S'$ of $\Sigma \setminus B_k$ that also lies in $\partial \Sigma_k$. But then as above when showing that $\delta_0$ must intersect some component of $B_2$, it follows similarly that $\alpha$ must intersect some component of $B_{k+2}$. In particular, as $d_{S}(\delta_0, B_{k+2})\geq 5$, $\delta_0$ intersects the component of $B_{k+2}$ which intersects $\alpha$, and thus it follows that the containment $\Sigma_k \subset \Sigma_{k+2}$ must be strict as $\Sigma_{k+2}$ contains some of the component of $\Sigma \setminus B_k$ not equal to $S'$ which has $\alpha$ as a boundary component.
\par
The bound $s\leq 2(\xi(\Sigma)+1)$ follows as going from $\Sigma_k$ to $\Sigma_{k+2}$ strictly increases the complexity. 
\end{proof}

\begin{lemma}\label{FarCurvesFill}
    Suppose $L\geq 18$. Let $s$ be in Lemma \ref{FarCurvesIntersect}, and suppose $n \geq 2s$. Fix $\delta_0$ and $\delta_n$ components of $A_0$ and $A_n$ respectively. Then $\delta_0$ and $\delta_n$ fill $\Sigma$. In particular, every vertex in $A_0$ has distance at least $3$ from every vertex in $A_n$ in $C(\Sigma)$.
\end{lemma}
\begin{proof}
    Fix $i$ with $s\leq i \leq n-s$. By Lemma \ref{FarCurvesIntersect}, $\delta_0$ and $\delta_n$ both intersect every component of $\Sigma \setminus B_i$. The distance of their projections to each such component is at least $3$ by Lemmas \ref{LocalToGlobal} and \ref{ProjLip}. In particular, every curve $\gamma$ with nonempty projection to some component of $\Sigma\setminus B_i$ (that is, any curve not in $B_i$) must intersect either $\delta_0$ or $\delta_n$. As $\delta_0$ intersects every component of $B_i$ by Lemma \ref{FarCurvesIntersect}, the first claim follows. The second claim follows since $\delta_0$ and $\delta_n$ were chosen arbitrarily from $A_0$ and $A_n$.
\end{proof}
Lemmas \ref{FarCurvesIntersect} and \ref{FarCurvesFill} are precisely the reason why we need the $B_i$'s to be sparse (and why twist vertex groups of normalized PGF graphs of groups have sparse associated multicurves). If one doesn't have this, then there may arise cases where, regardless of how large $L$ or $n$ are chosen to be, the subsurface filled by any component of $A_0$ or $A_n$ and some $A_i$ may always be a proper subsurface, as we have no way to "see" the pairs of pants via the projection data. In the proofs in Section \ref{SectionProofMainThemOtherCases}, when looking at certain paths between multicurves, one may get "stuck" forever, so lower bounds may be impossible to produce. 
\par 
We remark here that the next two lemmas are directly inspired by Lemma 4.4 of \cite{BBKL20}. In \cite{M13}, similar methods are used.

\begin{lemma}\label{OrderingOnMulticurves}
Suppose $L\geq M+18$, where $M$ is as in Proposition \ref{MMBGI}. Let $s$ be in Lemma \ref{FarCurvesIntersect}, and suppose $n \geq 2s$. Fix $\delta_0$ and $\delta_n$ components of $A_0$ and $A_n$ respectively. Fix a geodesic $[\delta_0, \delta_n]$. Then $[\delta_0, \delta_n]$ contains a vertex with distance at most $1$ from some component of $B_i$ for all $1\leq i \leq n$. 
    \par 
    Further, for $1\leq i \leq n$, let $\alpha_i$ denote the first vertex of $[\delta_0,\delta_n]$ within $1$ from some component of $B_i$, and let $\omega_i$ denote the last vertex of $[\delta_0,\delta_n]$ within $1$ from some component of $B_i$. Then for $s\leq j\leq k \leq n-s$, we have $\alpha_j \leq \omega_{k}$. If we further have that $k-j\geq 2s+2$, then $\omega_j \leq \omega_k$. Here the ordering is in the sense of the natural ordering on $[\delta_0, \delta_n]$ with $\delta_0$ the minimal element.
\end{lemma}
\begin{proof}
    The first claim follows from Lemmas \ref{lem:NoSharedComponentsAdmSeq} and \ref{FarCurvesIntersect}. Namely, $\delta_0$ and $\delta_n$ must simultaneously intersect some component $S$ of $\Sigma \setminus B_i$. For if $i\geq s$, then $\delta_0$ intersects every component by Lemma \ref{FarCurvesIntersect}, and $\delta_n$ intersects at least one component by Lemma \ref{lem:NoSharedComponentsAdmSeq}. If $i\leq n-s$ we may flip $\delta_0$ and $\delta_n$ in the above argument. Thus by Lemma \ref{LocalToGlobal} and \ref{ProjLip}, for any $1\leq i \leq n$ and any component $S$ of $\Sigma \setminus B_i$ that both $\delta_0$ and $\delta_n$ intersect,
     $$d_S(\delta_0, \delta_n)\geq L-8-2-2\geq M$$
     so Proposition \ref{MMBGI} gives the desired vertex of $[\delta_0, \delta_n]$ for every $i$. In particular the $\alpha_i$'s and $\omega_i$'s are well defined.
    \par 
    We begin with showing that $\alpha_j \leq \omega_k$ for all $s\leq j \leq k \leq n-s$. This is immediate when $j=k$. Assume for contradiction that there for some $j$ and $k$ with $s\leq j < k \leq n-s$ so that $\omega_{k}< \alpha_j$. First, note that $\omega_{k}$ is not a component of $B_j$. If it were, then there would be some component of $\Sigma \setminus B_{k}$ so that $\omega_{k}$ and $\delta_n$ have projections to this component which are at least $M$ apart, by Lemmas \ref{LocalToGlobal} and \ref{ProjLip}. By Proposition \ref{MMBGI}, this would give a vertex other than $\omega_{k}$ on $[\omega_{k}, \delta_n] \subset [\delta_0, \delta_n]$ disjoint from a component of $B_{k}$, contradicting the definition of $\omega_{k}$.
    \par 
    We then have the following inequality for any component $S$ of $\Sigma \setminus B_j$ that $\omega_{k}$ has nonempty projection to. Note that $\delta_0$ has nonempty intersection to this component by Lemma \ref{FarCurvesIntersect}.
    $$d_{S}(\delta_0, B_{k}) \leq d_{S}(\delta_0, \omega_{k})+d_{S}(\omega_{k}, B_{k}).$$
    %
\begin{figure}[H]
    \centering
    \includegraphics[scale=.85]{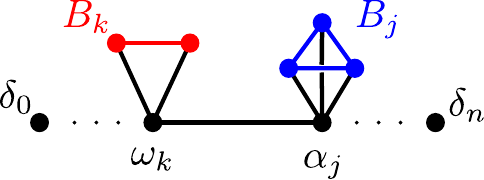}
    \caption{The hypothetical scenario discussed in the above proof. Lemma \ref{LocalToGlobal} says that $B_{k}$ and $\delta_0$ have far apart projections to every component of $\Sigma\setminus B_{j}$, which by Proposition \ref{MMBGI} and the definition of the $\omega$ curves shows that $B_{k}$ cannot be to the "left" of $B_j$, giving the contradiction.}
     \label{FigureOmegaExample}
\end{figure}
     The left hand side at least $L-8-2\geq M+5$ by Lemmas \ref{LocalToGlobal} and \ref{ProjLip}. On the other hand,  $d_{S}(\omega_{k}, B_{k})\leq 4$ as $\omega_{k}$ is disjoint from some component of $B_{k}$, and since $[\delta_0, \omega_{k}]\subset [\delta_0,\delta_n]$ contains no vertices disjoint from $S$ by the definition of $\alpha_{j}$ and the assumption that $\omega_{k}< \alpha_j$, $d_S(\delta_0, \omega_{k})\leq M$ by the contrapositive of Proposition \ref{MMBGI}. This is a contradiction.
     \par 
     Now suppose $\omega_j >\omega_k$ with $s \leq j < k \leq n-s$ and $k-j\geq 2s+2$. The sequence $B_j, B_{j+1}, \ldots, B_k$ is an $L$-admissible sequence, with at least $s$ terms between $B_j$ and $B_k$. By Lemma \ref{FarCurvesFill} and the triangle inequality, $\omega_j$ is at least distance $2$ in $C(\Sigma)$ from every component of $B_k$ as every component of $B_j$ is at least distance $3$ from every component of $B_k$, so for any component $S$ of $\Sigma \setminus B_k$,
     $$d_S(B_j, \delta_n)\leq d_S(B_j, \omega_j)+d_S(\omega_j, \delta_n).$$
     The left hand side is at least $L-10\geq M+5$ by Lemmas \ref{LocalToGlobal} and \ref{ProjLip}, while the first term on the right is at most $4$ and the second is at most $M$ as $[\omega_j, \delta_n]$ contains no vertices disjoint from $S$ by the definition of $\omega_k$. This again is a contradiction, so $\omega_j\leq \omega_k$.
\end{proof}

\begin{lemma}\label{ShortSegmentsLowerBound}
     Suppose $L\geq M+18$, where $M$ is as in Proposition \ref{MMBGI}. Let $s$ be as in Lemma \ref{FarCurvesIntersect}. There is a constant $E$ so that the following holds. For all components $\delta_0$ of $A_0$ and $\delta_n$ of $A_n$
$$d_{\Sigma}(\delta_0,\delta_n)\geq \frac{1}{E}n-E.$$
\end{lemma}
\begin{proof}
     If $n\leq 2s+2$, then clearly such an $E$ can be found, for example $E= 2s+2$ suffices. Thus we may assume $n\geq 2s+2$. Fix a geodesic $[\delta_0, \delta_n]$. We may split $[\delta_0, \delta_n]$ into subsegments intersecting only at their endpoints as follows.
     $$[\delta_0, \delta_n]=[\delta_0, \omega_{2s+2}]\cup [\omega_{2s+2}, \omega_{4s+4}]\cup \cdots \cup [\omega_{2s(m-1)}, \delta_n].$$
     Here, $m$ is the integer part of $\frac{n}{2s+2}$. Since $2k(s+1)-2(k-1)(s+1)=2s+2>s$, the claim about these subsegments intersecting only at their endpoints follows from Lemma \ref{OrderingOnMulticurves}.
     \par 
     By applying Lemma \ref{FarCurvesFill}, the triangle inequality, and the fact that the sequence $B_{(2s+2)k}, B_{(2s+2)k+1}, \dots B_{(2s+2)(k+1)}$ is $L$-admissible with $2s$ terms between the first and last multicurve, every segment except for perhaps the last has length at least $1$. This follows as $\omega_{(2s+2)k}$, $\omega_{(2s+2)(k+1)}$ are distance $1$ from some element of $B_{(2s+2)k}$, $B_{(2s+2)(k+1)}$, respectively, and each element of these two multicurves are at least $3$ from each other. In particular,
     $$d_{\Sigma}(\delta_0, \delta_n) \geq m-1.$$
     But we have
     $$m\geq \frac{1}{2s+2}n-(2s+2)$$
     giving the desired $E$.
\end{proof}

\section{Combination Theorems}\label{SectionProofMainThemOtherCases}

In the following lemma, we provide the underlying method for proving the combination theorems, including Theorem \ref{ComboQIEmb}. It is a technical result which gives a way to combine the QI lower bounds coming from the vertex groups of a graph of groups, and the lower bounds that result from Lemma \ref{ShortSegmentsLowerBound}.

\begin{lemma}\label{lem:admSequenceQIlowerbound}
    Fix $C\geq 1$ and an $L$-admissible sequence of multicurves $A_0, B_1, \ldots, B_n, A_n$ with $L\geq M+18$. Then there is a constant $K=K(C, E)$, where $E$ is the constant from Lemma \ref{ShortSegmentsLowerBound}, so that the following holds. Suppose $a_0, b_1, \ldots, b_n, a_n$ is a sequence of positive integers so that for any $\gamma_0\in A_0, \delta_i\in B_i, \gamma_n \in A_n$, with $i=1,\ldots, n-1$, we have for all $i$ that

    \begin{equation}\label{eq:AdmissNeighborLowerBounds1}
        d_{\Sigma}(\gamma_0,\delta_1)\geq \frac{1}{C}a_0-C 
    \end{equation}

    \begin{equation}\label{eq:AdmissNeighborLowerBounds2}
        d_{\Sigma}(\delta_i,\delta_{i+1})\geq \frac{1}{C}b_i-C 
    \end{equation}

    \begin{equation}\label{eq:AdmissNeighborLowerBounds3}
        d_{\Sigma}(\delta_n,\gamma_n)\geq \frac{1}{C}a_n-C. 
    \end{equation}

    Then we can conclude that
    $$d_{\Sigma}(\gamma_0, \gamma_n)\geq \frac{a_0+a_n+\sum_{i=1}^{n-1} b_i}{K}-K. $$
\end{lemma}
We remark before the proof that the constants $a_0, b_1, \ldots, b_n, a_n$ will in practice come from distances in the coned off Cayley graph of the fundamental group of the graph of groups in the theorems below.
\begin{proof}
    Assume first that $n\leq 2s$. Then $[\gamma_0, \gamma_n]$ contains a vertex $\beta_1$ at most $2$ from $\delta_1$. Indeed, if $\gamma_0$ is disjoint from some component of $B_1$, then $\beta_1=\gamma_0$ itself is such a vertex. Otherwise, by Lemma \ref{lem:NoSharedComponentsAdmSeq} there is some component $S$ of $\Sigma \setminus B_1$ that $\gamma_n$ intersects, and hence by Lemmas \ref{LocalToGlobal} and \ref{ProjLip} we have that $d_S(\gamma_0, \gamma_n)\geq M$, and the existence of the required vertex $\beta_1$ follows from Proposition \ref{MMBGI} (in fact, in this case we can choose $\beta_1$ to be at distance $1$ from $\delta_1$). In either case, we have
    $$d_{\Sigma}(\gamma_0, \gamma_n)=d_{\Sigma}(\gamma_0, \beta_1)+d_{\Sigma}(\beta_1, \gamma_n)\geq d_{\Sigma}(\gamma_0, \delta_1)+d_{\Sigma}(\delta_1, \gamma_n)-2.$$
    One can continue this inductively (using that for all $1\leq k\leq n-1$, $B_k, B_{k+1}, \ldots, A_n$ is an $L$-admissible sequence) to show that
    \begin{equation*}
        d_{\Sigma}(\gamma_0, \gamma_n)\geq d_{\Sigma}(\gamma_0, \delta_1)+d_{\Sigma}(\delta_n, \gamma_n)+\sum_{i=1}^{n-1}\bigg[d_{\Sigma}(\delta_i, \delta_{i+1})\bigg]-4s
    \end{equation*}
    where the $4s$ term appears as $n\leq 2s.$ Using the inequalities \ref{eq:AdmissNeighborLowerBounds1}, \ref{eq:AdmissNeighborLowerBounds2}, and \ref{eq:AdmissNeighborLowerBounds3}, this becomes
    \begin{equation*}
        d_{\Sigma}(\gamma_0, \gamma_n)\geq \frac{a_0+a_n+\sum_{i=1}^{n-1} b_i}{C}-2sC-4s. 
    \end{equation*}
    so in this case letting $K=2sC+4s$ suffices.
    \par 
    Otherwise, if $n\geq 2s$, we consider the vertices $\alpha_i$ and $\omega_i$ for the geodesic $[\gamma_0, \gamma_n]$ from Lemma \ref{OrderingOnMulticurves}. By Lemma \ref{OrderingOnMulticurves}, $\alpha_s\leq \omega_{n-s}$, so we can write
    \begin{equation*}
        [\gamma_0, \gamma_n]=[\gamma_0, \alpha_s]\cup [\alpha_s, \omega_{n-s}]\cup [\omega_{n-s}, \gamma_n]
    \end{equation*}
    where the geodesics on the right only intersect at endpoints. We obtain coarse Lipschitz lower bounds for the first and last geodesic segment in this equality using the case when $n\leq 2s$ above. Thus it will suffice to focus on the middle term.
    \par 
    For $s\leq i \leq n-s-1$, each $\alpha_i$ and $\omega_i$ has distance at most $2$ from $\delta_i$. Thus
    \begin{equation}\label{eq:alphaomegaLowerBound}
        d_{\Sigma}(\alpha_i, \omega_{i+1})\geq d_{\Sigma}(\delta_i, \delta_{i+1})-4\geq \frac{1}{C}b_i-C-4.
    \end{equation}
    For $s\leq i \leq n-s-1$ and a fixed $D>0$, we will call a geodesic $[\alpha_i, \omega_{k+1}]$ $D$-\textit{long} if 
    \begin{equation*}
        \frac{1}{C}b_i-D>0
    \end{equation*}
    If a segment is not $D$-long it will be called $D$-\textit{short}. We remark that when the $b_i$'s are given by distances in a coned off Cayley graph, this terminology makes more sense as one can obtain bounds on $d_{\Sigma}(\alpha_i, \omega_{i+1})$ in terms of $D$. 
    \par 
    We then split $[\alpha_s, \omega_{n-s}]$ into subsegments that are either $D$-long or maximal connected unions of $D$-short segments, where $D>C+2E+21$ with $E$ the constant in Lemma \ref{ShortSegmentsLowerBound}. Specifically, there is some $s\leq m\leq n-s$ and a strictly increasing function $N:\{s,\ldots, m\} \to \{s,\ldots, n-s\}$ with $N(s)=s$ and $N(m)=n-s$ so that we can write
    $$[\alpha_s, \omega_{n-s}]=[\alpha_{N(s)}, \omega_{N(s+1)}]\cup [\alpha_{N(s+1)}, \omega_{N(s+2)}]\cup \cdots \cup [\alpha_{N(m-1)}, \omega_{N(m)}]$$
    where $[\alpha_{N(j)}, \omega_{N(j+1)}]$ is either a $D$-long segment or a maximal union of neighboring $D$-short segments (i.e. either neighboring segment of this union is $D$-long). Here, neighboring subsegments are not typically disjoint, and may overlap on a subsegment of length at most $3$ as $[\alpha_k, \omega_k]$ has length at most $3$ for all $1\leq k \leq n$. 
    \par
     For all $D$-long segments $[\alpha_i, \omega_{i+1}]$ there is a $C'=C'(C, D)$ so that
     \begin{equation*}
         \frac{1}{C}b_i-C-2E-21\geq \frac{1}{C'}b_i.
     \end{equation*}
    This is because there is a uniform positive lower bound for all $i$ on the left hand side as each $b_i$ is a positive integer. On the other hand, for any union of neighboring $D$-short segments $[\alpha_{N(i)}, \omega_{N(i+1)}]$, we have
    \begin{equation*}
        \sum_{k=N(i)}^{N(i+1)-1} b_i \leq CD(N(i+1)-N(i))
    \end{equation*}
    by the definition of $D$-short. In particular, using Lemma \ref{ShortSegmentsLowerBound} and the triangle inequality, we obtain
    \begin{multline}\label{eq:alphaomegaShortSegmentLowerb}
        d_{\Sigma}(\alpha_{N(i)}, \omega_{N(i+1)})\geq d_{\Sigma}(\delta_{N(i)}, \delta_{N(i+1)})-4
        \\ \geq \frac{1}{E}(N(i+1)-N(i))-E-4\geq \frac{1}{CDE}\biggl(\sum_{k=N(i)}^{N(i+1)-1}b_i\biggl)-E-4.
    \end{multline}
    \par 
    If there are no long segments, inequality \ref{eq:alphaomegaShortSegmentLowerb} gives the desired lower bound for the $[\alpha_s, \omega_{n-s}]$ segment. Otherwise, let $J_{\ell}$ denote the indices of $\{s, \ldots, m\}$ so that $j\in J_{\ell}$ implies $[\alpha_{N(j)}, \omega_{N(j+1)}]$ is $D$-long, and $J_{s}$ the indices so that $j\in J_{s}$ implies $[\alpha_{N(j)}, \omega_{N(j+1)}]$ is a maximal union of neighboring short segments. Define $C''=\max \{C', CDE\}$. Combining all of this together, we obtain the following.
    \begin{multline*}
        d_{\Sigma}(\alpha_s, \omega_{n-s}) \geq \sum_{j\in J_{\ell}} \big(d_{\Sigma}(\alpha_{N(j)}, \omega_{N(j+1)})-3 \big ) +\sum_{j\in J_s}\big(d_{\Sigma}(\alpha_{N(j)}, \omega_{N(j+1)})-3\big)
        \\
        \geq \sum_{j\in J_{\ell}}\biggl(\frac{1}{C}b_{N(j)}-C-7\biggl) + \sum_{j\in J_s} \biggl( \frac{1}{CDE}\Bigl{(} \sum_{k=N(j)}^{N(j+1)-1} b_k\Bigl{)}-E-7\biggl)
        \\
        \geq \sum_{j\in J_{\ell}}\biggl(\frac{1}{C}b_{N(j)}-C-2E-21\biggl) +\sum_{j\in J_{s}} \frac{1}{CDE}\sum_{k=N(j)}^{N(j+1)-1}b_k\geq \frac{1}{C''}\sum_{i=s}^{n-s-1} b_i
    \end{multline*}
    The first inequality follows by breaking $[\alpha_s, \omega_{n-s}]$ into its $D$-long and unions of $D$-short segments, accounting for the overlaps of neighboring segments. The second results from applying inequalities \ref{eq:alphaomegaLowerBound} and \ref{eq:alphaomegaShortSegmentLowerb}. The third inequality is found by moving the additive constants from the outer second sum into the terms of the first, noting that the constants in the first will only increase by at most $2E+14$. This is because there is at most twice as many unions of $D$-short segments than there are $D$-long segments (typically there would only be at most one more). Finally, the last follows by applying the definition of $C'$ to the first sum, which replaces each term by $\frac{1}{C'}b_{N(j)}$, and then applying the definition of $C''$ and combining all the terms into one sum.
    \par 
    Letting $K=\text{max}\{C'', 4sC+8s+4\}$, we can then combine this with the $n\leq 2s$ case to obtain
    \begin{align*}
        &d_{\Sigma}(\gamma_0, \gamma_n)=d_{\Sigma}(\gamma_0, \alpha_s)+d_{\Sigma}(\alpha_s, \omega_{n-s})+d_{\Sigma}(\omega_{n-s}, \gamma_n)
        \\
        &\geq d_{\Sigma}(\gamma_0, \delta_s)+d_{\Sigma}(\alpha_s, \omega_{n-s})+d_{\Sigma}(\delta_{n-s}, \gamma_n)-4
        \\
        &\geq \frac{a_0+\sum_{i=1}^{s-1}b_i}{C}-2sC-4s+\frac{\sum_{i=s}^{n-s-1} b_i}{C''}+\frac{\sum_{i=n-s}^{n}b_i+a_n}{C}-2sC-4s-4
        \\
        &\geq \frac{a_0+a_n+\sum_{i=1}^{n-1} b_i}{K}-K
    \end{align*}
    as desired.
\end{proof}

We now give a proof of Theorem \ref{ComboQIEmb}. The QI lower bound part of the proof essentially comes down to an immediate application of Lemma \ref{lem:admSequenceQIlowerbound}, after some setup.

\begin{proof}[Proof of Theorem \ref{ComboQIEmb}]
Proposition \ref{ComboRelHyp} shows that $G$ is relatively hyperbolic, relative to the desired collection of twist subgroups. Thus it suffices to show that the map $\Psi$ defined before Lemma \ref{ConePointImages2} is a quasi-isometric embedding, and that $\phi$ is injective, which will follow easily from Lemma \ref{EquivQIEmb} and will be done at the end of the proof. Once we show these are true, it follows that $\phi(G)$ is a PGF group, as the equivariant embedding of $\widehat{\phi(G)}$ in $C(\Sigma)$ is the same as that of $\widehat{G}$ by the definition of $\Psi$.
    \par 
    It suffices by Lemma \ref{ConePointImages2} to show that $\Psi$ restricted to the cone points of $\widehat{G}$ is a quasi-isometric embedding. We have shown in Lemma \ref{PGFGraphEmbLipschitz} that $\Psi$ is a coarse Lipschitz map, and by the definition of PGF groups and Lemma \ref{ConePointImages}, $\Psi$ restricted to the vertex groups of $T$ is a $(C,C)$-quasi-isometric embedding for some fixed $C\geq 1$.
    \par 
    Thus we need to find a $K\geq 1$ so that
    \begin{equation}\label{PsiCLLowerBound}
        d_{\Sigma}(\Psi(p), \Psi(p'))\geq \frac{1}{K}d_{\widehat{G}}(p, p')-K
    \end{equation}
    for all cone points $p, p'$ of $\widehat{G}$. We apply Lemma \ref{lem:admSequenceQIlowerbound} by choosing an admissible sequence in the following way. Let $P_0$ and $P_n$ be peripheral subsets of $G$ so that $p\in \nu(P_0)$ and $p' \in \nu(P_n)$ with a fixed sequence of vertices $v_0,\ldots, v_n$ along a geodesic in $T$ with each $v_i$ a PGF vertex, and so that $v_i$ and $v_{i+1}$ are separated by only twist vertices with $P_0$ in the vertex group of $v_0$ and $P_n$ in the vertex group of $v_n$.  For $1\leq i \leq n$ let $Q_i$ denote the peripheral subset associated to the edge of $[v_0,v_n]$ before $v_i$, and fix $q_i\in \nu(Q_i)$. We may assume $Q_1\neq P_0$ and $Q_n\neq P_n$. Then the corresponding sequence of multicurves given by the associated multicurves of $\phi(P_0), \phi(Q_1),\ldots, \phi(Q_n), \phi(P_n)$ is $L$-admissible by Lemma \ref{AdmissiblePGFSequence}.
    \par 
    In the notation of Lemma \ref{lem:admSequenceQIlowerbound}, we let $\gamma_0=\Psi(p), \gamma_n=\Psi(p')$ and $\delta_i=\Psi(q_i)$ for $i=1,\ldots, n$. Further, we define $a_0=d_{\widehat{G}}(p, q_1), a_n=d_{\widehat{G}}(q_n, p')$ and $b_i=d_{\widehat{G}}(q_i, q_{i+1})$ for $i=1,\ldots n-1$.
    \par 
    Then by the triangle inequality,
    \begin{equation*}
        a_0+a_n+\sum_{i=1}^{n-1}b_i\geq d_{\widehat{G}}(p, p').
    \end{equation*}
    As $\Psi$ restricted to each vertex group is a $(C,C)$ quasi-isometric embedding, the inequalities (\ref{eq:AdmissNeighborLowerBounds1})(\ref{eq:AdmissNeighborLowerBounds2})(\ref{eq:AdmissNeighborLowerBounds3}) all hold, so Lemma \ref{lem:admSequenceQIlowerbound} and the previous inequality imply that
    \begin{equation*}
        d_{\Sigma}(\Psi(p), \Psi(p'))=d_{\Sigma}(\gamma_0, \gamma_n)\geq \frac{1}{K}d_{\widehat{G}}(p, p')-K
    \end{equation*}
    which is precisely the lower bound desired.
    \par 
    The final claim about pseudo-Anosovs follows from Lemma \ref{EquivQIEmb}. The only concern for injectivity left then are finite order elements not conjugate into any twist groups. But such elements are contained in a single vertex stabilizer as finite groups acting isometrically on trees always have fixed points and there are no edge inversions. Since each vertex group injects, by Lemma \ref{EquivQIEmb} $\phi$ is an injection.
\end{proof}

\subsection{Other combination theorems}\label{SubsectionComboConvexCocompact}

In the following three results, we will see how the language of admissible sequences can be used to provide other combination theorems of PGF and other related groups. The proofs of all three theorems follow that of Theorem \ref{ComboQIEmb}. We fix a sparse multicurve $A$, and let $\mathcal{S}_A$ denote the set of components of $\Sigma\setminus A$.

\begin{theorem}\label{ConvexCocompactCombo}    
Let $G_1$ and $G_2$ be PGF subgroups of $\text{MCG}(\Sigma)$ relative to $\mathcal{H}_1$ and $\mathcal{H}_2$ respectively. Assume that for all nontrivial $g_1\in G_1$ and $g_2\in G_2$, $g_1(A)$ and $g_2(A)$ share no components with $A$, and that for all $S\in \mathcal{S}_A$,
    $$d_{S}(g_1(A), g_2(A))\geq M+18$$
    with $M$ as in Proposition \ref{MMBGI}. Then the natural homomorphism of $G_1*G_2$ to $\text{MCG}(\Sigma)$ is injective and its image is a PGF group relative to $\mathcal{H}_1\cup\mathcal{H}_2$.
    \end{theorem}

\begin{proof}
    The proof of this theorem mirrors the proof of Theorem \ref{ComboQIEmb}. Instead of using the subsurface components of the complements of multicurves associated to twist groups, we use the translates of $S\in \mathcal{S}_A$ by elements of $G_1*G_2$. We will explicitly show using Lemma \ref{lem:admSequenceQIlowerbound} that the orbit map of $\gamma_0\in A$ under the group $G_1*G_2$ with the coned off metric is a quasi-isometric embedding. Once this is done it will follow that $\widehat{G_1*G_2}$ quasi-isometrically embeds as well as any equivariant choice of extension to the cone points will still be a quasi-isometric embedding. This follows from the same sort of argument as in Lemma \ref{ConePointImages}.
    \par 
    First note that once we know that orbits maps of $G_1*G_2$ on $C(\Sigma)$ are quasi-isometric embeddings, injectivity will follow immediately by Lemma \ref{EquivQIEmb}. Thus it suffices to study the orbit map.
    \par 
    To mimic the proof of Theorem \ref{ComboQIEmb}, we need to construct admissible sequences. Fix a reduced word $g=f_1g_1f_2\cdots f_ng_n$ in $G_1* G_2$, with $f_i\in G_1$ and $g_i \in G_2$, where $f_1$ or $g_n$ are possibly trivial. We study the sequence
    $$A, f_1(A), f_1g_1(A), f_1g_1f_2(A),\ldots, f_1g_1\cdots f_n(A), g(A)$$
    where we throw out any duplicates if $f_1$ or $g_n$ are trivial. By equivariance and our assumptions about how $G_1$ and $G_2$ act on $A$, it is straightforward to see that this sequence is $(M+18)$-admissible. 
    \par 
    In particular, we can apply Lemma \ref{lem:admSequenceQIlowerbound} in a similar way as in the proof of Theorem \ref{ComboQIEmb}. We assume for notational simplicity that $f_1$ and $g_n$ are not trivial. Let $\delta_i=f_1g_1\cdots f_i(\gamma_0)$, $\delta_{i+1}=f_1g_1\cdots f_ig_i(\gamma_0)$ for $i$ an odd integer between $1$ and $2n-1$, and finally $\gamma_{2n+1}=g(\gamma_0)$. We let $a_0=|f_1|_{\widehat{G_1}}$, $a_{2n+1}=|g_n|_{\widehat{G_2}}$, $b_i=|g_i|_{\widehat{G_2}}$ and $b_{i+1}=|f_{i+1}|_{\widehat{G_1}}$ for $i$ an odd integer between $1$ and $2n-1$. Then by the triangle inequality
    \begin{equation*}
        a_0+a_{2n+1}+\sum_{i=1}^{2n}b_i\geq |g|_{\widehat{G_1*G_2}}
    \end{equation*}
    and as the embedding restricted to both $\widehat{G_1}$ and $\widehat{G_2}$ is a $(C,C)$-quasi-isometric embedding for some $C\geq 1$, inequalities \ref{eq:AdmissNeighborLowerBounds1}, \ref{eq:AdmissNeighborLowerBounds2}, and \ref{eq:AdmissNeighborLowerBounds3} hold. Thus Lemma \ref{lem:admSequenceQIlowerbound} gives the QI lower bounds.

\end{proof}

We now describe a method for adding a free factor twist group to a PGF group. 

\begin{theorem}\label{ConvexCoCptTwists}
    Fix $G$ a PGF subgroup of $MCG(\Sigma)$ relative to $\mathcal{H}$, and let $H$ be any twist group in MCG$(\Sigma)$. Assume for all nontrivial $g\in G$ and $\tau\in H$ that $g(A)$ and $\tau(A)$ share no components with $A$, and that for all $S\in \mathcal{S}_A$,  
$$d_{S}(g(A), \tau(A))\geq M+18$$
    with $M$ as in Proposition \ref{MMBGI}. Then the natural homomorphism of $G*H$ to $\text{MCG}(\Sigma)$ is injective and its image is a PGF group relative to $\mathcal{H}\cup\{H\}$.
\end{theorem}
\begin{proof}
    The proof is essentially the same as Theorem \ref{ConvexCocompactCombo}. We again obtain an admissible sequence via the translates of $A$ by elements of $G*H$ defined in the exact same way, replacing $G_1$ with $G$ and $G_2$ with $H$. The constants $a_0, b_1,\ldots, b_{2n}, a_{2n+1}$ and also defined in the same way.  Injectivity follows similarly as well.
\end{proof}

We end with a result analogous to (but distinct from) Proposition \ref{LoaCombo}. 

\begin{theorem} \label{LoaAnalog}
    Fix twist groups $H_1$ and $H_2$. Suppose that for all nontrivial $h_1\in H_1$ and $h_2\in H_2$, $h_1(A)$ and $h_2(A)$ share no components with $A$. If for all $S\in \mathcal{S}_A$, we have
    $$d_S(h_1(A), h_2(A))\geq M+18$$
    with $M$ as in Proposition \ref{MMBGI}, then the natural homomorphism of $G=H_1*H_2$ to MCG$(\Sigma)$ is injective and its image is a PGF group relative to $\{H_1, H_2\}$.
\end{theorem}
\begin{proof}
    This proof follows similarly to Theorem \ref{ConvexCocompactCombo} as well, replacing $G_1$ with $H_1$ and $G_2$ with $H_2$. In this case, $a_0=a_{2n+1}=b_1=\cdots=b_{2n}=1$, and the lower bound essentially comes down to applying Lemma \ref{ShortSegmentsLowerBound} directly.
\end{proof}

One thing of interest to note with Theorem \ref{LoaAnalog} is that it gives PGF groups that are free products of multitwist groups whose multicurves are much closer together than the $D_0$ bound given in Proposition \ref{LoaCombo}, see Corollary \ref{ComboMultitwistGroups} and discussion after it.

\section{Undistorted subgroups of MCG$(\Sigma)$}\label{SectionStatementsProofsUndistorted}

\subsection{Generalizing PGF Groups}\label{SubsectionGeneralizePGFGroups}
A finitely generated subgroup $H$ of a finitely generated group $G$ is \textit{undistorted} if the inclusion map of $H$ into $G$ induces a QI embedding for some (and hence all) choices of finite generating sets of both $H$ and $G$. 
\par 
Before discussing the proof of that PGF groups are undistorted as subgroups of mapping class groups, we introduce a wider class of subgroups (see Definition \ref{RGFGroups}). Fix a closed orientable surface $\Sigma$.

\begin{definition}\label{ReducibleConvexCocompact}
    An infinite finitely generated subgroup $H<\text{MCG}(\Sigma)$ is a \textit{pure} reducible subgroup if every element of $H$ is pure reducible, in the sense of Definition \ref{pureRed}.
    \par 
    Given a pure subgroup $H$, we define the set $\Omega_H$ to be the set of subsurfaces $R$ of $\Sigma$ so that some $h\in H$ stabilizes $R$, and $h|_R$ acts loxodromically on $C(R)$. Note that this is actually the collection of the components of the supports of elements of $H$.
    \par
    A pure reducible subgroup $H$ will be called \textit{strongly undistorted} if for all markings $\mu$ of $\Sigma$, there is a $\sigma_0\geq 0$ so that if $\sigma\geq \sigma_0$, then there is a $\kappa\geq 1$ so that for all $h \in H$,

    \begin{equation}\label{DistanceFormulaSupports}
        d_H(1, h) \approx_{\kappa} \sum_{R\in \Omega_H} \{\!\{d_R(\mu, h\mu)\}\!\}_{\sigma}
    \end{equation}
     A subgroup $H$ is a \textit{virtually strongly undistorted pure reducible} subgroup if it contains a finite index strongly undistorted pure reducible subgroup. To every virtually strongly undistorted pure reducible subgroup $H$ there is an \textit{associated subsurface} that is the complement of the maximal multicurve $A$ so that every finite index pure reducible subgroup fixes every element of $A$, along with neighborhoods of the curve components of $A$ where every finite index pure reducible subgroup has canonical representatives which restrict to nontrivial Dehn twists in a neighborhood of this curve.

\end{definition} 

Note that by results of Ivanov \cite{I96}, every pure reducible subgroup $H$ has an associated multicurve $A$ so that every element of $H$ preserves $A$ and every component of $\Sigma\setminus A$, and so that every $R\in \Omega_H$ lies in $S\setminus A$ or is a component of $A$.

\begin{remark}\label{WhyStronglyUndistorted}
    Using Proposition \ref{MMDistance} one can see that strongly undistorted pure reducible subgroups are actually undistorted. Our assumption of "strongly undistorted" (that is, assuming \eqref{DistanceFormulaSupports} holds) is more restrictive than one might hope for. The assumption on the distance formula \eqref{DistanceFormulaSupports} for strongly undistorted pure reducible subgroups is due to a certain technical issue in the proofs below. If one drops \eqref{DistanceFormulaSupports}, then a distance formula for these groups still holds via applying Proposition \ref{MMDistance}. The sum however involves more subsurfaces, namely subsurfaces that are not components of supports of elements of $H$. These subsurfaces are a nuisance when attempting to prove a more general version of Proposition \ref{FarPeripheral->FarSubsurface}, see Remark \ref{WhyStronglyUndistorted2} for more details about this. Also, the obvious analog of Lemma \ref{FillingSubsurfaces} is false, see \cite{M18} for some discussion in this direction.
    \par 
    Assuming \eqref{DistanceFormulaSupports} lets us ignore these surfaces in all arguments, making the proofs much easier. It seems likely that Theorem \ref{Undistorted} still holds without assuming \eqref{DistanceFormulaSupports}, but the proof is currently out of reach. There is another assumption we need to make about the pure reducible subgroups too. Namely, one needs to ensure that any no element of the collection of supports of one pure reducible subgroup properly contains a component of a support of another pure reducible subgroup. see Definition \ref{RGFGroups}, as well as the proof of Lemma \ref{FillingSubsurfaces} for why this extra assumption is needed. 
\end{remark}

\begin{example}
\begin{enumerate}
    \item Twist groups are examples, as a finite index multitwist subgroup restricts to some power of a Dehn twist on the components of the associated multicurve. Multitwist groups are strongly undistorted by \cite{MM00}.
    \item If $f\in \text{MCG}$ is a reducible element with reducing multicurve $A$ there is a $k\in \Z$ so that $f^k$ fixes $A$ and stabilizes every surface component of $\Sigma \setminus A$. Thus, $\<f^k\>$ is a pure reducible subgroup. Let $R$ either be a annulus with core curve in $A$, or a component of $\Sigma\setminus A$, so that in either case the action of $f^k$ on $R$ is nontrivial. Then either the action is by a power of a Dehn twist, or by a pseudo-Anosov on the subsurface. In either case, the union of all such $R$ with nontrivial action is the subsurface associated to $\<f\>$. It is well known that cyclic subgroups of mapping class groups are strongly undistorted \cite{MM00}. The distance formula \eqref{DistanceFormulaSupports} holds in this case as $\<f^k\>$ quasi-isometrically embeds into $C(S)$ for each component $S$ of $\Sigma \setminus A$ that $f^k$ acts nontrivially on, and also into the curve graphs of annuli with core curves components of $A$ that are acted nontrivially on. 
    \item More generally, given a pure reducible subgroup $H$ with associated subsurface $S\setminus A=S_1\cup \cdots \cup S_n$, if the image of every restriction homomorphism from $H$ to MCG$(S_i)$ has convex cocompact image, then $H$ is strongly undistorted. This follows directly from the definition of convex cococompact groups and Proposition \ref{MMDistance}. 
    \item Other examples of strongly undistorted pure reducible subgroups can be built from the work of \cite{CLM12} and \cite{R20}.
\end{enumerate}
\end{example}

We now define the generalization of PGF groups that we will work with. 

\begin{definition}\label{RGFGroups}
    A group $G<\text{MCG}(\Sigma)$ is a \textit{reducible geometrically finite} group, or an RGF group, if 
    \begin{enumerate}
        \item $G$ is hyperbolic relative to a finite collection of subgroups $\{H_1,\ldots, H_n\}$ that are virtually strongly undistorted pure reducible subgroups. Further, for any distinct virtually strongly undistorted pure reducible subgroups $H$ and $H'$ in $G$, no connected component of an element of $\Omega_H$ properly contains a connected component of an element of $\Omega_{H'}$.
        \item $\widehat{G}$ admits a $G$ equivariant quasi-isometric embedding into $C(\Sigma)$. Here for every peripheral subset $P=gH_i$ with $A$ the reducing system for some finite index pure reducible subgroup of $H_i$, there are $|A|$ cone points connected to the vertices of $P$. 
    \end{enumerate}
    In this case we say that $G$ is RGF relative to the collection $\{H_1,\ldots, H_n\}$. 
\end{definition}
Again, just as with PGF groups, the reason why we potentially include more than $1$ cone point is for the sake of equivariance.
\par 
We remark that the assumption on supports in the first part of the definition automatically holds in a variety of cases. For example, it holds for PGF groups as every component of the support is an annulus. More generally it holds if the homeomorphism types of the connected components of the supports are all the same. It also holds if the components of the support of every pure reducible element is a connected subsurface and a collection of annuli on the boundary components of this subsurface. This follows from Lemma \ref{FillingSubsurfaces}, as if there was proper nesting of connected proper subsurfaces, there is no way they could fill the surface. It is unclear if there are any examples of groups which satisfy all the other parts of the definition of RGF groups, but the nesting assumption fails.
\subsection{Proving Undistortion}\label{SubsectionProvingUndistortion}

We need the following result, which can be thought of as an analog of Corollary \ref{InfOrderpA} and Lemma \ref{FillingMC} for RGF groups. It is a corollary of Proposition \ref{DGORotatingSubgroups} applied to RGF groups.
\begin{lemma}\label{FillingSubsurfaces}
    Consider an RGF group $G$ relative to a collection $\mathcal{H}$. Fix two distinct peripheral subgroups $H_1$, $H_2$ of $G$ and associated subsurfaces $S_1$, $S_2$, respectively. Suppose $f_i\in H_i$ are nontrivial for $i=1,2$ with support $R_i\subset S_i$. Then there is a $k$ so that $f_1^kf_2^k$ is pseudo-Anosov on $\Sigma$. In particular, $R_1$ and $R_2$ fill $\Sigma$. Further, for any marking $\mu$ of $\Sigma$ there is a $D$ only depending on $\mu$ and $H_2$ so that for any component $R_1'$ of $R_1$,
    $$d_{R_1'}(\mu, h\mu)\leq D$$
    for any $h\in H_2$. 
\end{lemma}
\begin{proof}
    By Proposition \ref{DGORotatingSubgroups}, for every $i$ there is a finite set $F_i\subset H_i$ so that if $N_i \lhd H_i$ and $N_i\cap F_i=\varnothing$, then the smallest normal subgroup $N$ in $G$ generated by the $N_i$'s is a free product of some collection of conjugates of the $N_i$'s. Since MCG$(\Sigma)$ is residually finite \cite{FM11}, we can find such normal subgroups $\{N_1,\ldots, N_n\}$ so that $N_i$ is also finite index in $H_i$. It follows that there is a $k\geq 1$ so that $f_1^k$ and $f_2^k$ are in two distinct conjugates of some element(s) of $\{N_1,\ldots, N_n\}$. Therefore, no power of $f_1^kf_2^k$ is conjugate into a free factor of $N$. If $f_1^kf_2^k$ were peripheral, then some power of it would lie in a conjugate of some $N_i$, as $N_i$ is finite index in $H_i$. But then this power of $f_1^kf_2^k$ would lie in a conjugate of a free factor of $N$, which is a contradiction.
    \par 
    It follows that $f_1^kf_2^k$ is nonperipheral. Thus $f_1^kf_2^k$ acts loxodromically on $\widehat{G}$ by Proposition \ref{NonperiLoxodromic}, and by the definition of RGF groups it also acts loxodromically on $C(\Sigma)$, which implies that $f_1^kf_2^k$ must be pseudo-Anosov.
    \par 
    Then $R_1$ and $R_2$ fill $\Sigma$ or else $f_1^kf_2^k$ could not be pseudo-Anosov. To prove the inequality, suppose first that $\mu$ is a marking that contains $\partial S_2$ in its collection of base curve. Note that some component of $\partial S_2$ has to have nonempty projection to $R_1'$. If not, then $S_2$ and $R_1'$ are either disjoint or $R_1'\subset S_2$. If they are disjoint, $f_1^kf_2^k$ would fix $\partial R_1'$, which it can't as it is pseudo-Anosov. If $R_1'\subset S_2$, then by part (1) of Definition \ref{RGFGroups} $R_1'$ is actually a component of $S_2$. But then again $f_1^kf_2^k$ would fix $\partial R_1'$.
    \par 
    Then as any $h\in H_2$ permutes the components of $\partial S_2$, and some component of $\partial S_2$ has nonempty projection to $R_1'$, the bound follows in this case (for example, $D=4$ suffices by Lemma \ref{ProjLip}). For a general marking $\mu'$, we have that
    $$d_{R_1'}(\mu', h\mu')\leq d_{R_1'}(\mu, h\mu)+d_{R_1'}(\mu, \mu')+d_{R_1'}(h\mu, h\mu').$$
    by the triangle inequality. The first term on the right is less than or equal to $4$ by above, the the second and third terms are bounded depending only on $H_2$ (due to the choice of $\mu$) and $\mu'$ by Lemma \ref{IntersectionBound}. Note that this shows that $D$ does not depend on the choice of $H_1$ or $R_1'$.
\end{proof}
\par 
We remark that one could slightly extend the class of groups being considered in the proof of Theorem \ref{Undistorted} to a variety of groups that don't have the "no proper nesting" assumption we made in Definition \ref{RGFGroups}, but are otherwise defined in the same way. Namely, if one simply assumes the bounded diameter projections in Lemma \ref{FillingSubsurfaces}, then the proof would go through for such a group. This would work, for example, if one assumes that the restriction of every pure reducible subgroup to a component of its support so that this component properly nests a component of the support of something in another pure reducible subgroup is a convex cocompact group. This follows from a modification of Proposition \ref{PGFProjBound} (which for convex cocompact groups holds even when the surface being projected to is an annulus). We don't include this possibility as being a part of the definition of RGF groups as the current definition is already sufficiently cumbersome.
\par
Before continuing, we note the following useful simplification. Fix an RGF group $G$. While we have defined RGF groups in terms of virtual strongly undistorted pure reducible subgroups, by intersecting $G$ with the kernel of the action of MCG$(\Sigma)$ on $H_1(\Sigma, \Z/3)$ (see \cite{I96} for more about this group), we obtain a finite index subgroup whose peripherals are just strongly undistorted pure reducible subgroups. Showing that this finite index subgroup is undistorted in MCG$(\Sigma)$ will imply that $G$ is undistorted as well. We thus make the following assumption: \textit{Every peripheral subgroup of any RGF group will be assumed to be an strongly undistorted pure reducible subgroup}.
\par

We now fix an RGF group $G$ relative to $\mathcal{H}$, a collection of strongly undistorted pure reducible subgroups, with relative generating set $X$. We have the following proposition which contains the main work in showing that $G$ is undistorted in MCG$(\Sigma)$.

\begin{proposition}\label{FarPeripheral->FarSubsurface}
Fix a peripheral subset $P=fH$ with $H\in \mathcal{H}$ and $f\in G$. For every marking $\mu$ of $\Sigma$ there is a $\sigma_0$ so that for all $\sigma\geq \sigma_0$ there is a constant $\kappa$ independent of $P$ so that for all $g\in G$,
$$d_P(id, g) \preceq_{\kappa} \sum_{R\in f(\Omega_H)}\{\!\{d_R(\mu, g(\mu))\}\!\}_{\sigma}$$

\end{proposition}
\begin{proof}
 We may assume $\kappa \geq L$, where $L$ is as in Lemma \ref{SistoBGI}, so it suffices to assume $d_P(id, g)\geq L$. By Lemma \ref{SistoBGI}, every geodesic in $\widehat{G}$ from $e$ to $g$ passes through $\nu(P)$. Let $\gamma$ be a lift of one such geodesic to $G$. 
\par 
Write the word label of $\gamma$ as $g=g_1g_2g_3g_4g_5$, chosen as follows. The middle word $g_3\in H$ labels the segment of $\gamma$ in $P$, and $g_2$ and $g_4$ label segments of $\gamma$ that are sufficiently long in $\widehat{G}$ but uniformly bounded (what suffices as sufficiently long will be determined, see the discussion around inequality \ref{eq:MMBoundInitialSegment}). Finally, $g_1$ and $g_5$ are the initial and terminal segment labels left on $\gamma$ after choosing $g_2$ and $g_4$. See Figure \ref{FigureWordDecomp}. We then have the following by the triangle inequality. 
\begin{equation}\label{PeripheralBound}
d_P(id,g)\leq d_P(id, g_1g_2)
+d_P(g_1g_2,g_1g_2g_3)+d_P(g_1g_2g_3, g).
\end{equation}
For any $R\in f(\Omega_H)$, we also have
\begin{multline}\label{SurfaceBound}
d_R(\mu,g\mu)-d_R(\mu,g_1\mu)-d_R(g_1\mu,g_1g_2\mu)-d_R(g_1g_2g_3\mu,g_1g_2g_3g_4\mu)
\\-d_R(g_1g_2g_3g_4\mu,g\mu)
\\ 
\leq d_R(g_1g_2\mu,g_1g_2g_3\mu) 
\\ 
\leq d_R(\mu,g\mu)+d_R(\mu,g_1\mu)+d_R(g_1\mu,g_1g_2\mu)+d_R(g_1g_2g_3\mu,g_1g_2g_3g_4\mu)
\\+d_R(g_1g_2g_3g_4\mu,g\mu)
\end{multline}

By Lemma \ref{SistoEndpointBound}(b), $\pi_P(id), \pi_P(g)$ are within $E$ from $g_1g_2, g_1g_2g_3$ respectively. In particular, by the triangle inequality we obtain 
\begin{equation}\label{PeripheralBoundModified}
     d_P(id, g)\leq d_P(g_1g_2,g_1g_2g_3)+2E
\end{equation}
We also have
\begin{equation} \label{TranslatedDistanceFormula}
    d_P(g_1g_2,g_1g_2g_3)\approx \sum_{R\in f(\Omega_H)}\{\!\{d_R(g_1g_2\mu,g_1g_2g_3(\mu))\}\!\}_{\sigma}
\end{equation}

Indeed, this follows since we may write $g_1g_2=fh$ with $h\in H$, and
$$d_P(g_1g_2,g_1g_2g_3)=d_{H}(h,hg_3)=d_{H}(id,g_3)$$
and by \eqref{DistanceFormulaSupports},
$$d_{H}(id, g_3)\approx \sum_{R\in \Omega_H}\{\!\{d_R(\mu,g_3\mu)\}\!\}_{\sigma}$$ 
  By reversing the argument above, the right hand side of this is equal to the right hand side of \eqref{TranslatedDistanceFormula}. 
  \par 
  Our end goal is to uniformly bound every expression in \eqref{SurfaceBound} besides $d_R(\mu, g\mu)$ and $d_R(g_1g_2\mu, g_1g_2g_3\mu)$. Once this is done, it will follow that there is a constant $C$ which we may assume is larger than $\sigma_0$ so that
  $$d_R(g_1g_2\mu, g_1g_2g_3\mu) \approx_{1, C} d_R(\mu, g\mu).$$
  Applying Lemma \ref{DoubleBracketQIBounds}, we obtain
  \begin{equation}\label{DistanceFormComparison}
      \sum_{R\in f(\Omega_H)}\{\!\{d_R(g_1g_2\mu,g_1g_2g_3(\mu))\}\!\}_{2C} \preceq \sum_{R\in f(\Omega_H)}\{\!\{d_R(\mu, g(\mu))\}\!\}_{C}
  \end{equation}
   The right hand side is exactly the term we are interested in with $\sigma=C$, so by combining \eqref{PeripheralBoundModified} with \eqref{TranslatedDistanceFormula} and  \eqref{DistanceFormComparison} the statement of the lemma follows.

\par 
We now fix a subsurface $R$ as in  \eqref{SurfaceBound}, and obtain the uniform bounds for the desired terms of \eqref{SurfaceBound}. We first work with $d_R(\mu, g_1\mu)$, and note that the bound for $d_R(g_1g_2g_3g_4\mu, g\mu)$ is handled similarly. We may assume $g_1$ is nontrivial. Pick a base curve $\alpha$ of $\mu$ with nonempty projection to $R$. If no such curve exists, then $R$ is an annulus and we instead take $\alpha$ to be a curve whose projection to $R$ is a component of the transversal of $\mu$ at $R$, and whose distance to the components of $\mu$ is at most $2$ (any uniform choice of distance will do).
\begin{figure}[H]
    \centering
    \includegraphics[scale=.60]{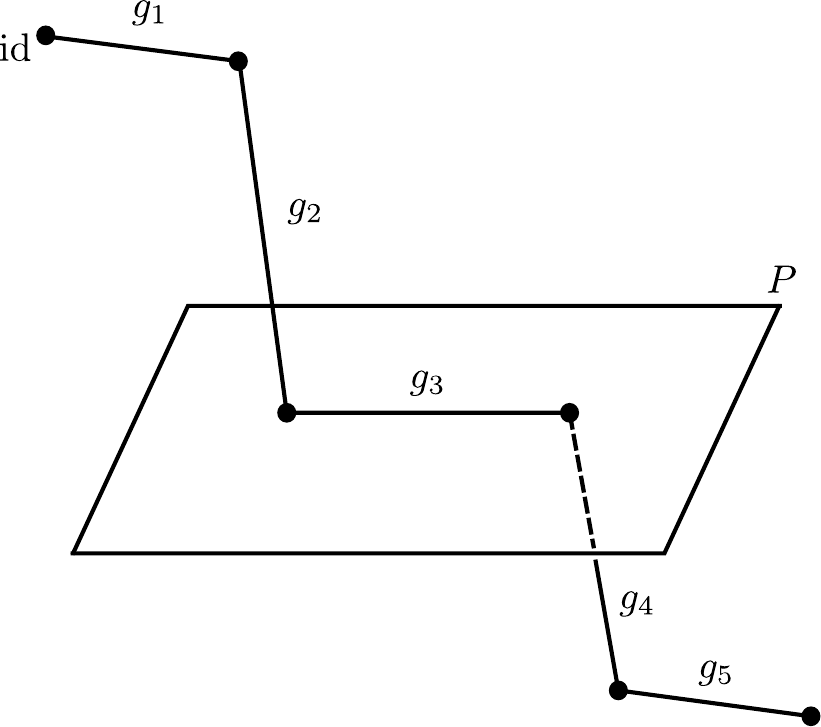}
    \caption{The decomposition of the element $g$.}
    \label{FigureWordDecomp}
\end{figure}
\par 
By assuming the $\widehat{G}$-length of $g_2$ is sufficiently long, we see that, since $\widehat{G}$ quasi-isometrically embeds in $C(\Sigma)$, the image of $[id,g_1]$, which is a uniform quasi-geodesic connecting $\alpha$ and $g_1\alpha$ for $\alpha$ a curve component of $\mu$ is sufficiently far from the image of $P$ in $C(\Sigma)$. The image of $P$ itself is in a bounded neighborhood (with bound depending only on $\mu$ and not on the peripheral $P$) of the multicurve $\partial R$. This is because the orbit of $\alpha$ under $P$ is the $f$ image of the $H$ orbit of $\alpha$, and $f^{-1}(\partial R)$ is disjoint from the multicurve associated to $H$. But the curves in $H\cdot f^{-1}(\partial R)$ are also disjoint from this multicurve, so the distance of the elements of $H\cdot \alpha$ from $f^{-1}(\partial R)$ is bounded in terms of the distance from $\alpha$ to $f^{-1}(\partial R)$. Hence the distance from $P\cdot \alpha$ to $\partial R$ is bounded similarly.  In particular, using the Proposition \ref{Morse}, we can apply Proposition \ref{MMBGI} to $[\alpha, g_1\alpha]$ to obtain 
\begin{equation}\label{eq:MMBoundInitialSegment}
    d_R(\alpha, g_1\alpha)\leq M
\end{equation}
as the geodesic $[\alpha, g_1\alpha]$ does not pass within the $1$ neighborhood of $\partial R$ as it is far from the image of $P$. By Lemma \ref{ProjLip}, this gives a uniform bound on $d_R(\mu, g_1\mu)$.
\par 
The argument of the previous two paragraphs only requires the $\widehat{G}$-length of $g_2$ and $g_4$ to be sufficiently large, only depending on the quasi-isometry constants of the embedding of $\widehat{G}$ into $C(\Sigma)$. In particular, this length is independent of the element $g$ chosen. (Recall if $g_2$ cannot be chosen to be sufficiently long, then $g_1$ is empty and the above bound is trivial).
\par 
Now, to obtain a bound on $d_R(g_1\mu, g_1g_2\mu)$ (for $d_R(g_1g_2g_3\mu, g_1g_2g_3g_4\mu)$, the argument once again goes the same way), we first write $R'=(g_1g_2)^{-1}(R)$, which is a element of $\Omega_H$, and then obtain
$$d_R(g_1\mu, g_1g_2\mu)=d_R(g_1g_2g_2^{-1}\mu, g_1g_2\mu)=d_{R'}(g_2^{-1}\mu, \mu)=d_{g_2(R')}(\mu, g_2\mu).$$

Let us also write
$$g_2=f_1\cdots f_{\ell}$$
where $f_i$ either an element of the relative generating set $X$ that is nonperipheral or is the label of a maximal subsegment of $g_2$ with label in some element of $\mathcal{H}$. Then writing 
$$h_j=f_1\cdots f_j$$
we obtain
\begin{equation}\label{BoundWordLength}
    d_{g_2(R')}(\mu, g_2\mu)\leq \sum_{j=0}^{\ell-1} d_{g_2(R')}(h_{j}(\mu), h_{j+1}(\mu)) = \sum_{j=0}^{\ell-1} d_{h_j^{-1}g_2(R')}(\mu, f_{j+1}(\mu)).
\end{equation}
But for $0\leq j\leq \ell-1$, $f_{j+1}\notin h_j^{-1}g_2Hg_2^{-1}h_j$. Indeed, if $f_{j+1}$ is nonperipheral then this automatically true. If $f_{j+1}$ is peripheral and $f_{j+1}\in h_j^{-1}g_2Hg_2^{-1}h_j$, then we write
$$f_{j+1}=h_j^{-1}g_2f'_{j+1}g_2^{-1}h_j$$
for some $f'_{j+1}\in H$. But $h_j^{-1}g_2=f_{j+1}\cdots f_{\ell}$, so the $\widehat{G}$ length of $h_j^{-1}g_2f'_{j+1}g_2^{-1}h_j$ is $2(\ell-j+1)+1$, while $|f_{j+1}|_{\widehat{G}}=1$. If $j<\ell-1$ this is a contradiction. If $j=\ell-1$ and $f_{\ell}\in h_{\ell-1}^{-1}g_2Hg_2^{-1}h_{\ell-1}=f_{\ell}Hf_{\ell}^{-1}$, then $f_{\ell}\in H$. But this is impossible since the endpoint of the subsegment of $\gamma$ labeled $g_2$ is assumed to be the first point of $\gamma$ in $P$, and if $f_{\ell}\in H$ this cannot be true.
\par 
Thus, each term on the right hand side of \eqref{BoundWordLength} is bounded by some $D'$, which depends on the collection of constants $D$ in Lemma \ref{FillingSubsurfaces} from each choice of element of $\mathcal{H}$, as well as the relative generating set $X$. Indeed, if $f_{j+1}$ is peripheral then the above argument shows that $f_{j+1}$ is not in the peripheral subgroup $h_j^{-1}g_2Hg_2^{-1}h_j$, so since $h_j^{-1}g_2(R')\in \Omega_{h_j^{-1}g_2Hg_2^{-1}h_j}$ the corresponding term in \eqref{BoundWordLength} is at most the constant $D$ depending on which element of $\mathcal{H}$ that $f_{j+1}$ is in, by Lemma \ref{FillingSubsurfaces}. If $f_{j+1}$ is nonperipheral, then it is an element of the relative generating set $X$. There are only finitely many elements in $X$, so there is a $D'$ which we can assume to be greater than all the $D$ constants from the elements of $\mathcal{H}$ so that for all $g\in X$,
$$d_Y(\mu, g\mu)\leq D'$$
for any subsurface $Y\subset S$. This follows from Lemmas \ref{IntersectionBound} and \ref{ProjLip}. As $\ell$ is uniformly bounded over all $g\in G$ we are done as we obtained the bound
$$d_R(g_1\mu, g_1g_2\mu)\leq D'\ell$$
with $D'$ and $\ell$ independent of the choice of $g$.
\end{proof}

\begin{remark}\label{WhyStronglyUndistorted2}
    Note that to obtain \eqref{DistanceFormComparison}, we require the bound on the relevant terms in \eqref{SurfaceBound} so that we can apply Lemma \ref{DoubleBracketQIBounds}. This bound uses the assumption that the subsurface $R$ is in $\Omega_H$, due to the application of Lemma \ref{FillingSubsurfaces} giving the bound $D$. The proof fails as is when the collection of subsurfaces is not restricted to $\Omega_H$, which shows as discussed in Remark \ref{WhyStronglyUndistorted} why we assume that the distance formula \eqref{DistanceFormulaSupports} holds. 
\end{remark}

\begin{theorem}\label{Undistorted}
Every RGF group $G$ of $\Sigma$ is undistorted in $\text{MCG}(\Sigma)$.
\end{theorem}
\begin{proof}
 Since RGF groups are finitely generated, the triangle inequality shows that there is coarse Lipschitz upper bound on the distance in the mapping class group compared to the distance in the PGF group. It thus suffices to show that 
$$d_{MCG}(id,g) \succeq d_G(id,g).$$
\par 
To do this, we compare the distance formula of $G$ as a relatively hyperbolic group to the Masur-Minsky distance formula. Let $\sigma_0$ be at least as large as both $\sigma_0$'s as in Propositions \ref{MMDistance} and \ref{SistoDistance}. Fix $\sigma\geq \sigma_0$ and a marking $\mu$. We then have

\begin{equation*}\label{DistanceFormulaComparison}
    d_{MCG}(id,f)\approx \sum_{R\subset \Sigma} \big{\{}\!\big{\{}d_{R}(\mu, f(\mu))\big{\}}\!\big{\}}_{\sigma}
    \geq
    d_{\Sigma}(\mu, f(\mu))+\sum_{R}\big{\{}\!\big{\{}d_R(\mu, f(\mu))\big{\}}\!\big{\}}_{\sigma}
\end{equation*}
where the later sum is over all subsurfaces $R$ that are in $\Omega_H$ for some peripheral subgroup $H$. The QI equivalence is Proposition \ref{MMDistance}, and the inequality follows from dropping all proper subsurfaces not in some $\Omega_H$. 
\par 
One can write
\begin{equation}\label{AllSurfacesSumRewrite}
    \sum_{R}\big{\{}\!\big{\{}d_R(\mu, f(\mu))\big{\}}\!\big{\}}_{\sigma} = \sum_{H} \sum_{R\in \Omega_H}\big{\{}\!\big{\{}d_R(\mu, f(\mu))\big{\}}\!\big{\}}_{\sigma}
\end{equation}
where the outer sum on the right hand side is over all peripheral subgroups $H$ of $G$, and the inner sum is as stated. This follows as there are no repeating terms on the right hand side, which follows from Lemma \ref{FillingSubsurfaces} as if two supports of elements in distinct peripherals contain a common component, then they could not fill $\Sigma$.
\par 
By the definition of RGF groups and Lemma \ref{ProjLip},
$$d_{\Sigma}(\mu, f(\mu)) \approx d_{\widehat{G}}(id, f).$$
If we write $H=fH'f^{-1}$ for $H'\in \mathcal{H}$, we denote by $P$ the corresponding peripheral subset $fH'$. By Proposition \ref{FarPeripheral->FarSubsurface}, there is a $K \geq 1$ so that, 
$$ \sum_{R\in \Omega_H}\big{\{}\!\big{\{}d_R(\mu, f(\mu))\big{\}}\!\big{\}}_{\sigma}\succeq_{K, 0} \{\!\{d_P(id, f)\}\!\}_{\sigma'}$$
for some large constant $\sigma'\geq \sigma$. Namely, we are saying that we can take the additive constant from Proposition \ref{FarPeripheral->FarSubsurface} to be $0$ by taking $\sigma'$ to be sufficiently large while also applying $\{\!\{\cdot\}\!\}_{\sigma'}$ to $d_P(id, f)$. This is because by assuming $\sigma'$ is sufficiently large, either $d_P(id, f)\leq \sigma'$, or the left hand side of the above expression must contain at least $1$ nonzero term, or else the conclusion of Proposition \ref{FarPeripheral->FarSubsurface} could not hold. Thus such a $K$ exists, 
$$\sum_H\sum_{R\in \Omega_H}\big{\{}\!\big{\{}d_R(\mu, f(\mu))\big{\}}\!\big{\}}_{\sigma}\succeq_{K, 0} \sum_P\{\!\{d_P(id, f)\}\!\}_{\sigma'}.$$
By Lemma \ref{lem:distincePeripherals}, distinct peripheral subgroups have distinct corresponding peripheral subsets. Thus there are no repeat terms on the right hand side.
\par 
 Combining everything together and applying Proposition \ref{SistoDistance}, we have
$$d_{MCG}(id,f)\succeq d_{\widehat{G}}(id, f)+\sum_{P}\{\!\{d_P(id, f)\}\!\}_{\sigma'} \approx d_G(id,f).$$
\end{proof}

In particular, Theorem \ref{Undistorted} implies Theorem 1.2 of \cite{L21} and Theorem 1.1 of \cite{T21} (noting Proposition \ref{TangVeechPGF}), giving an alternate proof of these results.

\section{Examples and applications}\label{SectionExamplesApplications}

\subsection{Constructing Compatible Homomorphisms}\label{SubsectionConstructingCompatibleHomomorphisms}
Following the techniques of \cite{LR06}, we can construct many examples of applications of Theorem \ref{ComboQIEmb}. These ideas allow one to build many new PGF groups from old ones. 
\par 
  Let us start first with any normalized graph of PGF groups $\mathcal{G}$ that is a finite tree. Denote by $G$ the fundamental group of $\mathcal{G}$. We assume here that there are no extension vertices. We will construct a compatible homomorphism for this tree of groups that has $L$-local large projections for $L$ arbitrarily large, under certain assumptions about the twist groups of $\mathcal{G}$.
\begin{proposition}\label{PGFTreeCompatibleHom}
    Let $\mathcal{G}$ be as above, and assume further that for all twist vertex groups $H$ of $\mathcal{G}$ with associated multicurve $A$, there is a partial pseudo-Anosov with reducing system $A$ centralizing $H$. For all $L$, there exists a compatible homomorphism $\phi$ so that $(\mathcal{G},\phi)$ satisfies the $L$-local large projections property. In particular, the fundamental group of every such tree of groups injects into MCG$(\Sigma)$ with image a group that is PGF relative to the $\phi$ images of the finitely many twist subgroups of each vertex group, with some twist groups identified according to $\mathcal{G}$.
\end{proposition}
\begin{proof}
To define the homomorphism $\phi$, we will need to associate maps to each vertex an element of MCG$(\Sigma)$. These elements will be exactly the conjugating elements that define the homomorphism $\phi$ when restricted to to corresponding vertex group. Let us first describe the technique for how we will do this in the simplest case. Fix a vertex $v$ in $\mathcal{G}$ with vertex group $G_{v}$. Consider a PGF vertex $v'$ with vertex group $G_{v'}$ with $v''$ being the only vertex between $v$ and $v'$. Let $H_{v''}$ be its vertex group, which is a twist group. Suppose $f_{e}, f_{e'}\in \text{MCG}(\Sigma)$ are the corresponding elements defining the edge map from $H_{v''}$ into $G_{v}$ and $G_{v'}$ respectively. (On a first reading it may be easiest to assume that $f_e=f_{e'}=id$, which means that $H_{v''}$ is an actual subgroup of $G_{v}$ and $G_{v'}$). Let $A_{v''}$ denote the multicurve associated to $H_{v''}$, and suppose $h_{v'}\in \text{MCG}(\Sigma)$ is a partial pseudo-Anosov with reducing system $f_e(A_{v''})$ centralizing $f_eH_{v''}f_{e'}^{-1}$. Such a map exists by the assumptions of the proposition. 
\par 
Since $f_{e'}H_{v''}f_{e'}^{-1}<G_{v'}$, it follows that $f_{e}H_{v''}f_{e}^{-1}<f_ef_{e'}^{-1}G_{v'}f_{e'}f_{e}^{-1}$, and thus $f_{e}H_{v''}f_{e}^{-1}<h_{v'}f_ef_{e'}^{-1}G_{v'}f_{e'}f_{e}^{-1}h_{v'}^{-1}$ because $h_{v'}$ centralizes $f_eH_{v''}f_e^{-1}$. As we also have that $f_eH_{v''}f_e^{-1}<G_v$ by the definition of $f_e$. Given any $L>0$, we may assume (potentially by replacing $h_{v'}$ with some large power) that for all multicurves $B_{v}$ and $B_{v'}$ associated to $G_v$ and $h_{v'}f_ef_{e'}^{-1}G_{v'}f_{e'}f_{e}^{-1}h_{v'}^{-1}$ respectively, with $B_v, B_{v'}\neq f_e(A_{v'})$, that we have the following inequality.
\begin{equation}\label{BaseCaseFarProjections}
    d_S(B_v, B_{v'})\geq L.
\end{equation}
To do this, apply Proposition \ref{PGFProjBound} to the projections of the multicurves of both $G_v$ and $f_ef_{e'}^{-1}G_{v'}f_{e'}f_{e}^{-1}$ to $S$ to get that they are bounded subsets of $C(S)$, and then use the fact that $h_{v'}$ acts loxodromically on $C(S)$. Note that we are also applying Lemma \ref{FillingMC} which implies that $\pi_S(B_v)$ and $\pi_S(B_{v'})$ are nonempty.
\par 
For any PGF vertex $v'$ so that there are no PGF vertices between $v$ and $v'$, we choose a map $h_{v'}$ as defined above. Our main goal is to define a map $h_{v'}$ associated to any PGF vertex $v'$ of $\mathcal{G}$. We think about $v$ as a basepoint for this construction. For convenience, we define $h_v=id$. Given any vertex $u$ of $\mathcal{G}$, we let $p_u$ denote the embedded path from $v$ to $u$.
\par
Now, given a vertex $u$ so that every map $h_{u'}$ is defined for every PGF vertex $u'$ on $p_u$, we define $h_{p_u}\in \text{MCG}(\Sigma)$ as follows. We let $h_{p_v}=id$. For $u\neq v$ a PGF vertex, we write $p_u=e_1e_1'\cdots e_{m-1}e_{m-1}'$. Then define $h_{p_u}$ as
$$h_{p_u}=h_{v_2}f_{e_1}f_{e_1'}^{-1}\cdots h_{u}f_{e_{m-1}}f_{e_{m-1}'}^{-1}$$
If $u$ is a twist vertex, and we write $p_u=e_1e_1'\cdots e_{m-1}$, then
$$h_{p_u}= h_{v_2}f_{e_1}f_{e_1'}^{-1}\cdots h_{v_{m-2}}f_{e_{m-2}}f_{e_{m-2}'}^{-1}f_{e_{m-1}}.$$
\par
To construct the maps $h_{v'}$ for all PGF vertices $v'$ of $\mathcal{G}$, we perform an inductive argument starting at $v$. Label the PGF vertices of $p_{v'}$ in order as $v=v_1, v_2, \ldots, v_{n-1}, v_n=v'$, and let $v_{i,i+1}$ denote the twist vertex between $v_i$ and $v_{i+1}$. We also label the edge between $v_i$ and $v_{i,i+1}$ as $e_i$, and the edge between $v_{i,i+1}$ and $v_{i+1}$ as $e_i'$. Let $G_{i}$ denote the PGF vertex group of $v_i$, and $H_{i, i+1}$ the twist vertex group of $v_{i,i+1}$ with associated multicurve $A_{i, i+1}$. Let $f_{e_i}$ and $f_{e_i'}$ denote the conjugating maps defining the edge maps from $H_{i, i+1}$ into $G_i$ and $G_{i+1}$, respectively. Note that with this notation we have
\begin{equation}\label{VertexConjugatorRelation1}
   h_{p_{v_{i-1,i}}}=h_{p_{v_{i-1}}}f_{e_{i-1}} 
\end{equation}
\begin{equation}\label{VertexConjugatorRelation2}
    h_{p_{v_i}}=h_{p_{v_{i-1}}}h_{v_i}f_{e_{i-1}}f_{e_{i-1}'}^{-1}
\end{equation}
\par 
Assume that we have inductively constructed $h_{v_2},\ldots, h_{v_{n-1}}$ with the following properties (note that the base case $i=2$ was done at the beginning of the proof). For every $2\leq i \leq n-1$ we have that
\begin{enumerate}[label=(\alph*)]
    \item $h_{v_i}$ is a partial pseudo-Anosov with reducing system $h_{p_{v_{i-1,i}}}(A_{i-1, i})$ centralizing $h_{p_{v_{i-1,i}}}H_{i-1, i}h_{p_{v_{i-1,i}}}$.
    \item Given a multicurve $B_{i-1}$ associated to $h_{p_{v_{i-1}}}G_{i-1}h_{p_{v_{i-1}}}^{-1}$
    and a multicurve $B_{i}$ associated to $h_{p_{v_{i}}}G_{i}h_{p_{v_{i}}}^{-1}$ both of which are not equal to $h_{p_{v_{i-1,i}}}(A_{i-1, i})$, for any component $S$ of $\Sigma \setminus h_{p_{v_{i-1,i}}}(A_{i-1, i})$ we have that
    $$d_S(B_{i-1}, B_i)\geq L.$$
\end{enumerate}
Then we can define $h_{v_n}$ by constructing it in the same way as we constructed $h_{v'}$ in the beginning of the proof. Namely, we mimic the construction, but instead starting with the PGF groups $h_{p_{v_{n-1}}}G_{n-1}h_{p_{v_{n-1}}}^{-1}$ and $h_{p_{v_{n}}}G_{n}h_{p_{v_{n}}}^{-1}$, and twist group $h_{p_{v_{n-1}}}H_{n-1, n}h_{p_{v_{n-1}}}^{-1}$. We choose a partial pseudo-Anosov $h_{v_n}$ with reducing system $h_{p_{v_{n-1}}}f_{e_{n-1}}(A_{n-1,n})=h_{p_{v_{n-1,n}}}(A_{n-1,n})$ centralizing $h_{p_{v_{n-1,n}}}H_{n-1, n}h_{p_{v_{n-1,n}}}^{-1}$. Such a map exists by the assumption on the twist groups of $\mathcal{G}$ given in the proposition. We may then further assume, possibly by replacing $h_{v_n}$ with a power, that given a multicurve $B_{n-1}$ associated to $h_{p_{v_{n-1}}}G_{n-1}h_{p_{v_{n-1}}}^{-1}$
    and a multicurve $B_{n}$ associated to $h_{p_{v_{n}}}G_{n}h_{p_{v_{n}}}^{-1}$ both of which are not equal to $h_{p_{v_{n-1,n}}}(A_{n-1, n})$, for any component $S$ of $\Sigma \setminus h_{p_{v_{n-1,n}}}(A_{n-1, n})$ we have that
    $$d_S(B_{n-1}, B_n)\geq L.$$
    To see why we can assume this, it suffices by equivariance (conjugating by $h_{p_{v_{n-1}}}^{-1}$) to choose $h_{v_n}$ so that for all multicurves $B_{n-1}'$ associated to $G_{n-1}$ and all multicurves $B_n'$ associated to $h_{v_n}f_{e_{m-1}}f_{e_{n-1}'}^{-1}G_n f_{e_{n-1}'}f_{e_{n-1}}^{-1}h_{v_n}^{-1}$ with $B_{n-1}', B_n'\neq f_{e_{n-1}}(A_{n-1,n})$, and any component $S'$ of $\Sigma\setminus f_{e_{n-1}}(A_{n-1,n})$, we have
    $$d_{S'}(B_{n-1}', B_n')\geq L.$$
    Note that the multicurve $f_{e_{n-1}}(A_{n-1,n})$ comes from applying \eqref{VertexConjugatorRelation1} to get $h_{p_{v_{n-1}}}^{-1}h_{p_{v_{n-1,n}}}=f_{e_{n-1}}$, and the group $h_{v_n}f_{e_{m-1}}f_{e_{n-1}'}^{-1}G_n f_{e_{n-1}'}f_{e_{n-1}}^{-1}h_{v_n}^{-1}$ comes from applying  \eqref{VertexConjugatorRelation2} in a similar way to the group $h_{p_{v_{n}}}G_{n}h_{p_{v_{n}}}^{-1}$. Applying the same argument used to obtain inequality \eqref{BaseCaseFarProjections}, we see that we can find a map $h_{v_n}$ satisfying (a) and (b) above with $i=n$, completing the induction.
    \par 
    At this point there is some ambiguity for how $h_{v'}$ is defined, as there may be multiple embedded paths containing $v'$ starting at $v$. However, there is no extra work to be done for us to simply assume that each choice for $h_{v'}$ is the same for any embedded path. Thus we have associated a well defined pure reducible mapping class $h_{v'}$ to every PGF vertex $v'$ of $\mathcal{G}$, and for any vertex $u$ we have a well defined mapping class $h_{p_u}$ associated to the path $p_u$ from $v$ to $u$.
\par 
    Finally, we can now define the map which will be our desired compatible homomorphism. It suffices to define the map on the vertex groups first. On a vertex group $G_u$ of the vertex $u$, we define for $g\in G_u$
$$\phi(g)=h_{p_u}gh_{p_u}^{-1}$$
Once we show that $\phi$ is a homomorphism, it will follow by construction and Lemma \ref{LocalLargeEquiv} that $\phi$ is a compatible homomorphism for $\mathcal{G}$ so that $(\mathcal{G},\phi)$ satisfies the $L$-local large projections property, finishing the proof. Note that applying Lemma \ref{LocalLargeEquiv} is straightforward in this case, as we can lift the entire tree $\mathcal{G}$ to its Bass--Serre tree homeomorphically.
\par 
Thus it suffices to check that $\phi$ is a homomorphism. To do this, it suffices to check that the relations coming from edges still hold. Consider a pair of PGF vertices $u$ and $u'$ with twist vertex $u''$ between them. Assume $p_{u'}$ contains $u$, and write $p_{u'}=e_1e_1'\cdots e_{n-1}e_{n-1}'$. Then we need to check for $h\in H_{u''}$, the twist vertex group of $u''$, that we have
$$\phi(h)=\phi(f_{e_{n-1}}hf_{e_{n-1}}^{-1})=\phi(f_{e_{n-1}'}hf_{e_{n-1}'}^{-1})$$
By definition, $\phi(h)=h_{p_{u''}}hh_{p_{u''}}^{-1}$. Since $$\phi(f_{e_{n-1}}hf_{e_{n-1}}^{-1})=h_{p_{u}}f_{e_{n-1}}hf_{e_{n-1}}h_{p_{u}}^{-1}=h_{p_{u''}}hh_{p_{u''}}^{-1}$$
by \eqref{VertexConjugatorRelation1}, the first equality holds. Similarly, 
\begin{multline*}
    \phi(f_{e_{n-1}'}hf_{e_{n-1}'}^{-1})=h_{p_{u'}}f_{e_{n-1}'}hf_{e_{n-1}'}^{-1}h_{p_{u'}}^{-1}
    \\
    =h_{p_{u}}h_{u'}f_{e_{n-1}}f_{e_{n-1}'}^{-1}f_{e_{n-1}'}hf_{e_{n-1}'}^{-1}f_{e_{n-1}'}f_{e_{n-1}}h_{u'}^{-1}h_{p_{u}}^{-1}
    \\
    =h_{p_{u}}h_{u'}f_{e_{n-1}}hf_{e_{n-1}}h_{u'}^{-1}h_{p_{u}}^{-1}=h_{p_{u}}f_{e_{n-1}}hf_{e_{n-1}}h_{p_{u}}^{-1}=h_{p_{u''}}hh_{p_{u''}}^{-1}
\end{multline*}
where the second equality is \eqref{VertexConjugatorRelation2}, the fourth is the fact that $h_{u'}$ centralizes $f_{e_{n-1}}H_{u''}f_{e_{n-1}}$, and the fifth is \eqref{VertexConjugatorRelation1}. Thus $\phi$ is a homomorphism. The latter statement of the proposition follows by applying Theorem \ref{ComboQIEmb}, taking $L$ to be sufficiently large.
\end{proof}

Next, we show how to construct HNN extensions and twist group rank increases (as in Definition \ref{PGFGraph}(2)). We fix two normalized PGF graphs of groups $\mathcal{G}_{HNN}$ and $\mathcal{G}_{Ext}$. The first is a graph with a two vertices and two edges, one with vertex group a PGF group $G$ and the other a twist group $H$ with sparse associated multicurve. The edge groups are both equal to $H$, with edge maps to the $H$ vertex group given by the identity, and edge maps to $G$ defined via conjugation by elements of MCG$(\Sigma)$ onto two nonconjugate (in $G$) twist subgroups that $G$ is PGF relative to. Similarly, $\mathcal{G}_{Ext}$ consists of three vertices arranged in a line segment with two edges. One of the valence one vertices is a PGF vertex with vertex group $G$, the middle vertex is a base vertex with vertex group $H$, and the other valence one vertex is an extension vertex with vertex group $H'$ with associated multicurve $A'$. The edge groups are both equal to $H$. We may assume by a conjugation that both edge maps into $G$ and $H'$ are just inclusions of $H$ into both groups. In both cases, let $A$ denote the multicurve associated to $H$. 

\begin{proposition}\label{PGFHNNCompatibleHom}
    Consider the normalized PGF graph of groups $\mathcal{G}_{HNN}$ as above, and suppose there is a partial pseudo-Anosov $h$ with reducing system $A$ with $h$ centralizing $H$. For all $L$, there exists a compatible homomorphism $\phi$ so that $(\mathcal{G}_{HNN},\phi)$ satisfies the $L$-local large projections property. In particular, the fundamental group of $\mathcal{G}$ injects into MCG$(\Sigma)$ with image a group that is PGF relative to the $\phi$ images of the twist groups of $G$, with one of the two twist groups that are images of $H$ under the edge maps of $\mathcal{G}$ removed from the collection.
\end{proposition}
\begin{proof}
    We may assume $H$ is equal to one of its $G$ images, and the edge map into $G$ to the corresponding subgroup of $G$ is the identity. Denote the image of the other edge homomorphism by $H'$, and let $f\in \text{MCG}(\Sigma)$ be a conjugating element inducing the edge isomorphism $\psi:H\to H'$. The construction here is essentially a simpler version of that in Theorem \ref{PGFTreeCompatibleHom}. Let $t$ denote the stable letter of $\mathcal{G}_{HNN}$. We define the homomorphism $\phi:G*_{\psi}\to \text{MCG}(\Sigma)$ by sending $G$ to itself identically, and sending $t$ to $h^nf^{-1}$, where $n$ will be determined later. To see that this is a homomorphism, it suffices to check for $g \in H$ that
    $$\phi(t)^{-1}g\phi(t)=fgf^{-1}.$$
    But this is immediate from the definition of $\phi(t)$ and the fact that $h$ commutes with $g$. 
    \par 
    To obtain $L$-local large projections, it suffices to choose $n$ so that for all multicurves $B_1$, $B_2$ associated to $G$ with $B_1\neq A$ and $B_2\neq \phi(t)^{-1}(A)$ and all components $S$ of $\Sigma \setminus A$,
    $$d_S(B_1, \phi(t)(B_2))\geq L.$$
    This is possible by Proposition \ref{PGFProjBound} and the fact that $h$ acts loxodromically on $C(S)$.
    \par 
    The final claim then follows from applying Lemma \ref{LocalLargeEquiv} and Theorem \ref{ComboQIEmb} by taking $L$ sufficiently large.
\end{proof}
Finally, we deal with rank extensions of twist groups. Before we can begin the argument, we will need a lemma. First, a result of Loa \cite{L21}. We fix a compact surface $S$, and a multicurve $A$ of $S$. We also write $\delta$ for a hyperbolicity constant for $C(\Sigma)$, and let $N>M+5$ where $M$ is as in Proposition \ref{MMBGI}.
\begin{lemma}[{\cite[Lemma 4.2]{L21}}]\label{LoaFarTwists}
    Take $\beta\in C(\Sigma)$ with $d_S(\beta, A)\geq D$ for some constant $D$. Then for all nontrivial multitwists $\tau$ on $A$,
    $$d_S(\beta, \tau(\beta))\geq 2D-2((N+1)\delta+2).$$
\end{lemma}
\begin{lemma}\label{FarMulticurvesFarTranslation}
    Fix a compact surface $S$, and let $\mathcal{B}\subset C(S)$ be a collection of curves of finite diameter $K$. Suppose $d_S(\beta, A)\geq D$ for some constant $D$ and all $\beta\in \mathcal{B}$. Then for all nontrivial multitwists $\tau$ on $A$ and any $\beta, \beta' \in \mathcal{B}$,
    $$d_S(\beta, \tau(\beta'))\geq 2D-2((N+1)\delta+2))-K.$$
\end{lemma}
\begin{proof}
    This follows by applying Lemma \ref{LoaFarTwists} to every $\beta\in \mathcal{B}$ first, and then using the triangle inequality to replace $\tau(\beta)$ with $\tau(\beta')$, which are at most $K$ from each other.
\end{proof}

\begin{proposition}\label{PGFRankExtCompatibleHom}
    Consider the graph $\mathcal{G}_{Ext}$, and suppose there is a partial pseudo-Anosov $h$ with reducing system $A$ centralizing $H$. Then for all $L>0$, there is a compatible homomorphism $\phi$ so that $(\mathcal{G}_{Ext},\phi)$ satisfies the $L$-local large projections property. In particular, the fundamental group of $\mathcal{G}_{Ext}$ injects into MCG$(\Sigma)$ with image a group that is PGF relative to the twist groups of $G$, with $H$ replaced with $H'$.
\end{proposition}
\begin{proof}
    We define $\phi: G*_H H'\to \text{MCG}(\Sigma)$ by sending $G$ to itself identically, and sending $h'\in H'$ to $h^nh'h^{-n}$ for some $n$ to be determined. As $h$ centralizes $H$, this map is a homomorphism. We also let $A'$ denote the multicurve associated to $H'$.
    \par 
    By Lemma \ref{LocalLargeEquiv}, to obtain $L$-local large projections, it suffices to show that for all nontrivial multitwists $\tau$ in the new multicurves of $h^nH'h^{-n}$ (which are multitwists on $h^n(A')\setminus A$), all multicurves $B$ associated to $G$ with $B\neq A$ and all components $S$ of $\Sigma \setminus A$,
    $$d_S(B, \tau(B))\geq L.$$
    But utilizing the projection bounds from Proposition \ref{PGFProjBound} and applying Lemma \ref{FarMulticurvesFarTranslation} and the fact that $h$ acts loxodromically on $C(S)$, this is immediate by taking $n$ large enough so that the components of the multicurve $h^n(A')\setminus A$ which are in $S$ are sufficiently far from the projections of the multicurves of $G$. Note that by Definition \ref{PGFGraph}(c), such components of $h^n(A')\setminus A$ actually exist, and further $\tau$ twists on at least one of them. 
    \par 
    The final claim then follows from applying Lemma \ref{LocalLargeEquiv} and Theorem \ref{ComboQIEmb} by taking $L$ sufficiently large.
\end{proof}
Combining all of these together, we obtain the following result.
\begin{theorem}\label{ComboApplication}
    Let $\mathcal{G}$ be a normalized PGF graph of groups so that for every base twist vertex group with associated multicurve $A$, there is a partial pseudo-Anosov $h$ with reducing system $A$ centralizing $H$. Then there is an injective compatible homomorphism $\phi$ whose image is PGF relative to the twists groups of the PGF vertex groups of $\mathcal{G}$, where any base twist group $H$ is replaced with the extension $H'$, and some twist groups are removed if they are identified with another as in Proposition \ref{DahmaniCombination}(3). 
\end{theorem}
\begin{proof}
    Apply Proposition \ref{PGFTreeCompatibleHom} to a subtree of $\mathcal{G}$ that contains all PGF vertices. This gives a PGF group $G$ which is the fundamental group of this tree of groups. By collapsing this tree we obtain a new graph of groups $\mathcal{G}'$ with only one PGF vertex $v$ with vertex group $G$. The underlying graph is isomorphic to a wedge of circles with each circle subdivided by a twist vertex, potentially along with edges from the center vertex to a base vertex, which themselves are connected to a degree $1$ extension vertex. 
    \par 
    By Definition \ref{PGFGraph}(a), the edge groups of any pair of edges of a circle cannot map to conjugate twist groups in the vertex group of the PGF vertex. Indeed, Definition \ref{PGFGraph}(a) ensures such examples did not already exist in $\mathcal{G}$, and that no such examples are made after the collapse of the tree. If such an example were made, then there would have to be a loop in the underlying graph of $\mathcal{G}$ (such as in Figure \ref{FigurePGFNonExample}, for example) where all the edge groups are identified in the fundamental group of the graph of groups $\mathcal{G}$. This contradicts Definition \ref{PGFGraph}(a). We can thus inductively apply Proposition \ref{PGFHNNCompatibleHom} to each circle, along with Proposition \ref{PGFRankExtCompatibleHom} to each nonloop edge pair to obtain the desired result. 
\end{proof}
If all twist groups are actually multitwist groups, then the centralizing partial pseudo-Anosovs will always exist and Theorem \ref{ComboApplication} is easily applied. More generally, one can produce these centralizing pseudo-Anosov elements if the quotient orbifolds coming from the action of a twist group on the complementary components of its associated multicurve admit pseudo-Anosovs. 
\par 
Note that we gave explicit bounds for $L$ and the constant $s$ from Lemma \ref{FarCurvesIntersect} in Sections \ref{SectionWorkingTowardsComboTheorem} and \ref{SectionProofMainThemOtherCases} independent of the PGF vertices of $\mathcal{G}$. Using this, one could formulate an effective version of Theorem \ref{ComboApplication} taking into account the bounds from Proposition \ref{PGFProjBound}.

\subsection{Explicit applications}\label{SubsectionApplicationsConvexCocompact}

We begin first with a straightforward application of Theorem \ref{ComboApplication}. First we give a definition.
\begin{definition}\label{BooksofIBundles}
    Begin first with a $2$-complex $X$ that consists of attaching a finite number of compact surfaces with a single boundary component along their common boundary component. Fix an embedding of $X$ into $\R^3$, and consider a regular neighborhood $M$. Such a $3$-manifold $M$ is called a \textit{book of I-bundles}. This name comes from the fact that the neighborhoods of each surface are $I$-bundles over the given surface.
\end{definition}
See \cite{CM04} for more about such spaces. We note that as $M$ deformation retracts to $X$,  $\pi_1(M)\cong \pi_1(X)$, which is a iterated free product with amalgamation of free groups along certain cyclic subgroups. The fundamental groups of books of $I$-bundles are natural generalizations of surface groups. We then have the following result.
\begin{corollary}\label{BookOfIBundleCombo}
    Fix a book of $I$-bundles $M$ so that each $I$-bundle is over a surface of genus $g$. Then $\pi_1(M)$ embeds in MCG$(\Sigma)$ as a PGF group.
\end{corollary}
\begin{proof}
    This is a straightforward generalization of Corollary 1.2 of \cite{LR06}. Specifically, Section 5 of that same paper discusses examples of Veech groups originally constructed in \cite{V89} which are isomorphic to $\pi_1(S_{g,1})$, the fundamental group of the surface of genus $g$ with $1$ boundary component, and whose peripheral subgroup is generated by a positive multitwist on a sparse multicurve, which corresponds to the element of $\pi_1(S_{g,1})$ that is the boundary of $S_{g,1}$. 
    \par 
    Then $\pi_1(M)$ can be formed by amalgamating a collection of Veech groups along their peripheral subgroups. As Veech groups are PGF groups (Proposition \ref{TangVeechPGF}), this means we can apply Proposition 7.1 to obtain the desired embedding of $\pi_1(M)$ as a PGF group.
\end{proof}

Next we give an example which can be thought of as a generalization of Proposition \ref{LoaCombo}.

\begin{example}\label{FreeProductMultitwists}
    Take two multicurves $A_1$ and $A_2$ with $A_i$ having $k_i$ components. If $d_{\Sigma}(A_1, A_2)$ is sufficiently large, we may apply Proposition \ref{LoaCombo} to get a PGF group isomorphic to $\Z^{k_1'} *\Z^{k_2'}$, where $k_i'\leq k_i$ and each free abelian factor is generated by multitwists on $A_i$. Further, this group is PGF relative to the two free abelian factors. Using Proposition \ref{PGFTreeCompatibleHom}, we may then construct PGF groups that are isomorphic to $\Z^{k_1'}*\Z^{k_2'}*\cdots *\Z^{k_n'}$, where each $\Z^{k_i'}$ factor is generated by multitwists on a multicurve with $k_i$ components, with $k_i'\leq k_i$, and the group is PGF relative to these factors, and so that at least one multicurve is sparse (to allow Theorem \ref{ComboQIEmb} to actually be applied). 
    \par 
    One may also drop the sparse factor, as it is clear that the subgroup given by doing so is still PGF. That is, dropping the sparse factor gives a relatively hyperbolic group, and it is clear that the "sub coned off graph" of the group obtained by dropping this factor quasi-isometrically embeds into the original coned off graph, and thus into $C(\Sigma)$. Thus we obtain as PGF groups all the isomorphism classes of groups of the form $\Z^{k_1'}*\Z^{k_2'}*\cdots *\Z^{k_n'}$ where $1\leq k_i'\leq 3g-3$, each factor is a multitwist group, and that are PGF relative to the factors of the free product. Any element that is not conjugate into one of the free factors is pseudo-Anosov. The fact that these groups are undistorted can be thought of as complementary to the results found in \cite{R20}.
    \par 
    We remark that it is possible to also obtain groups which look like those above purely from Proposition \ref{LoaCombo}. Namely, one uses that result to get a PGF group isomorphic to $\Z^{3g-3}*\Z^{3g-3}$. One can obtain subgroups of this isomorphic to $G=\Z^{k_1'}*\Z^{k_2'}*\cdots *\Z^{k_n'}$ where $1\leq k_i'\leq 3g-3$ by taking finite index subgroups and dropping generators. The action of this subgroup on a Bass--Serre tree $T$ for $\Z^{3g-3}*\Z^{3g-3}$ induces the splitting of the subgroup. The minimal invariant subtree of $T$ for this subgroup is quasi-isometric to the coned off graph of $G$ (with the free factors as the peripherals), and this subtree quasi-isometrically embeds into $C(\Sigma)$, implying $G$ is PGF.
\end{example}

We also prove an easy lemma which is something of a converse to Example \ref{FreeProductMultitwists}. Specifically, we are interested in right angled Artin groups that are subgroups of MCG$(\Sigma)$ whose vertex elements are multitwists, and that are PGF groups relative to subgroups generated by vertices. Given a graph $\Gamma$, we let $A(\Gamma)$ denote the corresponding right angled Artin group, or RAAG.

\begin{lemma}\label{GenericPGFRAAGS}
    Let $G=A(\Gamma)$ be a RAAG, and suppose it is embedded in MCG$(\Sigma)$ so that every vertex element is sent to a multitwist on some multicurve in $\Sigma$. Suppose $G$ is PGF relative to a collection of subgroups $\{H_i\}_{i=1}^n$, where $H_i=A(\Gamma_i)$ with $\Gamma_i$ a full subgraph of $\Gamma$. Then each $\Gamma_i$ is a complete graph, and $\Gamma$ is the disjoint union of the $\Gamma_i$'s. In particular, 
    $$G\cong  A(\Gamma_1)*\cdots *A(\Gamma_n).$$
\end{lemma}
\begin{proof}
    The first claim, that the $\Gamma_i$'s are complete graphs, just follows from definition, as the $H_i$'s must be groups of commuting multitwists. If there were an edge between $\Gamma_i$ and $\Gamma_j$ with $i\neq j$, let $a_i$ and $a_j$ denote the corresponding commuting group elements. Then clearly the cosets $H_i$ and $a_jH_i$ would have neighborhoods with infinite diameter intersection, which is impossible. Finally, the fact that there are no other generators follows from Corollary \ref{InfOrderpA}, as such generators could not be multitwists.
\end{proof}

In Example \ref{GeneralPGFRAAGS} we will also give examples with a free factor which is a free convex cocompact group. In fact, by the results on RAAGs in \cite{BDM09}, any RAAG that can be realized as a PGF group must be of the form $F*T$, where $F$ is a free convex cocompact group, and $T$ is a free product of free abelian subgroups of distinct twist groups. Namely, \cite{BDM09} shows that no RAAG whose graph is connected and not a point can be relatively hyperbolic. Further, if some component of the defining graph is not a isolated vertex, then it must be a peripheral subgroup. In particular, as the peripheral subgroups of a PGF RAAG must be abelian, it follows that every component of the defining graph is a complete graph on finitely many vertices. Note that the nonperipheral free group $F$ generated by all the nonperipheral elements (which correspond to a collection of isolated vertices in the defining graph) is indeed convex cocompact as $F$ equivariantly quasi-isometrically embeds into the coned off graph of $F*T$. 

We end with applications of Theorems \ref{ConvexCocompactCombo}, \ref{ConvexCoCptTwists}, and \ref{LoaAnalog}. Note that, given a nonseparating curve $\alpha\in C(\Sigma)$, if $g\in \text{MCG}(\Sigma)$ is such that $g(\alpha)\neq \alpha$ then $g(\alpha)$ intersects every component of $\Sigma \setminus \alpha$. This is trivial as there is only one component. We state this explicitly now as the hypotheses of Theorems \ref{ConvexCocompactCombo}, \ref{ConvexCoCptTwists}, and \ref{LoaAnalog} require this intersection to be nonempty, and we will utilize this in the following corollaries.

\begin{theorem}\label{SuperIndependentpA}
    Let $f_1,\ldots f_n\in \text{MCG}(\Sigma)$ be any pseudo-Anosov elements. Then there exists a reducible $f\in \text{MCG}(\Sigma)$ and numbers $K_2, \ldots, K_n\geq 0$ with $k_i\geq k_{i-1}+K_i$ for $i=2,\ldots n$ and $k_1=0$ so that the natural map
    $$\<f^{k_1}f_1f^{-k_1}\> * \cdots * \<f^{k_n}f_nf^{-k_n}\>\to \text{MCG}(\Sigma)$$
    is injective with convex cocompact image.
\end{theorem}
\begin{proof}
     Fix any choice of nonseparating curve $\alpha$. We can let $f$ be any partial pseudo-Anosov of $\Sigma$ on $\Sigma \setminus \alpha$. Then we can choose $k_2$ so that for $m_1, m_2\neq 0$,
     $$d_{\Sigma\setminus \alpha}(f_1^{m_1}(\alpha), f^{k_2}f_2^{m_2}f^{-k_2}(\alpha))\geq M+18$$    
     where $M$ is as in Proposition \ref{MMBGI}. This follows from Proposition \ref{PGFProjBound} applied to $\Sigma \setminus \alpha$ giving a bound on the projections of $\<f_1\>\cdot \alpha$ and $\<f^{k_2}f_2f^{-k_2}\>\cdot \alpha$ to $\Sigma \setminus \alpha$. Applying Theorem \ref{ConvexCocompactCombo} shows that $\<f_1\>*\<f^{k_2}f_2f^{-k_2}\>\to \text{MCG}(\Sigma)$ is injective with convex cocompact image.
     \par 
     The general statement follows from induction, applying Proposition \ref{PGFProjBound} to the groups $\<f^{k_1}f_1f^{-k_1}\>*\cdots \<f^{k_i}f_if^{-k_i}\>$ and $\<f_{i+1}\>$ and then conjugating the latter group so that the required projection bounds hold to apply Theorem \ref{ConvexCocompactCombo}.
\end{proof}

 One can think of this result as an analog of Theorem 1.4 of \cite{FM02}, where instead of using independent pseudo-Anosovs with sufficiently high powers, we take any collection of pseudo-Anosovs that are "sufficiently independent" (relative to some subsurface). 
\par 
For the rest of this section, we will first state more abstract results which follow quickly from Theorems \ref{ConvexCocompactCombo}, \ref{ConvexCoCptTwists}, and \ref{LoaAnalog}, which are then followed by more concrete applications. Their proofs follow from the same formula as Theorem \ref{SuperIndependentpA}. Many of these rely on the following proposition to guarantee that the twists groups involved (either the twist groups associated to a PGF group or some twist group which will be a free factor of a free product) don't fix a particular curve. This result can be found in Lemma 4.2 of \cite{I96}.

\begin{proposition}\label{IvanovIntersectionBound}
    Fix a multicurve $A$ in $\Sigma$ with components $\alpha_1,\ldots \alpha_m$. Let $\tau_i$ denote the Dehn twist on $\alpha_i$. Then for all integers $n_1,\ldots, n_m$ and all $\alpha, \beta \in C(\Sigma)$, if $\tau=\tau_1^{n_1}\cdots \tau_m^{n_m}$,
    \begin{multline*}
        \sum_{i=1}^m\big( (|n_i|-2)i(\alpha, \alpha_i)i(\alpha_i, \beta)\big)-i(\alpha,\beta) \leq i(\tau(\beta), \alpha) 
        \\
        \leq \sum_{i=1}^m \big(|n_i|i(\beta, \alpha_i)i(\alpha_i, \alpha)\big)+i(\beta,\alpha). 
    \end{multline*}
    If $n_i\geq 0$ or $n_i\leq 0$ for all $i$, then the left most expression can be taken to be
    $$\sum_{i=1}^m \big(|n_i|i(\beta, \alpha_i)i(\alpha_i, \alpha)\big)-i(\beta,\alpha).$$
\end{proposition}

We now state the application of Theorem \ref{ConvexCocompactCombo}. 

\begin{theorem}\label{CocptFreeProductApplicationAbstract}
    Suppose $G_1$, $G_2$ are PGF groups relative to $\mathcal{H}_1$, $\mathcal{H}_2$, respectively. If there is a curve $\alpha$ so that for all nontrivial $g_1\in G_1$, $g_2\in G_2$, $g_1(\alpha)$ and $g_2(\alpha)$ intersect every component of $\Sigma\setminus \alpha$, then there is a reducible $f\in \text{MCG}(\Sigma)$ so that the natural map $fG_1f^{-1}*G_2\to \text{MCG}(\Sigma)$ is injective with PGF image relative to the union of the $f$ conjugates of $\mathcal{H}_1$ along with the elements of $\mathcal{H}_2$.
\end{theorem}
\begin{proof}
    By applying Proposition \ref{PGFProjBound} to both $G_1$ and $G_2$, we may choose $f$ to be a partial pseudo-Anosov on $\Sigma \setminus \alpha$ with sufficiently large translation length in the curve graphs of the components of $\Sigma \setminus \alpha$ so that for each component $S$ of $\Sigma\setminus \alpha$ and all nontrivial $g_1\in G_1, g_2\in G_2$,
    $$d_S(fg_1f^{-1}(\alpha), g_2(\alpha))\geq M+18$$
    where $M$ is as in Proposition \ref{MMBGI}. The result then follows by applying Theorem \ref{ConvexCocompactCombo}.
\end{proof}

In particular, we obtain the following corollary. We state it in a somewhat specific form, though our applications will use a slightly generalized form of it. 

\begin{corollary}\label{PGFFreeProductApplication}
    Suppose $G$ is a torsion free PGF group relative to $\mathcal{H}$. Then there exists a sequence of reducible elements $f_i\in \text{MCG}(\Sigma)$ with $f_1=\text{id}$ for $i=1,\ldots n$ so that the natural map
    $$f_1Gf_1^{-1} * \cdots * f_nGf_n^{-1}\to \text{MCG}(\Sigma)$$
    is injective with PGF image, relative to the union of the conjugates of the elements of $\mathcal{H}$ by $f_i$ for $i=1,\ldots, n$.
\end{corollary}
\begin{proof}
    We take a nonseparating curve $\alpha_2$ which intersects every component of every multicurve associated to $G$. Such a curve exists by Corollary \ref{PGFNotQuasiDense}. In particular, for any nontrivial element $g\in G$, $\pi_{\Sigma\setminus \alpha_2}(g(\alpha_2))\neq \varnothing$. This is immediate if $g$ is pseudo-Anosov. Suppose $g$ is peripheral. Then choose an $n$ so that $g^n$ is a multitwist. Further increase $n$ so that some Dehn twist appearing in $g^n$ has the absolute value of its power at least $3$. By Proposition \ref{IvanovIntersectionBound} (with $\alpha=\beta=\alpha_2$), $g^n(\alpha_2)$ must intersect $\alpha_2$, so in particular $g(\alpha_2)\neq \alpha_2$. 
    \par 
    Applying Theorem \ref{CocptFreeProductApplicationAbstract}, we find a reducible element $f_2$ so that the natural map $G*f_2Gf_2^{-1}\to \text{MCG}(\Sigma)$ is injective with PGF image relative to the union of $\mathcal{H}$ and the $f_2$ conjugates of $\mathcal{H}$. 
    \par 
    The result then follows by induction. Here, it is possible that the curve $\alpha_i$ used to define $f_i$ does not intersect every component of every multicurve associated to the group $G*\cdots f_iGf_i^{-1}$. Instead we choose, using Corollary \ref{PGFNotQuasiDense}, a curve $\alpha_{i+1}$ which intersects every component of every multicurve associated to $G*\cdots f_iGf_i^{-1}$, and thus also intersects every multicurve associated to $G$. We then perform the same argument as above in order to apply Theorem \ref{CocptFreeProductApplicationAbstract}.
\end{proof}

\begin{example}\label{PGFFreeProductSurfaceGrp}
    Let $G$ be a PGF surface group as in Corollary \ref{BookOfIBundleCombo}. This group is torsion free, and by a slight extension of Corollary \ref{PGFFreeProductApplication}, we obtain examples of PGF groups which are free products of surface groups where each free factor is isomorphic to any finite index subgroup of $G$. This follows as the use of Corollary \ref{PGFNotQuasiDense} in both the base case and the inductive step in Corollary \ref{PGFFreeProductApplication} can be applied as the collection of multicurves associated to a finite index subgroup of a PGF group is the same as the collection associated to the full group.
    \par 
    Given a torsion free convex cocompact group $F$, one can then further apply a slight extension of Corollary \ref{PGFFreeProductApplication} along with Corollary \ref{PGFNotQuasiDense} to find PGF groups isomorphic to free products of $F$ and groups as in the previous paragraph (this is because $F$ will fix no curves, so the use of Corollary \ref{PGFNotQuasiDense} only needs to be applied to the free product of surface groups).
\end{example}

\begin{example}\label{GeneralPGFRAAGS}
    We can now find examples of PGF groups of the form $G=F*T$ where $F$ is a free convex cocompact group of arbitrary rank, and $T$ is any group as in Example \ref{FreeProductMultitwists}. This completes the discussion after Lemma \ref{GenericPGFRAAGS}.
    \par 
    To do this, one can apply of Theorem \ref{ConvexCocompactCombo} directly, as we may choose a nonseparating curve $\alpha$ intersecting every component of every multicurve associated to $T$ using Corollary \ref{PGFNotQuasiDense}. As no nontrivial element of $F$ will fix $\alpha$, the theorem applies (using Proposition \ref{IvanovIntersectionBound} as in the proof of Corollary \ref{PGFFreeProductApplication} to show that no nontrivial element of $T$ fixes $\alpha$ either).
\end{example}

Next we give the application of Theorem \ref{ConvexCoCptTwists}.

\begin{theorem}\label{CocptTwistsApplicationAbstract}
    Suppose $G$ is a PGF group relative to $\mathcal{H}$ and $H$ is an arbitrary twist group. If there is a curve $\alpha$ so that for all nontrivial $g\in G$, $h\in H$, $g(\alpha)$ and $h(\alpha)$ intersect every component of $\Sigma \setminus \alpha$, then there is a reducible $f\in \text{MCG}(\Sigma)$ so that the natural map $fGf^{-1}*H\to \text{MCG}(\Sigma)$ is injective with PGF image relative to $\mathcal{H}\cup\{H\}$. 
\end{theorem}
\begin{proof}
    The theorem follows by combining Theorem \ref{ConvexCoCptTwists} using the collection $\{\Sigma \setminus \alpha\}$ along with Proposition \ref{PGFProjBound} applied to $G$ as before, and letting $f$ be a partial pseudo-Anosov on $\Sigma\setminus \alpha$ with sufficiently high translation length. 
\end{proof}

\begin{corollary}\label{CocptTwistsApplication}
    Fix a torsion free PGF group $G$ relative to $\mathcal{H}$ and a torsion free twist group $H$. Then there is a reducible $f\in \text{MCG}(\Sigma)$ so that the natural map $fGf^{-1}*H\to \text{MCG}(\Sigma)$ is injective with PGF image relative to the union of $H$ and the $f$ conjugates of elements of $\mathcal{H}$. 
\end{corollary}
\begin{proof}
 Using Corollary \ref{PGFNotQuasiDense} we can choose a nonseparating curve $\alpha$ intersecting every component of every multicurve associated to $G$, as well as every component of the multicurve associated to $H$. Using Proposition \ref{IvanovIntersectionBound} as before we have for all nontrivial $g\in G$ and $h\in H$ that $g(\alpha)$ and $h(\alpha)$ have nonempty projection to $\Sigma \setminus \alpha$. The result then follows by Theorem \ref{CocptTwistsApplicationAbstract}.
\end{proof}

Using this we can obtain more general free products of PGF groups. The downside of this as compared to Corollary \ref{PGFFreeProductApplication} is that we have less control over the conjugating elements.
\begin{corollary}
    Fix torsion free PGF groups $G_1,\ldots, G_n$, relative to $\mathcal{H}_1, \ldots \mathcal{H}_n$, respectively. Then there are elements $f_1,\cdots, f_n\in \text{MCG}(\Sigma)$ (not necessarily reducible) so that $f_1G_1f_1^{-1}*\cdots*f_nG_nf_n^{-1}$ is PGF relative to the union of the $f_i$ conjugates of elements of $\mathcal{H}_i$, for $i=1,\ldots, n$.
\end{corollary}
\begin{proof}
    Fix a curve $\gamma$, and use Corollary \ref{CocptTwistsApplication} to find PGF groups of the form $f_i'G_if_i'^{-1}*\<\tau_{\gamma}\>$ for some reducible $f_i'\in \text{MCG}(\Sigma)$. We may then apply Theorem \ref{ComboApplication} using $\<\tau_{\gamma}\>$ as the common twist subgroup to give a PGF group of the form $f_1G_1f_1^{-1}*\cdots*f_nG_nf_n^{-1}*\<\tau_{\gamma}\>$ where the elements $f_i$ no longer are guaranteed to be reducible as the PGF factors are further conjugated by a partial pseudo-Anosov in $\Sigma\setminus\gamma$. The group obtained by dropping the $\<\tau_{\gamma}\>$ factor is still PGF, completing the proof.
\end{proof}

\par 
Finally we give applications of Theorem \ref{LoaAnalog}.
\begin{theorem} \label{ComboTwistGroups}
    Fix two twist groups $H_1$ and $H_2$. If there is a curve $\alpha$ so that for all nontrivial $h_i\in H_i$, $h_i(\alpha)$ intersects every component of $\Sigma \setminus \alpha$ for $i=1, 2$, then there is a reducible $f\in \text{MCG}(\Sigma)$ so that the natural map $fH_1f^{-1}*H_2\to \text{MCG}(\Sigma)$ is injective with PGF image relative to $\{fH_1f^{-1}, H_2\}$. 
\end{theorem}
\begin{proof}
    This is just applying Theorem \ref{LoaAnalog} to the collection $\{\Sigma \setminus \alpha\}$ with $f$ a partial pseudo-Anosov on $\Sigma\setminus \alpha$ with large translation length. 
\end{proof}

Again we obtain the following concrete application. Its proof follows in the same way as Corollary \ref{CocptTwistsApplication}.

\begin{corollary} \label{ComboMultitwistGroups}
    Fix two torsion free twist groups $H_1$ and $H_2$. Then there is a reducible $f\in \text{MCG}(\Sigma)$ so that the natural map $fH_1f^{-1}*H_2\to \text{MCG}(\Sigma)$ is injective with PGF image relative to $\{fH_1f^{-1}, H_2\}$.
\end{corollary}
\begin{proof}
    We can take any curve $\alpha$ which intersects all the components of the multicurves associated to $H_1$ and $H_2$. By applying Proposition \ref{IvanovIntersectionBound} as before, we can apply Theorem \ref{LoaAnalog} which gives the result.
\end{proof}

\begin{example}\label{PGFFreePrdocutsTwistGroupExamples}
    In the notation of Corollary \ref{ComboMultitwistGroups}, letting $H=H_1=H_2$ and choosing the curve $\alpha$ for the proof of Corollary \ref{ComboMultitwistGroups} to be exactly $2$ from some component of the multicurve $A$ of $H$ and at least $2$ from every other, we obtain PGF groups which are free products of multitwist groups whose multicurves have components that are at most $4$ from each other. This is much smaller than the $D_0$ given in Proposition \ref{LoaCombo}. Using Corollary \ref{CocptTwistsApplication}, one can then produce PGF free groups whose generators are Dehn twists so that every pair of twisting curves of these generators are distance at most $4$ from each other. There are examples where the curves are all distance $3$ from each other as well. Consider a genus $2$ surface, and let $\alpha$ in the proof of Corollary \ref{ComboMultitwistGroups} be a separating curve. Letting $H$ be a cyclic group generated by the Dehn twist on a curve $\gamma$ intersecting $\alpha$ twice, it is easy to see that, for a partial pseudo-Anosov $f$ fixing $\alpha$, the collection of curves $\{f^n(\gamma)\}_{n\in \Z}$ are all distance $3$ from each other.

\end{example}
We end with a remark about applying the combination theorems inside subgroups of $\text{MCG}(\Sigma)$. Fix $N<\text{MCG}(\Sigma)$ a normal subgroup and a collection of PGF groups and twist groups which we wish combine (in the sense of Theorems \ref{ComboQIEmb}, \ref{ConvexCocompactCombo}, \ref{ConvexCoCptTwists}, and \ref{LoaAnalog}) which are contained in $N$. Then the group produced by whichever theorems we use will still be contained in $N$. This is because in all the applications of these theorems, each group is only modified by a conjugation. For example, this gives many examples of PGF groups in the Torelli group. One can also extend this idea to subgroups which are not normal, but instead have a sufficiently large normalizer so that one can choose conjugating partial pseudo-Anosovs which preserve the subgroup. This allows for many examples of PGF groups in the handlebody group, for example.

\end{document}